\documentclass[twoside]{amsart}

\numberwithin{equation}{section}
\usepackage{amsmath,amssymb,amsfonts,amsthm,latexsym}
\usepackage{graphicx,color,enumerate}
\usepackage[margin=1in]{geometry}
\usepackage{longtable}
\usepackage{appendix}
\usepackage{ stmaryrd }
\usepackage{mleftright}
\usepackage{hyperref}

\newtheorem{theorem}{Theorem}[section]
\newtheorem{lemma}[theorem]{Lemma}
\newtheorem{proposition}[theorem]{Proposition}
\newtheorem{corollary}[theorem]{Corollary}

\theoremstyle{definition}
\newtheorem{definition}[theorem]{Definition}

\theoremstyle{remark}

\newcommand{\C}{\mathbb{C}}

\newcommand{\R}{\mathbb{R}}

\newcommand{\Z}{\mathbb{Z}}

\newcommand{\git}{/\!\!/}
\newcommand{\bs}{\boldsymbol}

\newcommand{\dR}{\mathrm{d}}
\newcommand{\dP}{\delta_{\operatorname{Poiss}}}
\newcommand{\Spec}{\operatorname{Spec}}
\newcommand{\I}{\operatorname{\sqrt{-1}}}

\newcommand{\fd}{\operatorname{fd}}

\newcommand{\Hom}{\operatorname{Hom}}

\newcommand{\UnSh}{\operatorname{UnSh}}
\newcommand{\Sym}{\operatorname{S}}

\begin{document}

\title{Higher Koszul brackets on the cotangent complex}

\author[H.-C.~Herbig]{Hans-Christian Herbig}
\address{Departamento de Matem\'{a}tica Aplicada,
Av. Athos da Silveira Ramos 149, Centro de Tecnologia - Bloco C, CEP: 21941-909 - Rio de Janeiro, Brazil}
\email{herbighc@gmail.com}

\author[D.~Herden]{Daniel Herden}
\address{Department of Mathematics, Baylor University,
Sid Richardson Building,
1410 S.4th Street,,
Waco, TX 76706, USA}
\email{Daniel\_Herden@baylor.edu}

\author[C.~Seaton]{Christopher Seaton}
\address{Department of Mathematics and Computer Science,
Rhodes College, 2000 N. Parkway, Memphis, TN 38112}
\email{seatonc@rhodes.edu}

\thanks{C.S. was supported by the E.C.~Ellett Professorship in Mathematics;
H.-C.H. was supported by CNPq through the \emph{Plataforma Integrada Carlos Chagas.}}

\keywords{Poisson algebras, cotangent complex, $L_\infty$-algebroids, Poisson cohomology, invariant theory}
\subjclass[2010]{primary 17B63,	secondary 13D02, 58A50, 17B66}

\begin{abstract}
Let $n\ge 1$ and $A$ be a commutative algebra of the form $\bs k[x_1,x_2,\dots, x_n]/I$ where $\bs k$ is a field of characteristic $0$ and $I\subseteq \bs  k[x_1,x_2,\dots, x_n]$ is an ideal.
Assume that there is a Poisson bracket $\{\:,\:\}$ on $S$ such that $\{I,S\}\subseteq I$ and let us denote the induced bracket on $A$ by $\{\:,\:\}$ as well. It is well-known that $[\dR x_i,\dR x_j]:=\dR\{x_i,x_j\}$ defines a Lie bracket on the $A$-module $\Omega_{A|\bs k}$ of K\"ahler differentials making  $(A,\Omega_{A|\bs k})$ a Lie-Rinehart pair. Recall that $A$ is regular if and only if $\Omega_{A|\bs k}$ is projective as an $A$-module. If $A$ is not regular, the cotangent complex $\mathbb L_{A|\bs k}$  may serve as a replacement for the $A$-module $\Omega_{A|\bs k}$. We prove that there is a structure of an $L_\infty$-algebroid on $\mathbb L_{A|\bs k}$, compatible with the Lie-Rinehart pair $(A,\Omega_{A|\bs k})$. The $L_\infty$-algebroid on $\mathbb L_{A|\bs k}$ actually comes from a $P_\infty$-algebra structure on the resolvent of the morphism $\bs k[x_1,x_2,\dots, x_n]\to A$. We identify examples when this $L_\infty$-algebroid simplifies to a dg Lie algebroid.
 For aesthetic reasons we concentrate on cases when  $ \bs k[x_1,x_2,\dots, x_n]$ carries a (possibly nonstandard) $\mathbb Z_{\ge 0}$-grading and both $I$ and $\{\:,\:\}$ are homogeneous.
\end{abstract}

\maketitle
\tableofcontents



\section{Introduction}\label{sec:intro}

Let $\bs k$ be a field of characteristic $0$. A unital commutative associative $\bs k$-algebra $A$ is called a \emph{Poisson algebra} if it is endowed with a Lie bracket $\{\:,\:\}:A\times A\to A$ such that  $\{a,bc\}=b\{a,c\}+\{a,b\}c$ for all $a,b,c\in A$. The bracket $\{\:,\:\}$ is referred to as the \emph{Poisson bracket}. If $S$ is a Poisson algebra, an ideal $I\subseteq S$ with respect to the multiplication is called a \emph{Poisson ideal} if $\{I,S\}\subseteq I$. If $I$ is a Poisson ideal, the Poisson bracket descends to the quotient algebra $S/I$. In this paper we study Poisson algebras of the form $A=S/I$ where $I$ is a finitely generated Poisson ideal in $S$. Throughout the paper we focus on the case when $S$ is a polynomial algebra $\bs k[x_1,x_2,\dots,x_n]$.  However, many results are also valid for other Poisson algebras, e.g.,
algebras of regular functions on a regular Poisson varieties or algebras of smooth functions on Poisson manifolds. For aesthetic reasons, we mainly investigate examples where there is a nonstandard $\Z_{\ge 0}$-grading on $S$.
This means that to each of the variables $x_i$ is attached an \emph{internal degree} $\deg(x_i)\ge 1$ such that the ideal $I$ is generated by elements that are homogeneous with respect to the internal degree. The Poisson bracket on $S$ is uniquely determined by the brackets between the coordinates, denoted
\begin{align}\label{eq:Lambda}
\{x_i,x_j\}=:\Lambda_{ij}\in S.
\end{align}
We use this notation throughout the paper. We typically assume that the bracket respects the internal degree so that $\deg(\Lambda_{ij})=\deg(\{\:,\:\})+\deg(x_i)+\deg(x_j)$.

An element $f\in S$ is called a \emph{Casimir} if $\{f,\:\}$ acts trivially on $S$. The set $\mathrm H^0_{\mathrm{Poiss}}(S)$ of Casimirs in $S$ is called the \emph{Poisson center} of $S$. The reason for the notation is that it can be identified with the zeroth Poisson cohomology of $S$ (cf. Appendix \ref{ap:Poissoncohomology}).
We say that the ideal $I\subseteq S$ is \emph{generated by Casimirs}  if there exist generators $f_1,\dots, f_k$ for $I$ that are Casimirs. More generally,  an ideal $I$ generated by $f_1,\dots, f_k\in S$ is Poisson if and only if there exist $Z_{i\mu}^\nu\in S$ with $i\in\{1,2,\dots,n\}$ and $\mu, \nu\in \{1,2,\dots,k\}$ such that
\begin{align}\label{eq:theZs}
\{x_i,f_\mu\}=\sum_\nu Z_{i\mu}^\nu f_\nu.
\end{align}
The constructions in this paper assume a fixed choice of such a tensor $Z_{i\mu}^\nu\in S$. If $f_1,\dots, f_k\in S$ form a complete intersection, then $Z_{i\mu}^\nu$ is unique up to $I$ (cf. Section \ref{sec:ci}). If the variables have internal degree $\deg(x_i)\ge 1$, the $f_\mu$'s are homogeneous, and the bracket respects the internal degree, then the $Z_{i\mu}^\nu$'s  should be chosen such that
$\deg(Z_{i\mu}^\nu)=\deg(x_i)+\deg(f_\mu)+\deg(\{\:,\:\})-\deg(f_\nu)$.

If $A$ is a Poisson algebra over $\bs k$, then there is a Lie bracket, the so-called \emph{Koszul bracket} (cf. \cite{Huebschmann}), on the module of K\"ahler differentials $\Omega_{A|\bs k}$ given by the formula
\begin{align} \label{eq:Koszulbr}
[a_1\dR a_2,b_1\dR b_2]:=a_1\{a_2,b_1\}\dR b_2+b_1\{a_1,b_2\}\dR a_2+a_1b_1\dR\{a_2,b_2\}
\end{align}
for $a_1,a_2,b_1,b_2\in A$. If $\Spec(A)$ is smooth, then $\Omega_{A|\bs k}$ with this bracket  forms a Lie algebroid over $\Spec(A)$ in the following sense.

\begin{definition}\label{def:liealgebroid} A \emph{Lie algebroid over $\Spec(A)$} is a projective $A$-module $L$ together with a Lie bracket $[\:,\:]$ and an $A$-linear morphism of Lie algebras $\rho:L\to D_A$, $X\mapsto \rho_X$, to the module of derivations $D_A:=\mathrm{Der}(A,A)$ such that  $[X,aY]=a[X,Y]+(\rho_X a)Y$ for $a\in A$ and $X,Y\in L$. The morphism $\rho$ is referred to as the \emph{anchor}.
\end{definition}
If we drop the assumption that $L$ is projective, we say that $(A,L)$ forms a \emph{Lie-Rinehart pair} over $\bs k$ (see \cite{Rinehart, Huebschmann}). If $A$ is a Poisson algebra, then $(A,\Omega_{A|\bs k})$ forms a Lie-Rinehart pair with anchor  $\rho_{a\dR b}(c)=a\{b,c\}$ for $a,b,c\in A$ \cite{Huebschmann}. Whenever $\Spec(A)$ is non-smooth the module of K\"ahler differentials $\Omega_{A|\bs k}$ is \emph{non-projective} (see \cite{AvramovHerzog}), which makes its homological algebra more intricate. The objective of this paper is to lift the Koszul bracket to the cotangent complex $\mathbb L_{A|\bs k}$ in the form of an $L_\infty$-algebroid over $\Spec(A)$ (for details on the cotangent complex see Section \ref{sec:cotangent}).

The principal tool to achieve this is Theorem \ref{thm:homotopyPoisson} below, which provides a $P_\infty$-algebra  structure on a resolvent $R$ of $A$  (for information on resolvents see Section \ref{sec:cotangent}). Before stating  Theorem \ref{thm:homotopyPoisson} let us recall some definitions. An \emph{$L_\infty$-algebra} is a graded vector space $L=\oplus_{k\in \Z}L^k$ over $\bs k$, whose degree is denoted $|\:|$,  with a sequence
$([\:,\dots,\:]_m)_{m\ge 1}$ of $\bs k$-linear operations $[\:,\dots,\:]_m:\bigwedge^m L\to L$ of degree $|l_m|=2-m$ such that for all $m\ge 1$
\begin{align}\label{eqn:Linftyalgebra}
\sum_{p+q=m+1} \sum_{\sigma \in \UnSh_{q, p-1}} (-1)^\sigma \varepsilon(\sigma,\bs x)(-1)^{q(p-1)}  \left[[x_{\sigma(1)},\dots,x_{\sigma(q)}]_q,x_{\sigma(q+1)},\dots,x_{\sigma(m)}\right]_p=0
\end{align}
for homogeneous $x_1,\dots,x_{m}\in L$. Here $\varepsilon(\sigma,\bs x)=(-1)^{\sum_{i<j,\sigma(i)>\sigma(j)}|x_i||x_j|}$ is the \emph{Koszul sign} of the permutation $\sigma$, $(-1)^\sigma$ its sign, and
$\UnSh_{q,p-1}$ stands for the $(q,p-1)$-unshuffle permutations, i.e., the set of permutations $\sigma$ of $\{1,2,\dots,m\}$ such that $ \sigma(1)<\sigma(2)<\dots<\sigma(q)$ and $\sigma(q+1)<\sigma(q+2)<\dots<\sigma(m)$.
Note that $[\:]_1$ is a codifferential. The grading $|\:|$ is referred to as the \emph{cohomological grading}.

By an \emph{$L_\infty$[1]-algebra} structure on a graded vector space $E=\oplus_{k\in \Z}E^k$ over $\bs k$ with degree $|\:|$ we mean a sequence $(l_m)_{m\ge 1}$ of $\bs k$-linear operations $l_m:\Sym^m E\to E$ of degree $|l_m|=1$ such that for all $m\ge 1$
\begin{align}\label{eqn:Linftyalgebra}
\sum_{p+q=m+1} \sum_{\sigma \in \UnSh_{q, p-1}}  \varepsilon(\sigma,\bs e)  l_p(l_q(e_{\sigma(1)},\dots,e_{\sigma(q)}),e_{\sigma(q+1)},\dots,e_{\sigma(m)})=0
\end{align}
for homogeneous $e_1,\dots,e_{m}\in E$. An  \emph{$L_\infty[1]$-algebra} structure on $L[1]$ is equivalent to an \emph{$L_\infty$-algebra} structure on $L$ by putting
\begin{align}\label{eq:decalage}
[x_1,\dots, x_m]_m=(-1)^{\sum_{i=1}^m(m-i)|x_i|} l_m(x_1[1],\dots,  x_m[1])[-1],
\end{align}
where $\downarrow:L\to L[1]$, $x\mapsto \downarrow x=x[1]$ is the identity seen as a map of degree $-1$. A more conceptual way to write this is $[\:,\dots,\:]_m=\uparrow\circ l_m\circ \downarrow^{\otimes n}$ where $\uparrow$ is the inverse of $\downarrow$. For details the reader may consult, e.g., \cite{Reinhold}.


The notion of a left $L_\infty$-module goes back to \cite{LadaMarkl, Lada}.
By a \emph{left $L_\infty$-module} over the $L_\infty$-algebra $L$ with brackets $[\:,\dots,\:]_m:\bigwedge^m L\to L$, $m\ge 1$ we mean a graded vector space $M=\oplus_k M^k$
over $\bs k$ with degree $|\:|$ and a sequence of operations $\rho_m:\bigwedge^{m-1} L \otimes M\to M$, $m\ge 1$, of degree $|\rho_m|=2-m$ such that for each $m\ge 1$
\begin{align}\label{eqn:Linftymodule}
\sum_{p+q=m+1} \sum_{\sigma \in \UnSh_{q, p-1}}  (-1)^\sigma\varepsilon(\sigma,\bs x)(-1)^{q(p-1)}  k_p(k_q(x_{\sigma(1)},\dots,x_{\sigma(q)}),x_{\sigma(q+1)},\dots,x_{\sigma(m)})=0
\end{align}
for homogeneous $x_1,\dots,x_{m-1}\in L$ and $x_m\in M$, where $k_m:\bigwedge^{m}(L \oplus M)\to L \oplus M$, $m\ge 1$, is the unique extension of the operations $[\:,\dots,\:]_m$ and $\rho_m$ such that
$k_m$ vanishes when two or more arguments are from $M$.
The definition entails that $\rho_1^2=0$, so that $(M,\rho_1)$ is actually a cochain complex.

The next definition is closely related to what has been called in \cite{Vitagliano} a \emph{strong homotopy Lie-Rinehart algebra}.

\begin{definition} \label{def:Linftyalgebroid} By an \emph{$L_\infty$-algebroid over $\Spec(A)$} we mean an $L_\infty$-algebra  $L=\oplus_{k\in \Z}L^k$ with brackets $([\:,\dots,\:]_m)_{m\ge 1}$, such that each $L^k$ is an $A$-module, together with operations $\rho_m:\bigwedge^{m-1} L \otimes A\to A$, $m\ge 1$, that make $A$ a left $L_\infty$-module over $L$ satisfying the following properties
\begin{enumerate}
\item $\partial:=[\:]_1$ is $A$-linear and $\rho_1=0$,
\item for each $k\in \Z$, $L^k$ is a finitely generated projective $A$-module,
\item for all $m\ge 2$, homogeneous $x_1,\dots, x_m\in L$, and $a,b\in A$ we have
\begin{align*}
[x_1,\dots, x_{m-1},ax_m]_m&=\rho_m(x_1,\dots, x_{m-1},a)x_m+a\,[x_1,\dots, x_{m-1},x_m]_m,\\
\rho_m(x_1,\dots, x_{m-1},ab)&=\rho_m(x_1,\dots, x_{m-1},a)b+a\,\rho_m(x_1,\dots, x_{m-1},b),\\
\rho_m(ax_1,\dots, x_{m-1},b)&=a\rho_m(x_1,\dots, x_{m-1},b).
\end{align*}
\end{enumerate}
We refer to the collection of operations $(\rho_m)_{m\ge 1}$ as the \emph{homotopy anchor}. In the special case when $[\:,\dots,\:]_m=0$ for all $m\ge 3$ and $\rho_m=0$ for all $m\ge 2$ we say that $L$ is a \emph{dg Lie algebroid} over $\Spec(A)$.
\end{definition}
Our definition of an $L_\infty$-algebroid is more general than the one used in \cite{Strobl}.  When $\rho_m=0$ for $m\ge 3$ we recover their definition after applying \eqref{eq:decalage}.
A more general definition has been suggested in \cite{Kjeseth} under the name homotopy Lie-Rinehart pair. We also need to recall the notion of a $P_\infty$-algebra (see, for example, \cite{VoronovHigherDerived,CF}).

\begin{definition}
A \emph{$P_\infty$-algebra} is a supercommutative $\Z$-graded $\bs k$-algebra $R=\oplus_k R^k$ with degree $|\:|$ that is also an $L_\infty$-algebra with brackets $\{\:,\dots,\:\}_m: \bigwedge^m R\to R$ such that the Leibniz rule
\begin{align*}
\{ab,a_2,\dots, a_m\}_m=a\{b,a_2,\dots, a_m\}_m+(-1)^{|a||b|}b\{a,a_2,\dots, a_m\}_m
\end{align*}
holds for $a,b,a_2,\dots, a_m \in R$ with $a,b$ homogeneous. If all $\{\:,\dots,\:\}_m$ are zero for $m\ge 3$, we say $R$ is a \emph{dg Poisson} algebra.
\end{definition}

Evidently, by symmetry, the Leibniz rule holds for any argument of $\{\:,\dots,\:\}_m$. To construct the $P_\infty$-algebra of our main theorem we use the \emph{higher derived brackets} of T. Voronov \cite{VoronovHigherDerived}.  The notions of resolvent and cotangent complex will be reviewed in Section \ref{sec:cotangent}.

\begin{theorem} \label{thm:homotopyPoisson} Let $I\subset S=\bs k[x_1,x_2,\dots,x_n]$  be a Poisson ideal, let $A=S/I$, and let $f_1,\dots, f_k$ be generators for $I$. Let $R$ be a resolvent of $S\to A$ on the generators  $f_1,\dots, f_k$.
Then there is the structure of a $P_\infty$-algebra $(\{\:,\dots,\:\}_m)_m$ on $(R,\partial)$ such that $\partial=\{\:\}_1$ and the quasi-isomorphism $R\to A$ is compatible with the brackets.  If the internal degree of the Poisson bracket on $S$ is $p$, the internal degree of the $n$-ary bracket on $R$ is $(n-1)p$. If the generators $f_1,\dots, f_k$ are Casimir and form a complete intersection, the $P_\infty$-algebra structure is trivial in the sense that the only nonzero Poisson brackets are $\partial=\{\:,\:\}_1$ and the bracket $\{\:,\:\}=\{\:,\:\}_2$ on $S$.
\end{theorem}

In the case of complete intersections a version of the theorem has been suggested in \cite{FresseCI}. In order to get a clear picture of the construction, we  felt it necessary to delve into details and elaborate examples.  The higher Koszul brackets are deduced from the $P_\infty$-algebra as a corollary.

\begin{corollary}\label{cor:homotopyLiealgebroid} Under the assumptions of Theorem \ref{thm:homotopyPoisson} there is the structure of an $L_\infty$-algebroid on the cotangent complex $\mathbb L_{A|\bs k}=A\otimes_R \Omega_{R|\bs k}$ such that the morphism $\mathbb L_{A|\bs k}\to\Omega_{A|\bs k}$ is compatible with the brackets and the anchors.
If the generators $f_1,\dots, f_k$ are Casimir and form a complete intersection, the $L_\infty$-algebroid structure is trivial in the sense that the only nonzero Lie brackets are given by $\left[\dR x_i,\dR x_j\right]=\dR\{x_i,x_j\}$. If the internal degree of the Poisson bracket on $S$ is $p$, the internal degree of the $n$-ary bracket and  $n$-ary anchor on $\mathbb L_{A|\bs k}$ is $(n-1)p$.
\end{corollary}

In this paper two tensors play a major role. Both depend on the choice of $Z_{i\mu}^\nu$. The first tensor is given by
\begin{align}\label{eq:Amunu}
\mathcal A_{\mu \nu}^\lambda :=\sum_{i=1}^n\left( {\partial f_\mu\over \partial x_i}Z_{i\nu}^\lambda+{\partial f_\nu\over \partial x_i}Z_{i\mu}^\lambda \right)\in S
\end{align}
for indices $\mu,\nu,\lambda\in\{1,\dots,k\}$. Evidently $\mathcal A_{\mu \nu}^\lambda=\mathcal A_{\nu \mu}^\lambda$.
The second tensor is constructed from Poisson cohomology (cf. Appendix \ref{ap:Poissoncohomology}) as follows.
Fixing the index $i$ we can view $Z_{i\mu}^\nu$ as the entry in row $\mu$ and column $\nu$ of a $k\times k$-matrix $Z_i$ with coefficients in $S$. For the set of $k\times k$-matrices with entries in $S$  we write $\mathfrak{gl}_k(S)$. Using the commutator of matrices $\mathfrak{gl}_k(S)$ forms a Lie algebra over $S$. Tensoring this Lie algebra with a supercommutative dg algebra we can form a dg Lie algebra, for example $\mathrm{C}^\bullet_{\operatorname{Poiss}}(S)\otimes_S\mathfrak{gl}_k(S)$. For the definition of the complex  $(\mathrm{C}^\bullet_{\operatorname{Poiss}}(S),\dP)$ see Appendix \ref{ap:Poissoncohomology}.
Writing $Z:=\sum_i \partial/\partial x_i\otimes Z_i\in \mathrm{C}^1_{\operatorname{Poiss}}(S)\otimes_S\mathfrak{gl}_k(S)$, the tensor relevant for our discussion is
\begin{align}\label{eq:MC}
\dP Z-[Z,Z]\in\mathrm{C}^2_{\operatorname{Poiss}}(S)\otimes_S\mathfrak{gl}_k(S).
\end{align}
In a similar way $\mathrm C^\bullet_{\operatorname{Poiss}}(S)\otimes_S\mathfrak{gl}_k(A)$ is a dg Lie algebra as well
and, taking classes modulo $I$, we have a surjective morphism $\mathrm C^\bullet_{\operatorname{Poiss}}(S)\otimes_S\mathfrak{gl}_k(S)\to \mathrm C^\bullet_{\operatorname{Poiss}}(S)\otimes_S\mathfrak{gl}_k(A)$.
It is an easy exercise to check that the image of $\dP Z-[Z,Z]$ in $\mathrm C^\bullet_{\operatorname{Poiss}}(S)\otimes_S\mathfrak{gl}_k(A)$ does merely depend on the classes $Z_{i\mu}^\nu+I$.

If the aforementioned tensors vanish certain simplifications occur at the start of the iterative procedure constructing the $P_\infty$-algebra structure of Theorem \ref{thm:homotopyPoisson}. If all of the tensors vanish and $f_1,\dots, f_k$ is a complete intersection our $P_\infty$-algebra simplifies to a dg Poisson algebra. This happens for example for principal Poisson ideals, see Theorem \ref{thm:principaldgPoisson}. In order to get a better impression which simplifications can occur we present numerous examples. The main examples of Poisson ideals we consider come from symplectic reduction or from linear Poisson structures. We also include some examples of Poisson brackets of higher degree. These enable us to concoct more academic examples where the resolvent is more tractable in low cohomological degrees. To our knowledge the technology of how to construct Poisson ideals for polynomial Poisson structures is not completely developed, and we feel that the problem deserves more attention. As the example calculations get quickly overwhelming we relied on computer calculations using \emph{Mathematica} and \emph{Macaulay2}. We mention that all examples of quadratic Poisson ideals $I$ that occur in this work give rise to Koszul algebras $S/I$ (for an introduction to Koszul algebras see, e.g., \cite{polishchuk2005quadratic}). In the Koszul case simplifications in the $P_\infty$-algebra structure of Theorem \ref{thm:homotopyPoisson} arise if the degree of the Poisson bracket is $\le 0$, see Theorem \ref{thm:Koszul}. The adaptation of our results to the situation when $S=\bs k[x_1,\dots,x_n]$ is replaced by the non-Noetherian $\R$-algebra $\mathcal C^\infty(\R^k)$ of smooth functions on $\R^k$ is possible and will be elaborated at another occasion. We are also working on an application to the problem of deformation quantization of singular Poisson varieties.
We are planning to address applications to deformation theory of Poisson singularities and investigate repercussions for Poisson cohomology.

Our paper is structured as follows. In Section \ref{sec:affine} we present examples of Poisson ideals. Here we focus on symplectic quotients, conical varieties stable under coadjoint actions and certain polynomial Poisson structures (diagonal and determinantal brackets).
In Section \ref{sec:ci} we present a version of Corollary \ref{cor:homotopyLiealgebroid} for the case of complete intersections with a proof independent of Theorem \ref{thm:homotopyPoisson}. In Section \ref{sec:connection} we interpret the $Z_{i\mu}^\nu$ as the Christoffel symbols of a Poisson connection; this material is not used in later sections. In Section \ref{sec:cotangent} we recall the notions of resolvent and cotangent complex and fix notations for later use. In Section \ref{sec:homotopystuff} we develop our homological perturbation theory argument that enables us to prove our main results Theorem \ref{thm:homotopyPoisson} and Corollary \ref{cor:homotopyLiealgebroid}. In Section \ref{sec:examples} we provide sample calculations that were mostly obtained using a combination of \emph{Mathematica} and the \emph{Macaulay2} package \emph{dgalgebras}.

The paper is mainly addressed to two audiences: 1) people from commutative algebra and 2) people from Poisson geometry and physics. We have resisted the temptation to write coordinates with upper indices to not irritate the former. However, we often use index notation for tensors and Einstein summation convention to make the presentation of the material clearer to the latter. Our article touches on three seemingly unrelated subjects that are named after J.-L. Koszul (1921--2018): the Koszul bracket, the Koszul complex, and Koszul algebras.

\vspace{1cm}
\noindent\emph{Acknowledgements.} The authors would like to thank Daniel Levcovitz for sharing information on the cotangent complex. They are indebted to Srikanth Iyengar and Benjamin Briggs for their interest in this project and, in particular, for catching an error in an earlier draft. HCH would like to thank Martin Bordemann for indoctrination.
We benefited from the \emph{Mathematica} packages DiracQ \newline
\url{http://physics.ucsc.edu/~sriram/DiracQ/} \newline
and the \emph{grassmann.m} \emph{Mathematica} package by Matthew Headrick \newline
\url{https://people.brandeis.edu/~headrick/Mathematica/}\newline in learning how to perform the computations needed for this paper with \emph{Mathematica}.

\section{Examples of affine Poisson algebras}\label{sec:affine}

\subsection{Symplectic quotients of cotangent lifted $G$-modules} Let $G$ be a complex reductive Lie group and $V$ be an $m$-dimensional $G$-module. Let us write $q_1,q_2,\dots, q_m$ for linear coordinates on $V$ and $p_1,p_2,\dots, p_m$ for linear coordinates on $V^*$.
We assign to them the internal degree $1$. Then $\C[V\times V^*]$ is a Poisson algebra with bracket $\{q_i,p_j\}=\delta_{ij}$ of degree $-2$. The moment map for the representation is the map $J:V\times V^*\to \mathfrak g^*$
defined by $J(q,p)(\xi)=p(\xi q)$ where $p\in V^*$, $\xi\in \mathfrak g$ and $q\in V$.
Using the linear coordinates this is $\sum_{i,j}p_j\xi_{ij}q_i$, where $\xi_{ij}$ is the representation matrix of $\xi$ in the  representation $V=\C^m$. Let $N\subseteq V\times V^*$ be locus of all the polynomials $J_\xi(q,p):=J(q,p)(\xi)$ where $\xi$ ranges over $\mathfrak g$. For simplicity we assume that the ideal $(J_\xi(q,p)\mid \xi\in\mathfrak g)$ in $\C[V\times V^*]$ is radical (for details see \cite{HSSCompositio}). The \emph{symplectic quotient} of the $G$-action on $V$ is the categorical  quotient $N/\!\!/G$, i.e., $\Spec (\C[V\times V^*]^G/(\C[V\times V^*]^G\cap (J_\xi(q,p)\mid \xi\in\mathfrak g)))$. It inherits a Poisson bracket of degree $-2$.

If $G\subset\operatorname{SL}_2(\C)$ is a finite subgroup let us write $q,p$ for the linear coordinates of $\C^2$. Since the $G$-action is unimodular,
it preserves the Poisson bracket defined by $\{q,p\}=1$. We set $\deg(q)=\deg(p)=1$
so that $\C[q,p]$ is a Poisson algebra whose bracket has internal degree $-2$. Since the moment map $J$ is zero here, the symplectic quotient $N/\!\!/G$ is simply $\C^2/G$. Using polynomial $G$-invariants, the latter is determined as a hypersurface with algebra of functions $A=\C[x_1,x_2,x_3]/(f_G)$. Recall that those finite subgroups are classified by Dynkin diagrams $A_m$ $(m\ge 1)$, $D_m$ ($m\ge 0$), and the three exceptional $E_6,E_7$ and $E_8$. Below we will analyze $A=\C[x_1,x_2,x_3]/(f_G)$ as Poisson algebras. The internal degrees of the variables $x_i$ are determined by the degrees of the corresponding invariants. Our convention is that it is weakly increasing with the index of the variable. For the polynomial invariant theory of these groups see \cite{klein1993felix,Dolgachev,ONAsh}.
It is closely related to the \emph{Grundformen} of F. Klein. The Poisson bracket on the algebra of invariants has been addressed before (see, e.g., \cite{Alev}), but we could not find in the literature the observation that the Poisson ideals are generated by Casimirs. Low dimensional symplectic quotients are often symplectomorphic to (cartesian products of) Kleinian singularities, see for example \cite{FHSSigma,HSSadv,CHS}.

We also include two examples of symplectic quotients by nonfinite $G$ that are not symplectomorphic to orbifolds. We have used \emph{Macaulay2} to present the algebra of the symplectic quotient in terms of generators and relations. We emphasize that for the majority of representations it is practically hopeless to find such a presentation.

\subsubsection{Kleinian singularity $A_m$}
In the case of the diagram  $A_m$ the group $G$ is the cyclic group $\mathbb \Z_{N}$ with $N=m+1$ where  $\zeta=\exp(2\pi\I/N)$ acts by    $q\mapsto \zeta q$ and $p\mapsto \zeta^{-1}p$. A complete set of polynomial invariants is given by $\varphi_1(q,p)=qp$, $\varphi_2(q,p)=q^N$, and $ \varphi_3(q,p)=p^N$ and satisfies the relation $\varphi_2\varphi_3=\varphi_1^N$. The brackets can be easily evaluated:
$\{\varphi_1,\varphi_2\}=-N \varphi_2$, $\{\varphi_1,\varphi_3\}=N \varphi_3$ and $\{\varphi_2,\varphi_3\}=N^2\varphi_1^{N-1}$. We conclude that the Poisson algebra for the case of $A_m$ is  $\C[x_1,x_2,x_3]/(x_2x_3-x_1^N)$ with bracket table:
\begin{align*}
\begin{tabular}{c||c|c|c}
  $\{\:,\:\}_{A_m}$   & $x_1$&$x_2$&$x_3$\\\hline\hline
$x_1$ &   0  &$-Nx_2$&$Nx_3$\\\hline
$x_2$ &      &$0$    & $N^2x_1^{N-1}$\\\hline
$x_3$ &      &      & $0$
\end{tabular}
\end{align*}
Note that $\deg( \{x_1,f_{A_m}\})-\deg(f_{A_m})=\deg(x_1)-2$, hence $Z_{11}^1=0$. Let us point out  that,  by the same reasoning, in the case of a hypersurface with $\deg(\{\:,\:\})<0$ we have $Z_{11}^1=0$, as by our convention $x_1$ is a variable of lowest internal degree. For $i=2$ or $3$, $\deg( \{x_i,f_{A_m}\})-\deg(f_{A_m})=N-2$.
If $N$ is odd or equal to $2$ this implies $Z_{21}^1=0=Z_{31}^1$
since $\deg(x_1)$ is even. Otherwise, we do not know a better way than verifying by hand $\{x_2,x_2x_3-x_1^N\}=x_2\{x_2,x_3\}-Nx_1^{N-1}\{x_2,x_1\}=N^2x_1^{N-1}x_2-N^2x_1^{N-1}x_2=0$, and similarly for $x_3$. We conclude that
$f_{A_m}(x_1,x_2,x_3)=x_2x_3-x_1^N$ is a Casimir generator.

\subsubsection{Kleinian singularity $D_m$}
In the case of the diagram  $D_m$ the group $G$ is the binary dihedral group $\operatorname{BD_{4N}}$ of order $4N$ with $N=m+2$. The action is generated by $(q,p)\mapsto (\zeta q,\zeta^{-1}p)$, where  $\zeta=\exp(\pi\I/N)$, and $(q,p)\mapsto (p,-q)$.  A complete set of polynomial invariants is given by $\varphi_1(q,p)=q^2p^2$, $\varphi_2(q,p)=q^{2N}+p^{2N}$, and $ \varphi_3(q,p)=qp(q^{2N}-p^{2N})$ and satisfies the relation $\varphi_1\varphi_2^2-\varphi_3^2=4\varphi_1^{N+1}$. The Poisson relations are $\{\varphi_1,\varphi_2\}=-4N\varphi_3=-2\deg(\varphi_2)\varphi_3$,  $\{\varphi_1,\varphi_3\}=-4N\varphi_1\varphi_2=-2\deg(\varphi_2)\varphi_1\varphi_2$,
and
\begin{align*}
\{\varphi_2,\varphi_3\}&=\{q^{2N}+p^{2N},qp(q^{2N}-p^{2N})\}=2N(q^{2N}-p^{2N})^2-2qp\{q^{2N},p^{2N}\}\\
&=2N\varphi_2^2-8N\varphi_1^N-8N^2\varphi_1^N=2N\varphi_2^2-8N(N+1)\varphi_1^N=\deg(\varphi_2)(\varphi_2^2-2\deg(\varphi_3)\varphi_1^N).
\end{align*}
We conclude that the Poisson algebra for the case of $D_m$ is  $\C[x_1,x_2,x_3]/( x_1x_2^2-x_3^2-4x_1^{N+1})$ with bracket table:
\begin{align*}
\begin{tabular}{c||c|c|c}
  $\{\:,\:\}_{D_m}$   & $x_1$&$x_2$&$x_3$\\\hline\hline
$x_1$ &   0  &$-2(\deg x_2)x_3$&$-2(\deg x_2)x_1x_2$\\\hline
$x_2$ &      &$0$    & $(\deg x_2)(x_2^2-2(\deg x_3)x_1^N)$\\\hline
$x_3$ &      &      & $0$
\end{tabular}
\end{align*}
We invite the reader to verify from the bracket table that  $f_{D_m}(x_1,x_2,x_3)=x_1x_2^2-x_3^2-4x_1^{N+1}$ is a Casimir generator, as degree considerations appear to be inconclusive.

\subsubsection{Kleinian singularity $E_6$}
In the case of the diagram  $E_6$ the group $G$ is the binary tetrahedral group $\operatorname{BT}$ of order $24$. It is generated by the matrices
\begin{align*}
\left(
\begin{matrix}
\I&0\\
0&-\I
\end{matrix}
\right), \quad
\left(
\begin{matrix}
0&-1\\
1&0
\end{matrix}
\right), \quad\mbox{and }\quad
{1\over 2}\left(
\begin{matrix}
1+\I&-1+\I\\
1+\I&1-\I
\end{matrix}
\right).
\end{align*}
A complete set of  polynomial invariants is given by
\begin{align*}
&\varphi_1(q,p)=q^{5}p-q p^5,\\
&\varphi_2(q,p)=q^{8}+14 q^{4}p^4+p^8,\\
&\varphi_3(q,p)=q^{12}-33 (q^{8}p^4+q^{4}p^{8})+p^{12}
\end{align*}
and satisfies the relation $\varphi_3^2-\varphi_2^3 =-108\varphi_1^4$. We invite the reader to check the commutation relations
$\{\varphi_1,\varphi_2\} =-8 \varphi_3=-(\deg \varphi_2)\varphi_3$,  $\{\varphi_1,\varphi_3\} =-12 \varphi_2^2=-(\deg \varphi_3)\varphi_2^2$ and
$\{\varphi_2,\varphi_3\} =-1728 \varphi_1^3=-(\deg \varphi_3)^3\varphi_1^3$.
Hence for the binary tetrahedral group the Poisson algebra is $\C[x_1,x_2,x_3]/(x_3^2-x_2^3 +108x_1^4)$ with bracket table:
\begin{align*}
\begin{tabular}{c||c|c|c}
  $\{\:,\:\}_{E_6}$   & $x_1$&$x_2$&$x_3$\\\hline\hline
$x_1$ &   0  &$-(\deg x_2)x_3$&$-(\deg x_3)x_2^2$\\\hline
$x_2$ &      &$0$    & $-(\deg x_3)^3x_1^3$\\\hline
$x_3$ &      &      & $0$
\end{tabular}
\end{align*}
Obviously $Z_{11}^1=0$. Also $\deg( \{x_3,f_{E_6}\})-\deg(f_{E_6})=\deg(x_3)-2=10$ is not in the $\Z_{\ge 0}$-span of $\deg(x_1)=6$ and $\deg(x_2)=8$, and hence $Z_{31}^1=0$. We leave is to the reader to verify from the bracket table that $Z_{21}^1=0$, and conclude that $f_{E_6}(x_1,x_2,x_3)=x_3^2-x_2^3 +108x_1^4$ is a Casimir generator.

\subsubsection{Kleinian singularity $E_7$}
In the case of the diagram  $E_7$ the group $G$ is the binary octahedral group $\operatorname{BO}$ of order $48$. It is generated by the matrices
\begin{align*}{1\over \sqrt{2}}
\left(
\begin{matrix}
1+\I&0\\
0&1-\I
\end{matrix}
\right), \quad
\left(
\begin{matrix}
0&-1\\
1&0
\end{matrix}
\right), \quad\mbox{and }\quad
{1\over 2}\left(
\begin{matrix}
1+\I&-1+\I\\
1+\I&1-\I
\end{matrix}
\right).
\end{align*}
A complete set of polynomial invariants is given by
\begin{align*}
&\varphi_1(q,p)=q^{8}+14q^4 p^4+p^{8},\\
&\varphi_2(q,p)=q^{10}p^{2}-2 q^{6}p^6+q^{2}p^{10},\\
&\varphi_3(q,p)=q^{17}p-34 (q^{13}p^5-q^{5}p^{13})-qp^{17}
\end{align*}
and satisfies the relation $\varphi_1^3\varphi_2-\varphi_3^2=108\varphi_2^3$. We invite the reader to check the commutation relations
$\{\varphi_1,\varphi_2\} =16 \varphi_3=2(\deg \varphi_1)\varphi_3$,  $\{\varphi_1,\varphi_3\} =8( \varphi_1^3-324\varphi_2^2)=\deg \varphi_1( \varphi_1^3-(\deg\varphi_3)^2\varphi_2^2)$, and
$\{\varphi_2,\varphi_3\} =-24 \varphi_1^2 \varphi_2=-2(\deg \varphi_2) \varphi_1^2 \varphi_2$.
Hence for the binary octahedral group the Poisson algebra is $\C[x_1,x_2,x_3]/(x_1^3x_2-x_3^2-108x_2^3)$ with bracket table:
\begin{align*}
\begin{tabular}{c||c|c|c}
  $\{\:,\:\}_{E_7}$   & $x_1$&$x_2$&$x_3$\\\hline\hline
$x_1$ &   0  &$2(\deg x_1)x_3$&$\deg x_1( x_1^3-(\deg x_3)^2x_2^2)$\\\hline
$x_2$ &      &$0$    & $-2(\deg x_2) x_1^2x_2$\\\hline
$x_3$ &      &      & $0$
\end{tabular}
\end{align*}
Obviously $Z_{11}^1=0$. Also $\deg( \{x_2,f_{E_7}\})-\deg(f_{E_7})=\deg(x_2)-2=10$ is not a multiple of $\deg(x_1)=8$, and hence $Z_{21}^1=0$. We leave is to the reader to verify from the bracket table that $Z_{31}^1=0$, and conclude that $f_{E_7}(x_1,x_2,x_3)=x_1^3x_2-x_3^2-108x_2^3$ is a Casimir generator.

\subsubsection{Kleinian singularity $E_8$}
In the case of the diagram  $E_8$ the group $G$ is the binary icosahedral group $\operatorname{BI}$ of order $120$. It is generated by the matrices
\begin{align*}
\left(
\begin{matrix}
\zeta^3&0\\
0&\zeta^2
\end{matrix}
\right), \quad\left(
\begin{matrix}
0&-1\\
1&0
\end{matrix}
\right),\quad \mbox{and }\quad {1\over \sqrt{5}}\left(
\begin{matrix}
-\zeta+\zeta^4&\zeta^2-\zeta^3\\
\zeta^2-\zeta^3&\zeta-\zeta^4\\
\end{matrix}
\right),
\end{align*}
where $\zeta=\exp(2\pi\I/5)$.
A set of complete polynomial invariants is given by
\begin{align*}
&\varphi_1(q,p)=qp(q^{10}+11q^5 p^5-p^{10}),\\
&\varphi_2(q,p)=-(q^{20}+p^{20})+228 (q^{15}p^5-q^{5}p^{15})-494q^{10}p^{10},\\
&\varphi_3(q,p)=q^{30}+p^{30}+522 (q^{25}p^5-q^{5}p^{25})-10005(q^{20}p^{10}+q^{10}p^{20}).
\end{align*}
Note that $\varphi_2$ is proportional to the determinant of the Hessian of
$\varphi_1$ and that $\varphi_3$ is proportional to the Jacobian $\partial(\varphi_1,\varphi_2)/\partial(q,p)$. The invariants can be shown to satisfy the relation
$\varphi_2^3+\varphi_3^2=1728\varphi_1^5$. We notice that $1728=12^3$.
For degree reasons $\{\varphi_1,\varphi_2\}=a\varphi_3$, $\{\varphi_1,\varphi_3\}
=b\varphi_2^2$ and  $\{\varphi_2,\varphi_3\}=c\varphi_1^4$ for certain proportionality factors $a,b,c\in \C$. A tedious calculation gives $a=20$, $b=-30$ and $c=-86400$. We notice that $86400=6\cdot 120^2=6|\operatorname{BI}|^2=(\deg{\varphi_1})^2\deg{\varphi_2}\deg{\varphi_3}$.  Hence for the binary icosahedral group the Poisson algebra is $\C[x_1,x_2,x_3]/(x_2^3+x_3^2-12^3x_1^5)$ with bracket table:
\begin{align*}
\begin{tabular}{c||c|c|c}
  $\{\:,\:\}_{E_8}$   & $x_1$&$x_2$&$x_3$\\\hline\hline
$x_1$ &   0  &$(\deg x_2)x_3$&$-(\deg x_3) x_2^2$\\\hline
$x_2$ &      &$0$    & $- (\deg{x_1})^2(\deg {x_2})(\deg {x_3})x_1^4$\\\hline
$x_3$ &      &      & $0$
\end{tabular}
\end{align*}
Obviously we have $Z_{11}^1=0$. Also $\deg(\{x_2,f_{E_8}\})-\deg(f_{E_8})=18$ is not a multiple of $\deg(x_1)$ and is smaller than $\deg(x_2)=20$, so $Z_{21}^1=0$. Finally, $\deg( \{x_3,f_{E_8}\})-\deg(f_{E_8})=28<\deg(x_3)=30$ is not in the $\Z_{\ge 0}$-span of $\deg(x_1)$ and $\deg(x_2)$, and hence $Z_{31}^1=0$. We conclude that $f_{E_8}(x_1,x_2,x_3)=x_2^3+x_3^2-12^3x_1^5$ is a Casimir generator.

\subsubsection{Symplectic quotient of the circle action with weight vector $(-1,1,1)$} \label{subsubsec:-1,1,1}

In this example $G$ is the complex circle $\C^\times$  and $V=\C^3$. The weight vector  of our circle action is supposed to be $(-1,1,1)$ (see also \cite{FHSSigma}). Identifying $\mathfrak g$ with $\C$, the corresponding moment map is
\[J(q_1,q_2,q_3,p_1,p_2,p_3)=-q_1p_1+q_2p_2+q_3p_3\in\C[q_1,q_2,q_3,p_1,p_2,p_3].\]
Since $\mathfrak g$ is abelian $J(q_1,q_2,q_3,p_1,p_2,p_3)$ is $G$-invariant.
A complete system of polynomial $G$-invariants is given by the quadratic polynomials
\begin{align*}
&\varphi_0 =q_1p_1, \quad \varphi_1 =q_2p_2, \quad \varphi_2 =q_3p_3, \quad \varphi_3 =q_1q_2, \quad\varphi_4 =p_1p_2, \\
&\varphi_5 =q_1q_3,  \quad\varphi_6 =p_1p_3,  \quad\varphi_7 =q_2p_3, \quad \varphi_8 =q_3p_2.
\end{align*}
The condition $J=0$ introduces the linear relation $\varphi_0=\varphi_1+\varphi_2$.  Moreover, the invariants restricted to $N$ satisfy the nine degree four relations
\begin{align*}
&f_1(\bs\varphi) =\varphi_3\varphi_6 -\varphi_1\varphi_7 -\varphi_2\varphi_7=0, \quad
f_2 (\bs\varphi)=\varphi_1\varphi_6 -\varphi_4\varphi_7=0,  \quad
f_3 (\bs\varphi)=\varphi_4\varphi_5 -\varphi_1\varphi_8 -\varphi_2\varphi_8=0, \\
&f_4 (\bs\varphi)=\varphi_1\varphi_5 -\varphi_3\varphi_8=0,  \quad
f_5 (\bs\varphi)=\varphi_2\varphi_4 -\varphi_6\varphi_8=0,  \quad
f_6 (\bs\varphi)=\varphi_2\varphi_3 -\varphi_5\varphi_7=0,\\
&f_7 (\bs\varphi)=\varphi_2^2 -\varphi_5\varphi_6 +\varphi_7\varphi_8=0,  \quad
f_8 (\bs\varphi)=\varphi_1\varphi_2 -\varphi_7\varphi_8=0,  \quad
f_9 (\bs\varphi)=\varphi_1^2 -\varphi_3\varphi_4 +\varphi_7\varphi_8=0.
\end{align*}
Hence the Poisson algebra of the symplectic quotient $N/\!\!/G$ is $S:=\C[\bs x]:=\C[x_i\mid 1\le i\le 8]$ modulo the ideal $I=(f_\mu(\bs x)\mid 1\le\mu\le 9)$. The internal degrees are all $\deg x_i=\deg\varphi_i=2$. The table of Poisson brackets is worked out in Table \ref{tab:brackettable-1,1,1}.
\begin{table}[h!]
\begin{align*}
\begin{tabular}{c||c|c|c|c|c|c|c|c}
  $\{\:,\:\}$   & $x_1$ & $x_2$& $x_3$ &    $x_4$ &  $x_5$ &$x_6$ & $x_7$ & $x_8$\\\hline\hline
  $x_1$         & $0$   &   $0$& $-x_3$&    $x_4$&  $0$   &$0$ & $-x_7$&   $x_8$\\\hline
  $x_2$         &       &   $0$&    $0$&       $0$&  $-x_5$   &  $x_6$ &$x_7$& $-x_8$\\\hline
  $x_3$         &       &      &    $0$&$2x_1+x_2$&  $0$   &$x_7$ &$0  $& $x_5$\\\hline
  $x_4$         &       &      &       &    $0$&   $-x_8$&$0 $ &  $-x_6$& $0$\\\hline
  $x_5$         &       &      &       &       &   $0$&  $x_1+2x_2$ &$x_3$&   $0$\\\hline
  $x_6$         &       &      &       &       &      &  $0$       &$0$& $-x_4$\\\hline
  $x_7$         &       &      &       &       &      &            &$0$& $-x_1+x_2$\\\hline
  $x_8$         &       &      &       &       &      &            &   & $0$
\end{tabular}
\end{align*}
\caption{Poisson brackets of the symplectic circle quotient with weight vector $(-1,1,1)$}
\label{tab:brackettable-1,1,1}
\end{table}
It turns out that the nine generators of the ideal $I$ are not Casimirs. The  Poisson relations of the generators  are depicted in Table \ref{tab:Poissonrelations-1,1,1}.
\begin{table}[h!]
\begin{align*}
\begin{tabular}{c||c|c|c|c|c|c|c|c|c}
 $\{\:,\:\}$   & $f_1$ & $f_2$& $f_3$ &    $f_4$ &  $f_5$ &$f_6$ & $f_7$ & $f_8$&$f_9$\\\hline\hline
 $x_1$         & $-f_1$   &   $0$& $f_3$&    $0$&  $f_5$   &$-f_6$ & $0$&   $0$&$0$\\\hline
 $x_2$         & $f_1$   &   $f_2$& $-f_3$&    $-f_4$&  $0$   &$0$ & $0$&   $0$&$0$\\\hline
 $x_3$         & $0$   & $f_1$ &$f_4$&  $0$  &   $f_7+2f_8$  &$0$ &$0$ & $f_6$  &$-f_6$\\\hline
 $x_4$         & $-f_2$   & $0$ &$0$&  $-f_3$  &   $0$  &$-f_7-2f_8$ &$0$ & $-f_5$  &$f_5$\\\hline
 $x_5$         & $f_6$   & $2f_8+f_9$ &$0$&  $0$  &   $f_3$  &$0$ &$-f_4$ & $f_4$  &$0$\\\hline
 $x_6$         & $0$   & $0$ &$-f_5$&  $-2f_8-f_9$  &   $0$  &$-f_1$ &$f_2$ & $-f_2$  &$0$\\\hline
 $x_7$         & $0$   & $0$ &$-f_7+f_9$&  $-f_6$  &   $f_2$  &$0$ &$f_1$ & $0$  &$-f_1$\\\hline
 $x_8$         & $f_7-f_9$   & $f_5$ &$0$&  $0$  &   $0$  &$-f_4$ &$-f_3$ & $0$  &$f_3$
\end{tabular}
\end{align*}
\caption{Poisson relations for the symplectic circle quotient with weight vector $(-1,1,1)$}
\label{tab:Poissonrelations-1,1,1}
\end{table}
We observe that the $Z_{i\mu}^\nu$ are constant and hence $\deg\left(\mathcal A_{\mu\nu}^\lambda\right)=2<4$ (cf.  Equation \eqref{eq:Amunu}). We found nonzero components in the tensor $(\mathcal A_{\mu\nu}^\lambda)_{\mu,\nu,\lambda}$. Closer inspection reveals that $\dP Z-[Z,Z]=0$ (cf. Proposition \ref{prop:deg-1} below). An explanation for this is that $I\subset S$ can be seen as an example of the type discussed in  Subsection \ref{subsec:deg-1} (see Proposition \ref{prop:deg-1}). For this interpretation, however, we have to divide $\deg(x_i)$, $\deg(f_\mu)$, and $\deg(\{\:,\:\})$ by $2$. If we interpret accordingly $\deg(f_\mu)$ as degree $2$ we observe that $A=S/I$ is a Koszul algebra.
This follows from the observation that $f_1,\dots,f_9$ forms a Groebner basis in the graded reverse lexicographic term order.


\subsubsection{Two particles in dimension three of zero total angular momentum}\label{subsubsec:angmom}

Let $V=\C^3\oplus\C^3$ and let $\bs q_1=(q_{11},q_{12},q_{13})$ be linear coordinates for the first copy of $\C^3$ in $V$ and $\bs q_2=(q_{21},q_{22},q_{23})$ linear coordinates for the second copy. We let the orthogonal group $G:=\mathrm O_3$ act diagonally on $V$. After identifying $V\times V^*=\C^3\oplus\C^3\oplus(\C^3)^*\oplus(\C^3)^*$ with $\C^3\oplus\C^3\oplus\C^3\oplus\C^3$ the cotangent lifted action is in fact the diagonal action.
We write $\bs p_1=(p_{11},p_{12},p_{13})$ for the linear coordinates of the third copy of $\C^3$ in $V\times V^*$ and $\bs p_2=(p_{21},p_{22},p_{23})$ for the linear coordinates of the fourth copy. A complete set of polynomial $G$-invariants is given by
\begin{align*}
    \varphi_1     &=  \langle \bs q_1,\bs q_1 \rangle,
    &
    \varphi_2     &=  \langle \bs q_1, \bs q_2 \rangle,
    &
    \varphi_3     &=  \langle \bs q_2, \bs q_2 \rangle,
    &
    \varphi_4     &=  \langle \bs q_1, \bs p_1 \rangle,
    &
    \varphi_5     &=  \langle \bs q_1, \bs p_2 \rangle,
    \\
    \varphi_6     &=  \langle \bs q_2, \bs p_1 \rangle,
    &
    \varphi_7     &=  \langle \bs q_2, \bs p_2 \rangle,
    &
    \varphi_8     &=  \langle \bs p_1,\bs  p_1 \rangle,
    &
    \varphi_9     &=  \langle \bs p_1, \bs p_2 \rangle,
    &
    \varphi_{10}  &=  \langle \bs p_2, \bs p_2 \rangle,
\end{align*}
where $\langle\:,\:\rangle$ denotes the Euclidean inner product on $\C^3$. The Poisson bracket is the canonical bracket $\{q_{mi},p_{lj}\}=\delta_{mn}\delta_{ij}$.
Then moment map is given by $J(\bs q_1,\bs p_1,\bs q_2,\bs p_2)=\bs q_1\wedge \bs p_1+\bs q_2\wedge \bs p_2$, where we use the identification of $\mathfrak o_3^*$ with $\wedge^2 \C^3$. Sending
$x_i\mapsto {\varphi_i}_{|N}$ we identify the Poisson algebra $\C[N\git G]$ with $\bs k[x_1,x_2,\dots,x_{10}]/I$, where $I$ is the Poisson ideal generated by
\begin{align*}
    f_1     &=   - x_4 x_9 + x_5 x_8 - x_6 x_{10} + x_7 x_9,
    \quad
    f_2     =   - x_2 x_9 - x_3 x_{10} + x_5 x_6 + x_7 x_7,
    \\
    f_3     &=   - x_2 x_8 - x_3 x_9 + x_4 x_6 + x_6 x_7,
   \quad
    f_4     =   - x_1 x_9 - x_2 x_{10} + x_4 x_5 + x_5 x_7,
    \\
    f_5     &=   - x_1 x_8 + x_3 x_{10} + x_4 x_4 - x_7 x_7,
    \quad
    f_6     =   - x_1 x_6 + x_2 x_4 - x_2 x_7 + x_3 x_5,
    \\
    f_7     &=   - x_3 x_8 x_{10} + x_3 x_9 x_9 + x_6 x_6 x_{10} - 2 x_6 x_7 x_9 + x_7 x_7 x_8,
    \\
    f_8     &=  x_2 x_6 x_{10} - 2 x_2 x_7 x_9 - x_3 x_4 x_{10} + x_3 x_5 x_9 - x_3 x_7 x_{10} + x_4 x_7 x_7 + x_7 x_7 x_7,
    \\
    f_9     &=  x_1 x_6 x_6 - x_2 x_2 x_8 - 2 x_2 x_3 x_9 + 2x_2 x_6 x_7 - x_3 x_3 x_{10} + x_3 x_7 x_7,
    \\
    f_{10}  &=   - x_1 x_3 x_{10} + x_1 x_7 x_7 + x_2 x_2 x_{10} - 2 x_2 x_5 x_7 + x_3 x_5 x_5,
    \\
    f_{11}  &=  x_1 x_3 x_6 x_{10} - x_1 x_6 x_7 x_7 - x_2 x_2 x_6 x_{10} + 2 x_2 x_2 x_7 x_9 - x_2 x_3 x_5 x_9
            \\&\quad
                + 2 x_2 x_3 x_7 x_{10} - 2 x_2 x_7 x_7 x_7 - x_3 x_3 x_5 x_{10} + x_3 x_5 x_7 x_7.
\end{align*}
For more detail the reader may consult \cite{CHS}.
For lack of space we will note spell out here the tensors $Z_{i\mu}^\nu$, $\mathcal A_{\mu\nu}^\lambda$ and $\dP Z-[Z,Z]$.

\subsection{Brackets of degree $-1$} \label{subsec:deg-1}
Brackets of degree $-1$ are deduced from the so-called \emph{linear Poisson} structures.
Let us for simplicity restrict the discussion to the case $\bs k=\C$.
Let $V$ be a finite dimensional $\C$-vector space and assign to each element in $V^*$ the internal degree $1$. A Poisson bracket of degree $-1$ on the algebra $S=\C[V]$ is nothing other than a $\C$-linear Lie algebra structure on $V^*$. Let us write for this Lie algebra $\mathfrak g$ so that $V=\mathfrak g^*$. Let $G$ be the connected, simply connected Lie group associated to $\mathfrak g$. Now we observe that an ideal $I\in S$ is a Poisson ideal if and only if the fundamental vector fields $X_{\mathfrak g^*}$ of the coadjoint $G$-action preserve the ideal $I$, i.e., $X_{\mathfrak g^*}I\subseteq I$ for all $X\in \mathfrak g$. This means that $I$ is the ideal of the Zariski closure $\overline{GW}$ of the $G$-saturation $GW$ of a subset $W\subseteq \mathfrak g^*$. For aesthetic reasons we would like to restrict the attention to ideals $I$ generated by homogeneous polynomials. Notice that if $W$ is conical, i.e., $tW\subseteq W$ for all $t\in \C^\times$, then $I$ will have this property. Another important class of examples is provided by nilpotent orbit closures (see e.g. \cite{Collingwood}).

One way to construct concrete examples is to decompose the $\mathfrak g$-module $\C[\mathfrak g^*]$ into irreducible $\mathfrak g$-modules. Knowing explicit bases for the irreducible components one can form finite-dimensional $\mathfrak g$-submodules $U\subset S$. Their ideals $I_U=\{f\in S\mid f=\sum_\mu a_\mu f_\mu,\; a_\mu\in S, f_\mu\in U\}$ are clearly Poisson. We emphasize that for the situation described above the tensors $Z_{i\mu}^\nu$ can be chosen to be $-1$ times the representation matrices of the $\mathfrak g$-module $U$. In particular, in this case the $Z_{i\mu}^\nu$
are \emph{constant} and as such Casimirs in $S$. In principle one can also consider the categorical quotient $U^{-1}(0)/\!\!/G$, i.e., $\Spec(S^G/I_U\cap S^G)$. The Poisson bracket on $\Spec( S/I_U)$ descends to a Poisson bracket on $U^{-1}(0)/\!\!/G$ of degree $-1$. For lack of space we refrain from elaborating examples of this type.

\begin{proposition} \label{prop:deg-1} In the situation described above we have
\begin{enumerate}
\item \label{item:MClinear} $\dP Z-[Z,Z]=0$,
\item $I_U$ is generated by Casimirs if and only if $U\subset S^G$,
\item \label{item:degAmunulambda} if  $U\subseteq S_m$ then $\deg\left(\mathcal A_{\mu\nu}^\lambda\right)=m-1$, and hence a nonzero $\mathcal A_{\mu\nu}^\lambda$ cannot be in $I_U$.
\end{enumerate}
\end{proposition}
\begin{proof}
The fundamental vector fields of the coadjoint $\mathfrak g$-action are given by $-\{x_i,\:\}=-\sum_{j,k} c_{ij}^k x_k\partial /\partial x _j$. Here $c_{ij}^k$ are the structure constants of the Lie algebra in the basis $e_1,\dots,e_n$ of $\mathfrak g$ corresponding to the choice of linear coordinates  $x_1,\dots,x_n$ for $\mathfrak g^*$, i.e.,
$[e_i,e_j]=\sum_k c_{ij}^k e_k$.
To show \eqref{item:MClinear} note that, since $Z_{i\mu}^\nu$ are Casimir, we have
\begin{align*}
\dP\left(\sum_k Z_{k\mu}^\nu{\partial\over\partial x_k}\right)(\dR x_i,\dR x_j)=-\sum_{k} Z_{k\mu} {
\partial \Lambda_{ij}
\over \partial x_k}=-\sum_{k} c_{ij}^k Z_{k\mu}^\nu=[Z_i,Z_j]_\mu^\nu.
\end{align*}
The other statements are obvious.
\end{proof}

\subsubsection{Harmonic polynomials in $3$ variables}\label{subsubsec:harmonic} Let us consider $\mathfrak g=\mathfrak{so}_3$. Note that the adjoint representation is isomorphic to the dual of the standard representation $\C^3$.
Hence $S=\C[\mathfrak{so}_3^*]=\C[x_1,x_2,x_3]$ with the standard grading and Poisson bracket
\[\{x_1,x_2\}=x_3, \quad \{x_2,x_3\}=x_1,\quad \{x_3,x_1\}=x_2.\]
The subspaces of $S$ of harmonic polynomials of degree $\ell$
\[\mathcal H_\ell:=\{f\in S\mid \Delta f=0,\;\deg(f)=\ell\}\]
form the irreducible components in the decomposition $S=\oplus_{\ell\ge 0}\mathcal H_\ell$ of the $\mathfrak{so}_3$-module $S$. Here $\Delta=\sum_i(\partial/\partial x_i)^2$ is the Laplacian.
The simplest nontrivial case is that of degree $2$. Here $\mathcal H_2$ is the $\C$-span of
\begin{align*}
f_1=x_1x_2,\quad f_2=x_1x_3,\quad f_3=x_2x_3,\quad f_4=x_1^2-x_2^2,\quad f_5=x_1^2-x_3^2.
\end{align*}
We can easily work out the bracket table:
\begin{align*}
\begin{tabular}{c||c|c|c|c|c}
 $\{\:,\:\}$   & $f_1$ & $f_2$& $f_3$ &    $f_4$ &  $f_5$ \\\hline\hline
 $x_1$         & $f_2$   &   $-f_1$& $f_4-f_5$&    $-2f_3$&  $2f_3$   \\\hline
 $x_2$         & $-f_3$  &   $f_5$&   $f_1$&    $-2f_1$& $-4f­_2$ \\\hline
 $x_3$         & $-f_4$   & $f_3$ &$-f_2$&$4f_1$&2$f_1$
\end{tabular}
\end{align*}
None of the generators $f_\mu$ are Casimirs. The Hilbert series of $A$ is $(1-5t^2+5t^3-5t^5)/(1-t)^3=1+3t+t^2$. This means that $f_1,\dots, f_5$ is not a complete intersection and $A$ is actually a finite dimensional algebra. However, as $f_1,\dots,f_5$ forms a Groebner basis for $I$ (e.g. in the graded reverse lexicographic term order), $A=S/I$ is a Koszul algebra. We found nonzero components in the tensors $(\mathcal A_{\mu\nu}^\lambda)_{\mu,\nu,\lambda}$.


\subsubsection{$\ell\times \ell$-minors of a generic $m\times m$-matrix}\label{subsubsec:detideal}
Suppose $S=\bs k[x_{ij}|1\le i,j\le m]$ and $U$ is the vector subspace of $S$ generated by the $\ell\times \ell$-minors. We denote by $I$ the ideal generated by $U$. The spectrum of $S/I$ is the determinantal variety of $m\times m$ matrices that have rank $\le \ell-1$. The ideal $I$ is stable under the action of $\mathfrak g=\mathfrak{gl}_m$.
Using the coordinates $x_{ij}$ on $\mathfrak{gl}_m^*$ the commutation relations are
\[\{x_{ij},x_{kl}\}=\delta_{jk}x_{il}-\delta_{li}x_{kj}.\]
To be more specific we take a look into the case $\ell=2$ and $m=3$. We have nine $2\times 2$-minors
in nine variables $x_{ij}$:
\begin{align*}
f_{11}=x_{22}x_{33}-x_{23}x_{32},\quad f_{12}=x_{21}x_{33}-x_{23}x_{31}, \quad f_{13}=x_{21}x_{32}-x_{22}x_{31},\\
f_{21}=x_{12}x_{33}-x_{13}x_{32},\quad  f_{22}=x_{11}x_{33}-x_{13}x_{31},\quad f_{23}=x_{11}x_{32}-x_{12}x_{31},\\
f_{31}=x_{12}x_{23}-x_{13}x_{22},\quad f_{32}=x_{11}x_{23}-x_{13}x_{21},\quad f_{33}=x_{11}x_{22}-x_{12}x_{21}.
\end{align*}
We can check by hand that $I$ is preserved be the $\mathfrak{gl}_3$-action by writing
\begin{align*}
&\{x_{rs},x_{il}x_{jk}-x_{ik}x_{jl}\}=(x_{rl}\delta_{si}-x_{is}\delta_{rl})x_{jk}+(x_{rk}\delta_{sj}-x_{js}\delta_{rk})x_{il}-(x_{rk}\delta_{si}-x_{is}\delta_{rk})x_{jl}-(x_{rl}\delta_{sj}-x_{js}\delta_{rl})x_{ik}\\
&=\delta_{si}(x_{rl}x_{jk}-x_{rk}x_{jl})-\delta_{rl}(x_{is}x_{jk}-x_{ik}x_{js})+\delta_{sj}(x_{rk}x_{il}-x_{rl}x_{ik})-\delta_{rk}(x_{js}x_{il}-x_{is}x_{jl}).
\end{align*}
We conclude that $A=\bs k[x_{ij}|1\le i,j\le 3]/(f_{ij}|1\le i,j\le 3)$ is a Poisson algebra and notice that none of the $f_{ij}$ are Casimirs. The ideal $I$ is not a complete intersection. However, as $A$ is a quadratic extremal Gorenstein algebra it is a Koszul algebra \cite{Froeberg,polishchuk2005quadratic}. We found nonzero components in the tensors $(\mathcal A_{\mu\nu}^\lambda)_{\mu,\nu,\lambda}$. We expect $A/\operatorname{tr}$
to be isomorphic to the Poisson algebra of Subsection \ref{subsubsec:-1,1,1} but have not been able to work out the concrete isomorphism. Here $\operatorname{tr}=\sum_{i=1}^3 x_{ii}$ is a Casimir obtained by taking the trace of the generic  $3\times 3$-matrix $(x_{ij})$.

\subsection{Brackets of degree  $\ge 0$}\label{subsec:brackets of degree >=0}

\subsubsection{Diagonal Poisson bracket}\label{subsec:diag}
\begin{lemma} Let $S=\bs k[x_1,x_2,\dots,x_n]$ be the polynomial algebra with $\deg(x_i)=1$ for all $i$. For an antisymmetric $n\times n$-matrix $(c_{ij})_{i,j=1,\dots,n}$ consider the diagonal Poisson bracket
\[\{x_i,x_j\}:=c_{ij}x_ix_j,\quad i,j=1,\dots n.\]
Then for any monomial  $f=\prod_{j=1}^n x_j^{m_j}$ we have $\{x_i,f\}=\sum_{j=1}^{n}c_{ij}m_jx_i f$.
\end{lemma}

Let us introduce an antisymmetric bilinear form $(\:,\:)$ on $\bs k^n$ by $(e_i,e_j):=c_{ij}$ for $i,j=1,\dots,n$, where $e_1,\dots, e_n$ denotes the standard basis. Also for a monomial $f=\prod_{j=1}^n x_j^{m_j}$ we use the shorthand $\bs x^{\bs m}$, where $\bs m=(m_1,\dots,m_n)\in \bs k^n$ is the vector of exponents of $f$.

\begin{proposition} \label{prop:diag} Let $(S,\{\:,\:\})$ be the  Poisson algebra of the previous lemma.
Then any monomial ideal $I=(f_1,f_2,\dots,f_k)$ is a Poisson ideal. Writing $f_\mu=\bs x^{\bs m_\mu}$ with $\bs m_\mu\in \bs k^n$ the tensors $Z_{i\mu}^\nu$ in
$\{x_i,f_\mu\}=\sum_\nu Z_{i\mu}^\nu f_\nu$ can be chosen to be $Z_{i\mu}^\nu =(e_i,\bs m_\mu)x_i\delta_\mu^\nu$. As a consequence the tensors $\mathcal A_{\mu \nu}^\lambda$ (cf. Equation \eqref{eq:Amunu}) can be written as
\begin{align*}
\mathcal A_{\mu \nu}^\lambda =(\bs m_\mu,\bs m_\nu)(f_\mu\delta_{\nu}^\lambda-f_\nu\delta_{\mu}^\lambda).
\end{align*}
Moreover, $\dP Z=0=[Z,Z]$.
\end{proposition}
\begin{proof} We notice that $\dP Z=0$ is equivalent to
$0=\{x_i,Z_{j\mu}^\nu\}-\{x_j,Z_{i\mu}^\nu\}-\sum_\ell Z_{\ell\mu}^\nu \partial \Lambda_{ij}/\partial x_\ell$,
which in turn can be easily verified. On the other hand $[Z,Z]$ vanishes since the matrices $Z_i$ are diagonal.
\end{proof}

We emphasize that we are not obliged to choose $Z_{i\mu}^\nu$ as in the Proposition. Actually there are examples with  choices of $Z_{i\mu}^\nu$ that make both $\mathcal A_{\mu \nu}^\lambda$ and $\dP Z-[Z,Z]$ vanish (cf. Subsection \ref{subsec:nondiag}).

\subsubsection{Brackets from the Volterra system} Let $S=\bs k[x_1,x_2,\dots,x_n]$ with $n\ge 2$ and
$\deg(x_i)=1$ for all $i=1,\dots,n$.
The following compatible Poisson brackets  $\{\:,\:\}_0$ of degree zero and $\{\:,\:\}_1$ degree one can be deduced from \cite{faddeev2007hamiltonian} (see also \cite{Damianou}). They are defined by
\begin{align*}
&\{x_i,x_{i+1}\}_0:=x_ix_{i+1},\\
&\{x_i,x_{i+1}\}_1:=x_ix_{i+1}(x_i+x_{i+1}), \qquad
\{x_i,x_{i+2}\}_1:=x_ix_{i+1}x_{i+2}.
\end{align*}
All other brackets between coordinates are understood to be zero. The bracket $\{\:,\:\}_0$ is a diagonal Poisson bracket. The bracket  $\{\:,\:\}_1$ also preserves monomial ideals. If the ideal $I$ is generated by monomials $f_1,\dots,f_k$ with $f_\mu=\prod_i x^{m_{\mu,i}}$ then for the bracket $\{\:,\:\}_1$ the we get
\begin{align*}
Z_{i\mu}^\nu&=  \big( (m_{\mu,i+1} + m_{\mu,i+2}) x_i x_{i+1} - (m_{\mu,i-1} + m_{\mu,i-2}) x_i x_{i-1}
                + (m_{\mu,i+1} - m_{\mu,i-1}) x_i^2 \big) \delta_\mu^\nu,\\
\mathcal A_{\mu \nu}^\lambda&=   \sum_{i=1}^n \Big( m_{\mu,i}
                \big( (m_{\nu,i+1} + m_{\nu,i+2}) x_{i+1} - (m_{\nu,i-1} + m_{\nu,i-2}) x_{i-1}
                + (m_{\nu,i+1} - m_{\nu,i-1}) x_i \big) \delta_\nu^\lambda f_\mu
        \\&\quad
                + m_{\nu,i}
                \big( (m_{\mu,i+1} + m_{\mu,i+2}) x_{i+1} - (m_{\mu,i-1} + m_{\mu,i-2}) x_{i-1}
                + (m_{\mu,i+1} - m_{\mu,i-1}) x_i \big) \delta_\mu^\lambda f_\nu \Big),\\
\dP Z&=0=[Z,Z].
\end{align*}
The verification of the last identity is sort of tedious.
A monomial $f$ is Casimir for $\{\:,\:\}_0$ if and only if $n$ is odd and $f$ is a power of $x_1x_3x_5\cdots x_{n-2}x_n$. The bracket $\{\:,\:\}_1$ does not have nontrivial polynomial Casimirs.

\subsection{Other brackets}

\begin{theorem}\label{thm:detbracket}
Let $\bs k$ be $\R$ or $\C$ and $S=\bs k[x_1,x_2,\dots, x_n]$. Assume that $X^1,X^2,\dots,X^{k+2}\in D_S$ are pairwise commuting derivations, $g\in S$, and $I$ is the ideal of $S$ generated by $f_1,f_2,\dots, f_k\in S$.
For $a=:f_{k+1}, b=:f_{k+2}\in S$, let
\begin{align} \label{eq:Palamodovgeneral}
\{a,b\}=g\operatorname{Det}\left(((X^\nu(f_\mu))_{\mu,\nu=1,2,\dots,k+2}\right).
\end{align}
This defines a Poisson bracket on $S$ such that $f_1,f_2,\dots, f_k$ are Casimirs.  In the graded case the degree of the bracket is $\deg(\{\:,\:\})=\sum_{\mu=1}^k\deg(f_\mu)+\sum_{\nu=1}^{k+2}\deg(X^\nu).$
\end{theorem}

In \cite[Proposition 4.2]{PalamodovInfDef} it is only shown that this  defines a bracket in $S/I$. A more direct proof can be deduced from \cite{Filippov} and a special case is treated in \cite[Section 8.3]{LGPV}. For the convenience of the reader we include another demonstration.

\begin{proof} Let us assume that $\bs k=\R$ (the proof for $\bs k=\C$ is analogous to the real case).
Let $\mathcal U\subset \R ^n$ be the open subset where  $X^1,X^2,\dots,X^{k+2}$ are linearly independent. On $\R^n\backslash \mathcal U$ the bracket is identically zero and there is nothing to show. Let $x\in \mathcal U$
and $\phi$ be a diffeomorphism from a neighborhood $\mathcal V$ of $0\in \{(t_1,\dots t_n)\in \R^n\} $ to a neighborhood $ \mathcal U_x$ of $x$ in $\mathcal U$ such that $T\phi$ sends $\partial/\partial t_\nu$ to $X^\nu$ for $\nu=1,\dots,k+2$.

Let us write $F_\mu:=f_\mu\circ \phi$ for $\mu=1,\dots, k$ and $G:=g\circ \phi$. For $A=:F_{k+1}, B=:F_{k+2}\in\mathcal C^\infty(\mathcal V)$ define a skew-symmetric bracket
\begin{align*}
\{A,B\}:=G\operatorname{Det}\left(\left({\partial F_\mu\over \partial t_\nu}\right)_{\mu,\nu=1,2,\dots,k+2}\right).
\end{align*}
It is enough to show that this is a Poisson bracket since $\{a,b\}\circ \phi=\{a\circ \phi,b\circ \phi\}$. Let
$\mathcal J:=\operatorname{Ann}(\mathcal F)=\oplus_{p\ge 0} \mathcal J^p$,
\begin{align*}
\mathcal J^p:=\{\alpha \in\Omega^p(\mathcal V)\mid \alpha(V_1,\dots,V_p)=0\;\;\forall V_1,\dots,V_p\in \mathcal F\}
\end{align*}
be the annihilator in the algebra of smooth differential forms $\Omega(\mathcal V)$ of the distribution $\mathcal F=\mathcal C^\infty(\mathcal V)\{\partial /\partial t_\nu\mid 1 \le \nu\le k+2\}$. It is the differential ideal generated by $\dR t_\mu$ for $\mu=k+3,\dots,n$. Now $\{A,B\}$ is uniquely determined by the relation
\begin{align*}
G\dR F_1\wedge \dots \wedge \dR F_k\wedge \dR A\wedge \dR B+\mathcal J=\{A,B\}\;\dR t_1\wedge \dots \wedge \dR t_{k+2}+\mathcal J.
\end{align*}
As $\dR F_1\wedge \dots \wedge \dR F_k\wedge \dR (A_1A_2)\wedge \dR B+\mathcal J=A_1\dR F_1\wedge \dots \wedge \dR F_k\wedge \dR A_1\wedge \dR B+A_2\dR F_1\wedge \dots \wedge \dR F_k\wedge \dR A_2\wedge \dR B+\mathcal J$ the bracket satisfies the Leibniz rule in the first argument. Hence there exists a bivector $\pi=\sum_{i,j=1}^{k+2}\pi_{ij}\partial /\partial t_i\wedge\partial /\partial t_j\in \wedge ^2D_{\mathcal C^\infty(\mathcal V)}$ in the second exterior power of the $\mathcal C^\infty(\mathcal V)$-module of derivations $D_{\mathcal C^\infty(\mathcal V)}$ of $\mathcal C^\infty(\mathcal V)$ such that $\{A,B\}=-i_\pi(\dR A\wedge \dR B)$. Here $i_\pi: \Omega^\bullet(\mathcal V)\to  \Omega^{\bullet-2}(\mathcal V)$ is the contraction with the bivector $\pi$. Recall that the Lie derivative $L_\pi=[i_\pi,\dR]=i_\pi\circ \dR-\dR\circ i_\pi: \Omega^\bullet(\mathcal V)\to  \Omega^{\bullet-1}(\mathcal V)$ along
$\pi$ is a second order superdifferential operator annihilating $\Omega^0(\mathcal V)$ and $\Omega^1(\mathcal V)\cap Z(\mathcal V)$, where $Z(\mathcal V)$ is the space of closed forms. It is straightforward to check that $L_\pi(\mathcal J\cap Z(\mathcal V))\subseteq \mathcal J$.

If $\alpha,\beta$, and $\gamma$ are differential forms of degree  $|\alpha|,|\beta|, |\gamma|$, then
\begin{align}\label{eq:2ndOrder}
\nonumber L_\pi(\alpha\wedge\beta\wedge\gamma)&=(-1)^{|\alpha|}\alpha\wedge L_\pi(\beta\wedge\gamma)+(-1)^{|\beta|(|\alpha|-1)}\beta\wedge L_\pi(\alpha\wedge\gamma)+(-1)^{|\gamma|(\beta|+|\alpha|-1)}\gamma\wedge L_\pi(\alpha\wedge\beta)\\
&\qquad\qquad-(-1)^{|\alpha|+|\beta|}\alpha\wedge\beta\wedge L_\pi(\gamma)-(-1)^{|\alpha|+|\gamma|(|\beta|-1)}\alpha\wedge\gamma\wedge L_\pi(\beta).
\end{align}
This can be easily shown from the fact that the supercommutator of the left multiplication $\alpha \wedge $ and $ L_\pi$ is a derivation of $\Omega(\mathcal V)$. In the special case when $\alpha=\dR A,\beta=\dR B$, and $\gamma=\dR C$ this simplifies to
\begin{align*}
L_\pi( \dR A\wedge\dR B\wedge\dR C) &=- \dR A\wedge L_\pi(\dR B\wedge\dR C)+\dR B\wedge L_\pi( \dR A\wedge\dR C)-\dR C \wedge L_\pi( \dR A\wedge\dR B)\\
&=-\dR A\wedge \dR\{B,C\}-\dR B\wedge \dR\{C,A\}-\dR C\wedge \dR\{A,B\}.
\end{align*}
Consequently, with $\operatorname{Jac}(A,B,C):=\{A,\{B,C\}\}+\{B,\{C,A\}\}+\{C,\{A,B\}\}$ we have
\begin{align*}
\operatorname{Jac}(A,B,C)\;\dR t_1\wedge \dots \wedge \dR t_{k+2}+\mathcal J=-G\dR F_1\wedge \dots \wedge \dR F_k\wedge
L_\pi( \dR A\wedge\dR B\wedge\dR C)+ \mathcal J,
\end{align*}
and it remains to show that $\dR F_1\wedge \dots \wedge \dR F_k\wedge
L_\pi( \dR A\wedge\dR B\wedge\dR C)\in \mathcal J$. Let $\alpha\in\Omega(\mathcal V)$ be a closed form. Since $\{F_\mu,\:\}=0$ by construction of $\pi$ it follows that $L_\pi(\dR F_\mu\wedge\alpha)=-\dR F_\mu\wedge L_\pi(\alpha)$. Hence we have
\begin{align*}
\dR F_1\wedge \dots \wedge \dR F_k\wedge
L_\pi( \dR A\wedge\dR B\wedge\dR C)=(-1)^k L_\pi(\dR F_1\wedge \dots \wedge \dR F_k\wedge
 \dR A\wedge\dR B\wedge\dR C).
\end{align*}
But $\dR F_1\wedge \dots \wedge \dR F_k\wedge\dR A\wedge\dR B\wedge\dR C$, being a $(k+3)$-form, is in $\mathcal J\cap Z(\mathcal V)$.
\end{proof}


\section{Complete intersections and the conormal sequence}\label{sec:ci}

Let $S=\bs k[x_1,x_2,\dots, x_n]$ be our polynomial algebra and consider the ideal $I=(f_1,f_2,\dots,f_k)$ generated by homogeneous polynomials $f_1,f_2,\dots,f_k\in S$.
We write $A=S/I$ for the quotient algebra. We use Latin indices $i,j,\dots$ to index the $x$'s and Greek indices $\mu,\nu,\dots$ to index the $f$'s.

The \emph{Koszul complex} $K_\bullet(S,(f_1,f_2,\dots,f_k))=: K_\bullet(S,\bs f)$ is as an $S$-algebra the graded polynomial algebra $S[y_1,y_2,\dots,y_k]$ in the \emph{odd} variables
$y_\mu$, i.e., we have $y_\mu y_\nu=-y_\nu y_\mu$. The \emph{Koszul differential} $\partial$ is the unique $S$-linear derivation that sends  $y_\mu$ to $f_\mu$. In other words, we can write $\partial$ as a vector field
\begin{align*}
\partial=\sum_{\mu=1}^k f_\mu{\partial\over \partial y_\mu}.
\end{align*}
It is clear that $\partial^2=0$. The homological degree $|y_\mu|=1$. We see that $(K_\bullet(S,\bs f), \partial)$ is a chain complex. In fact,
$(K_\bullet(S,\bs f), \partial)$ is a supercommutative dg algebra with $S$-linear differential $\partial$.
We also declare the internal degree to be $\deg(y_\mu):=\deg(f_\mu)$ so that $\partial$ has internal degree zero. Note that the zeroth homology module $H_0{K}(S,\bs f)$ is isomorphic as an $S$-algebra to $A=S/I$. We say that $\bs f= (f_1,f_2,\dots,f_k)$ is a \emph{complete intersection} if the homologies $H_i{K}(S,\bs f)$ are trivial in homological degree $i\ge 1$.

Based on the coefficients $Z_{i\mu}^\nu$ in the Poisson relation $\{x_i,f_\mu\}=
\sum_\nu Z_{i\mu}^\nu f_\nu$, we defined the tensor $\dP Z-[Z,Z]$ in $\mathrm C^1_{\operatorname{Poiss}}(S)\otimes_S\mathfrak{gl}_k(S)$ (see Equation \eqref{eq:MC}).

\begin{theorem} \label{thm:MC} Suppose $\bs f$ is a complete intersection. Then image of  $\dP Z-[Z,Z]$ in
$\mathrm C^2_{\operatorname{Poiss}}(S)\otimes_S\mathfrak{gl}_k(A)$ vanishes. Moreover, the image of $Z$ in $\mathrm C^1_{\operatorname{Poiss}}(S)\otimes_S\mathfrak{gl}_k(A)$ is uniquely determined by $\bs f$.
\end{theorem}

\begin{proof}
We use Einstein summation convention, i.e., summation over repeated indices is understood. If two terms $X$ and $Y$ are equal modulo $I$ we will write $X\sim Y$.

From Appendix \ref{ap:Poissoncohomology} we deduce
\begin{align}\label{eq:Poiss2coboundary}
\dP\left(Z_{i\mu}^\nu {\partial
\over\partial x_i}\right)(\dR x_i,\dR x_j)=\{x_i,Z_{j\mu}^\nu\}-\{x_j,Z_{i\mu}^\nu\}- Z_{k\mu}^\nu{\partial\Lambda_{ij}\over\partial x_k}.
\end{align}
Note furthermore that
\begin{align}\label{eq:ZLambdaidentity}
Z_{i\mu}^\nu f_\nu=\{x_i,f_\mu\}=\Lambda_{ij}{\partial f_\mu\over \partial x_j}.
\end{align}
We proceed by analyzing the Jacobi identity
\begin{align*}
&0=\{x_i,\{x_j,f_\mu\}\}-\{x_j,\{x_i,f_\mu\}\}+\{f_\mu,\{x_i,x_j\}\}\\
&=\{x_i,Z_{j\mu}^\nu f_\nu\}- \{x_j,Z_{i\mu}^\nu f_\nu\} +\{f_\mu,\Lambda_{ij}\}\\
&=\{x_i,Z_{j\mu}^\nu\} f_\nu+Z_{j\mu}^\nu Z_{i\nu}^\lambda f_\lambda -
\{x_j,Z_{i\mu}^\nu\} f_\nu-Z_{i\mu}^\nu Z_{j\nu}^\lambda f_\lambda
+\Lambda_{k\ell}{\partial f_\mu\over \partial x_k}{\partial \Lambda_{ij}\over \partial x_\ell}\\
&=\left(\{x_i,Z_{j\mu}^\nu\}-
\{x_j,Z_{i\mu}^\nu\}\right)f_\nu+([Z_j,Z_i])_\mu^\nu f_\nu
-Z_{\ell\mu}^\nu f_\nu{\partial \Lambda_{ij}\over \partial x_\ell}\\
&=\partial\left(\left(\{x_i,Z_{j\mu}^\nu\}-
\{x_j,Z_{i\mu}^\nu \}-Z_{\ell\mu}^\nu {\partial \Lambda_{ij}\over \partial x_\ell} -([Z_i,Z_j])_\mu^\nu\right)y_\nu\right).
\end{align*}
This means that the expression in the last equation is a Koszul $1$-cycle, hence a Koszul $1$-boundary and as such $\sim 0$.
We conclude that
\begin{align*}
\dP\left(Z_{i\mu}^\nu {\partial
\over\partial x_i}\right)\sim {1\over 2}([Z_i,Z_j])_\mu^\nu {\partial
\over\partial x_i}\wedge{\partial
\over\partial x_j}.
\end{align*}
The uniticity statement is a rephrasing of a theorem of Vasconcelos (see \cite[Theorem 19.9]{Matsumura}).
\end{proof}

\begin{corollary} \label{cor:principalideal} If $I=(f_1)$ is a principal ideal with $f_1\ne 0\in S$, then  $Z=\sum_i Z_{i1}^1\partial /\partial x_i\in \mathrm C^1_{\operatorname{Poiss}}(S)\cong\mathrm C^1_{\operatorname{Poiss}}(S)\otimes \mathfrak {gl}_1(S) $ is a cocycle: $\dP Z=0$.
If in addition $Z_{i1}^1$ and $Z_{j1}^1$ are Casimirs in $S$, then $Z(\{x_i,x_j\})=0\in S$. If $Z_{i1}^1$ is a Casimir for all $i$, then $Z$ commutes with all Hamiltonian vector fields: $[Z,\dP S]=0$.
\end{corollary}
\begin{proof}
Observing that the commutator in $\mathfrak{gl}_1$ vanishes, we analyze the relation
\begin{align*}
0=\left(\{x_i,Z_{j1}^1\}-
\{x_j,Z_{i1}^1\}
-Z_{\ell1}^1 {\partial \Lambda_{ij}\over \partial x_\ell}\right)f_1
\end{align*}
from the proof of the previous theorem. As the complement of the locus of $f_1$ is dense, we can discard the factor of  $f_1$. The claims follow from the relation $Z(\{x_i,x_j\})=Z_{\ell1}^1 {\partial \Lambda_{ij}/ \partial x_\ell}$.
\end{proof}

In the following proposition we use our standing assumption that $S$ has a $\Z_{\ge 0}$-grading compatible with the Poisson bracket such that $S_0=\bs k$.

\begin{proposition} Let $I$ be a graded principal Poisson ideal in $S$ and suppose that the first Poisson cohomology $\mathrm H^1_{\operatorname{Poiss}}(S)$ is trivial. Then any generator of the principal Poisson ideal $I$ is a Casimir.
\end{proposition}
\begin{proof}
Let us simplify the notation by writing $f=f_1$ and $Z_i=Z_{i1}^1$. As $Z$ is a Poisson coboundary there is a function $a\in S$ such that $Z_i=\{x_i,a\}$ for all $i=1,\dots,n$. We can assume $a\in S$ is homogeneous.
By the definition of $Z_i$ we have  $\{x_i,f\}=\{x_i,a\}f$.
This implies $\deg(a)=0$, hence $a$ is a scalar, which means that $Z_i=0$.
\end{proof}

\begin{corollary} For any of the Poisson brackets of Subsection \ref{subsec:brackets of degree >=0}  $\mathrm H^1_{\operatorname{Poiss}}(S)$  is nontrivial.
\end{corollary}

\begin{lemma} Assuming $\bs f$ is a complete intersection, let $g_1,g_2,\dots,g_\ell\in S$ be another set of generators for the ideal $I$
such that each $g_\alpha$ is a Casimir modulo I, i.e., $\{x_i,g_\alpha\}\in I$ for all $i$ and $\alpha$. For the transition matrices $N\in S^{\ell\times k}$ and $M\in S^{k\times \ell}$ with
\[g_\alpha=\sum_{\mu=1}^k N_\alpha^\mu f_\mu,\qquad  f_\mu=\sum_{\alpha=1}^\ell M_\mu^\alpha g_\alpha\]
we have
\begin{align*}
Z_{i\mu}^\nu+I=-\sum_{\alpha=1}^\ell M_\mu^\alpha \{x_i,N_\alpha^\nu\}+I=\sum_{\alpha=1}^\ell N_\alpha^\nu\{x_i,M^\alpha_\mu\}+I.
\end{align*}
This can be also written in the form $Z=-M\dP N=(\dP M)N \in\mathrm C^1_{\operatorname{Poiss}}(S)\otimes_S\mathfrak{gl}_k(A)$.
\end{lemma}

\begin{proof} We use the same shorthand notations as in the  proof of Theorem \ref{thm:MC}. We have
$f_\nu=M_\nu^\alpha N_\alpha^\mu f_\mu$. We write this as $(\delta_\nu^\mu-M_\nu^\alpha N_\alpha^\mu)f_\mu=0$ using the Kronecker $\delta$-symbol. Therefore
\begin{align}\label{eq:rightinverse}
(\delta_\nu^\mu-M_\nu^\alpha N_\alpha^\mu)y_\mu
\end{align}
is a Koszul $1$-cycle, hence a Koszul $1$-boundary and as such $\sim 0$. We conclude that $M_\nu^\alpha N_\alpha^\mu\sim \delta_\nu^\mu$.

Next we are analyzing
\begin{align*}
0\sim\{x_i,g_\alpha\}=\{x_i,N_\alpha^\nu f_\nu\}=\{x_i,N_\alpha^\nu\} f_\nu+N_\alpha^\nu Z_{i\nu}^\mu f_\mu=\left(\{x_i,N_\alpha^\nu\} +N_\alpha^\mu Z_{i\mu}^\nu \right)f_\nu.
\end{align*}
We conclude that $\left(\{x_i,N_\alpha^\nu\} +N_\alpha^\mu Z_{i\mu}^\nu \right)y_\nu$ is a Koszul $1$-cycle modulo $I$, hence a  Koszul $1$-boundary modulo $I$ and as such $\sim 0$. We write this as $N_\alpha^\mu Z_{i\mu}^\nu\sim \{x_i,N_\alpha^\nu\}$. Multiplying this with matrix $M$ we arrive at $Z_{i\mu}^\nu\sim-M_\mu^\alpha \{x_i,N_\alpha^\nu\}$. Applying \eqref{eq:rightinverse} and the fact that $\delta_\mu^\nu$ is a Casimir we can rewrite this as $N_\alpha^\nu\{x_i,M^\alpha_\mu\}$.
\end{proof}

Using $\bs f=(f_1,f_2,\dots,f_k)$, we form a two-term chain complex $\mathbb L=\mathbb L_0\oplus\mathbb L_1$ of free $A$-modules $\mathbb L_0:=A^n$ and $\mathbb L_1:=A^k$:
\begin{align}\label{eq:cicotangent}
0\leftarrow\mathbb L_0\overset{\partial}{\longleftarrow} \mathbb L_1 \leftarrow 0
\end{align}
The differential is given as follows. Let $e_i$ be the canonical basis for $\mathbb L_0=A^n$ and $\varepsilon_\mu$ be the canonical basis for $\mathbb L_1=A^k$. Then $\partial$ is the unique $A$-linear map such that  $\partial \varepsilon_\mu=\sum_{i=1}^n\partial f_\mu/\partial x_i\:e_i$. By sending  $e_i \mapsto \dR x_i$ and $\varepsilon_\mu\mapsto 0$ we obtain an $A$-linear map $\mathbb L\to\Omega_{A|\bs k}$.
The complex $\mathbb L$ is sometimes referred to as the \emph{conormal sequence} and is a special case of the cotangent complex (see Section \ref{sec:cotangent}).
Recall the following well-known fact.

\begin{lemma} If $\bs f$ is a reduced complete intersection it follows that $\mathbb L$ is exact in homological degree $1$, and hence the map $\mathbb L\to\Omega_{A|\bs k}$ is a quasiisomorphism.
\end{lemma}
\begin{proof} See for example \cite[Theorem 9.5]{Kunz}.
\end{proof}


We caution the reader that the symbols $e_i$ and $\epsilon_\mu$ will be  written later as $\dR x_i$ and $\dR y_\mu$, respectively, see Section \ref{sec:cotangent}. To avoid confusion, in this section, we stick to the notation $e_i$ and $\epsilon_\mu$.

\begin{theorem}\label{thm:ciLiealgebroid} If $\bs f$ is a complete intersection, there is a graded Lie bracket on $\mathbb L$ defined by
\begin{align} \label{eq:Koszulbracketci}
[e_i,e_j]:=\sum_{k=1}^n {\partial \Lambda _{ij}\over \partial x_k}e_k, \qquad [e_i,\varepsilon_\mu]:=\sum_{\nu=1}^k Z_{i\mu}^\nu \varepsilon_\nu,
\end{align}
making $(\mathbb L,\partial, [\:,\:])$ a dg Lie algebroid over $\Spec(A)$ with anchor $\rho: \mathbb L\to D_A$ given by $e_i\mapsto \{x_i,\:\}$ and $\varepsilon_\mu\mapsto 0$.
The coefficients in \eqref{eq:Koszulbracketci} are understood to be taken modulo $I$.
Moreover, the $A$-linear quasi-isomorphism $\mathbb L\to \Omega_{A|\bs k}$ is compatible with the brackets.
We have that the internal degree of $[\:,\:]$ coincides with the internal degree of $\{\:,\:\}$ while $\deg\partial=0$.
\end{theorem}
\begin{proof} We use the shorthand notations from the proof of Theorem \ref{thm:MC}. To verify the Leibniz rule
\begin{align*}
\partial[X,Y]=[\partial X,Y]+(-1)^{|X|}[X,\partial Y]
\end{align*}
we have to address two cases: the bracket on the lefthand side evaluated on $\mathbb L_0\times\mathbb L_1$ or on
$\mathbb L_1\times\mathbb L_1$. In the former case the Leibniz rule spells out as
\begin{align}\label{eq:LeibnizL0L1}
\partial\left[U^ie_i,\xi^\mu\varepsilon_\mu\right]\sim\left[U^ie_i,\partial(\xi^\mu\varepsilon_\mu)\right]
\end{align}
where $U^i,\xi^\mu\in S$.
The lefthand side of \eqref{eq:LeibnizL0L1} can be understood as
\begin{align*}
\partial\left[U^ie_i,\xi^\mu\varepsilon_\mu\right]=\partial\left( U^i\xi^\mu Z_{i\mu}^\nu \varepsilon_\nu+U^i\{x_i,\xi^\mu\}\varepsilon_\mu\right)= U^i\xi^\mu Z_{i\mu}^\nu {\partial f_\nu\over \partial x_k}e_k+U^i\Lambda_{ij}{\partial \xi^\mu\over \partial x_j}{\partial f_\mu\over \partial x_k}e_k.
\end{align*}
We have the freedom to add terms in $I$,
\begin{align*}
Z_{i\mu}^\nu {\partial f_\nu\over \partial x_k}\sim Z_{i\mu}^\nu {\partial f_\nu\over \partial x_k}+{\partial  Z_{i\mu}^\nu \over \partial x_k}f_\nu
={\partial \over \partial x_k}\left( \{x_i,f_\mu\}\right)={\partial \over \partial x_k}\left( \Lambda_{ij}{\partial f_\mu\over \partial x_j}\right)
={\partial \Lambda_{ij}\over \partial x_k}{\partial f_\mu\over \partial x_j}+\Lambda_{ij}{\partial^2 f_\mu\over \partial x_k\partial x_j}e_k.
\end{align*}
Hence
\begin{align*}
&\partial\left[U^ie_i,\xi^\mu\varepsilon_\mu\right]\sim U^i\xi^\mu{\partial \Lambda_{ij}\over \partial x_k}{\partial f_\mu\over \partial x_j}e_k+  U^i\xi^\mu \Lambda_{ij}{\partial^2 f_\mu\over \partial x_k\partial x_j} e_k+U^i\Lambda_{ij}{\partial \xi^\mu\over \partial x_j}{\partial f_\mu\over \partial x_k}e_k\\
&= U^i\xi^\mu{\partial \Lambda_{ij}\over \partial x_k}{\partial f_\mu\over \partial x_j}e_k+ U^i\xi^\mu \Lambda_{ij}{\partial \over \partial x_j}\left(\xi^\mu{\partial f_\mu\over \partial x_k}\right)
=U^i\xi^\mu{\partial \Lambda_{ij}\over \partial x_k}{\partial f_\mu\over \partial x_j}e_k + U^i\xi^\mu \left\{x_i,\xi^\mu{\partial f_\mu\over \partial x_k}\right\}\\
&=\left[U^ie_i,\xi^\mu{\partial f_\mu\over \partial x_j}e_j\right]=\left[U^ie_i,\partial(\xi^\mu\varepsilon_\mu)\right],
\end{align*}
establishing \eqref{eq:LeibnizL0L1}.

The other case boils down to
\begin{align}\label{eq:LeibnizL1L1}
\left[\partial(\xi^\mu\varepsilon_\mu),\eta^\mu\varepsilon_\nu\right]\sim\left[\xi^\mu\varepsilon_\mu,\partial(\eta^\nu\varepsilon_\nu)\right]
\end{align}
where $\xi^\mu, \eta^\nu\in S$. To show this we need an auxiliary result. Namely, we
claim that (cf. Equation \eqref{eq:Amunu})
\begin{align}\label{eq:alphamunu}
\alpha_{\mu\nu}:=\left( {\partial f_\mu\over \partial x_i}Z_{i\nu}^\lambda+ {\partial f_\nu\over \partial x_i}Z_{i\mu}^\lambda\right)y_\lambda=\mathcal A_{\mu\nu}^\lambda y_\lambda
\end{align}
is a Koszul $1$-cycle. This implies of course that $\alpha_{\mu\nu}$ is a Koszul $1$-boundary and as such $\sim 0$, meaning that
\begin{align}
{\partial f_\mu\over \partial x_i}Z_{i\nu}^\lambda\sim -{\partial f_\nu\over \partial x_i}Z_{i\mu}^\lambda.
\end{align}
The claim follows easily using relation \eqref{eq:ZLambdaidentity}:
\begin{align*}
\partial \alpha_{\mu\nu}=\left( {\partial f_\mu\over \partial x_i}Z_{i\nu}^\lambda+ {\partial f_\nu\over \partial x_i}Z_{i\mu}^\lambda\right)f_\lambda={\partial f_\mu\over \partial x_i}{\partial f_\nu\over \partial x_j}\Lambda_{ij}+{\partial f_\nu\over \partial x_i}{\partial f_\mu\over \partial x_j}\Lambda_{ij}=0.
\end{align*}
We remind the reader that terms of the form $\{f_\mu,\eta^\nu\}$ and $\{f_\nu,\xi^\mu\}$ are $\sim 0$, since $I$ is a Poisson ideal. Now we are ready to take a  look into
\begin{align*}
\left[\partial(\xi^\mu\varepsilon_\mu),\eta^\nu\varepsilon_\nu\right]
=\left[\xi^\mu {\partial f_\mu\over \partial x_i} e_i,\eta^\nu\varepsilon_\nu\right]
=\xi^\mu \underbrace{{\partial f_\mu\over \partial x_i}\{x_i,\eta^\nu\}}
_{=\{f_\mu,\eta^\nu\}\sim 0}\varepsilon_\nu
+\xi^\mu\eta^\nu {\partial f_\mu\over \partial x_i}Z_{i\nu}^\lambda\varepsilon_\lambda
\sim -\xi^\mu\eta^\nu {\partial f_\nu\over \partial x_i}Z_{i\mu}^\lambda\varepsilon_\lambda.
\end{align*}
Similarly, we find
\begin{align*}
\left[\xi^\mu\varepsilon_\mu,\partial(\eta^\nu\varepsilon_\nu)\right]
=\left[\xi^\mu \varepsilon_\mu,\eta^\mu{\partial f_\nu\over \partial x_i} e_i\right]
=-\eta^\nu {\partial f_\nu\over \partial x_i}\{x_i,\xi^\mu\}
-\xi^\mu\eta^\nu {\partial f_\nu\over \partial x_i}Z_{i\mu}^\lambda\varepsilon_\lambda
\sim -\xi^\mu\eta^\nu {\partial f_\nu\over \partial x_i}Z_{i\mu}^\lambda\varepsilon_\lambda,
\end{align*}
establishing \eqref{eq:LeibnizL1L1}.

In order to check the Jacobi identity, our strategy might be to show that the Jacobiator
\begin{align*}
\operatorname{Jac}:\mathbb L\times\mathbb L\times\mathbb L\to\mathbb L, \quad
\operatorname{Jac}(X,Y,Z)=(-1)^{|X||Z|}[X,[Y,Z]]+(-1)^{|X||Y|}[Y,[Z,X]]+(-1)^{|Y||Z|}[Z,[X,Y]]
\end{align*}
is $A$-trilinear and to check its vanishing on basis elements. There are two cases to consider: $\operatorname{Jac}$ on $\mathbb L_0\times\mathbb L_0\times\mathbb L_0$ and on $\mathbb L_0\times\mathbb L_0\times\mathbb L_1$. However, in the former case there is a more elegant way to check the vanishing of $\operatorname{Jac}(U^ie_i,V^je_j,W^ke_k)$. We remind ourselves that in $\mathbb L_0$, $e_i$ is nothing other than $\dR x_i$. Moreover, in $S$ every $1$-form is exact: given coefficients $U^i\in S$ we can write $U^i\dR x_i=\dR U$ for some function $U\in S$. There are similar functions $V,W\in S$. For our argument we can use the relation
$\left[\dR U,\dR V\right]=\dR\{U,V\}$.
Hence
\[\left[U^ie_i,\left[V^je_j,W^k e_k\right]\right]=\left[\dR U,\left[\dR V,\dR W\right]\right]=\left[\dR U,\dR\{V,W\}\right]=\dR\{U,\{V,W\}\}.\] Therefore the Jacobiator is $\dR$ applied to the Jacobi identity $\{U,\{V,W\}\}+\{V,\{W,U\}\}+\{W,\{U,V\}\}=0$.

The vanishing of $\operatorname{Jac}(e_i,e_j,\varepsilon_\mu)$ is closely related to Theorem \ref{thm:MC}. We leave it to the reader to verify that
$\operatorname{Jac}(e_i,e_j,\varepsilon_\mu)=0$ is a rephrasing of the Maurer-Cartan equation $\dP Z-[Z,Z]=0$.

We are left with the task to prove that $\operatorname{Jac}$ is $A$-trilinear on $\mathbb L_0\times\mathbb L_0\times\mathbb L_1$. To check whether $\operatorname{Jac}(U^ie_i,V^je_j,\xi^\mu\varepsilon_\mu)$ is $A$-linear in the first argument, we notice that $\left[U^ie_i,\left[V^je_j,\xi^\mu\varepsilon_\mu\right]\right]$ is $A$-linear in the first argument.
The claim follows from the calculation
\begin{align*}
\left[V^j e_j,
\left[
\xi^\mu\varepsilon_\mu,U^ie_i
\right]\right]
&=V^j\left[e_j,U^i(-\{x_i,\xi^\nu\}\varepsilon_\nu-\xi^\mu Z_{i\mu}^\nu)\varepsilon_\nu\right]\\
&=-\{x_j,U^i\}V^j\left(\{x_i,\xi^\nu\}\varepsilon_\nu+\xi^\mu Z_{i\mu}^\nu\right)\varepsilon_\nu+\mbox{ $U^i$-linear terms},\\
\left[\xi^\mu\varepsilon_\mu,
\left[
U^ie_i,V^j e_j
\right]\right]&=\left[\xi^\mu\varepsilon_\mu,-\{x_j,U^i\}V^je_j+\mbox{ $U^i$-linear terms}\right]\\
&=-\{x_j,U^i\}V^j\left(-\{x_i,\xi^\nu\}\varepsilon_\nu-\xi^\mu Z_{i\mu}^\nu\right)\varepsilon_\nu+\mbox{ $U^i$-linear terms}.
\end{align*}
By symmetry it follows that $\operatorname{Jac}(U^ie_i,V^je_j,\xi^\mu\varepsilon_\mu)$ is $A$-linear in the second argument. To make sure that $\operatorname{Jac}(e_i,e_j,\xi^\mu\varepsilon_\mu)$ is $A$-linear in the last argument we calculate
\begin{align*}
\left[e_i,
\left[
e_j,\xi^\mu\varepsilon_\mu
\right]\right]
&=\left[e_i,\xi^\mu Z_{j\mu}^\nu\varepsilon_\nu+
\left\{
x_j,\xi^\mu
\right\}\varepsilon_\mu\right]\\
&= \left\{
x_i,\xi^\mu
\right\}Z_{j\mu}^\nu\varepsilon_\nu+\left\{
x_i,\{x_j,\xi^\mu
\}\right\}\varepsilon_\mu+
\left\{
x_j,\xi^\mu
\right\}Z_{i\mu}^\nu\varepsilon_\nu
+\mbox{ $\xi^\mu$-linear terms},\\
\left[e_j,
\left[
\xi^\mu\varepsilon_\mu,e_i
\right]\right]
&=-\left[e_j,\xi^\mu Z_{i\mu}^\nu\varepsilon_\nu+
\left\{
x_i,\xi^\mu
\right\}\varepsilon_\mu\right]\\
&= -\left\{
x_j,\xi^\mu
\right\}Z_{i\mu}^\nu\varepsilon_\nu-\left\{
x_j,\{x_i,\xi^\mu
\}\right\}\varepsilon_\mu-
\left\{
x_i,\xi^\mu
\right\}Z_{j\mu}^\nu\varepsilon_\nu
+\mbox{ $\xi^\mu$-linear terms},\\
\left[\xi^\mu\varepsilon_\mu,
\left[
e_i,e_j
\right]\right]&=-\{\{x_i,x_j\},\xi^\mu\}\varepsilon_\mu+\mbox{ $\xi^\mu$-linear terms}.
\end{align*}
Adding things up, the claim follows from the Jacobi identity for $\{\:,\:\}$.
\end{proof}

The above theorem actually reflects of a more fundamental structure. In fact, we can make the following Ansatz for a differential graded Poisson algebra structure on the Koszul complex:
\begin{align}\label{eq:dgPoisson}
\{x_i,x_j\}=\Lambda_{ij},\quad \{x_i,y_\mu\}:=\sum_\nu Z_{i\mu}^\nu y_\nu,\quad \{y_\mu,y_\nu\}=-\beta_{\mu\nu},
\end{align}
where the $\beta_{\mu\nu}=\beta_{\nu\mu}\in \operatorname K_2(S,\bs f)$ are chosen such that $\partial\beta_{\mu\nu}=\alpha_{\mu\nu}=\sum_\lambda\mathcal A_{\mu\nu}^\lambda y_\lambda$ (cf. Equation \eqref{eq:alphamunu}). The problem with this definition, however, is that the Jacobi identity  is typically violated. More precisely, the Jacobiators $\operatorname{Jac}(x_i,x_j,y_\mu)$ are zero precisely when $\dP Z-[Z,Z]\in \mathrm{C}^1_{\operatorname{Poiss}}(S)\otimes_S\mathfrak{gl}_k(S)$ vanishes. For the Jacobiators $\operatorname{Jac}(x_i,y_\mu,y_\nu)$ and $\operatorname{Jac}(y_\mu,y_\nu,y_\lambda)$ one can work out concrete formulas that are already a bit messy. The only complete intersections known to us with $\dP Z-[Z,Z]$ and all the $\beta_{\mu\nu}$ vanishing are hypersurfaces. The vanishing of $\beta_{\mu\nu}$ is due to the fact that the Koszul complex of a hypersurface is concentrated in homological degree zero and one, while Corollary \ref{cor:principalideal} tells us that $\dP Z-[Z,Z]=\dP Z=0$.

\begin{theorem}\label{thm:principaldgPoisson} Let $I=(f)$ be a principal Poisson ideal in $S$ with $f\ne 0\in S$. Then the Koszul complex $(\operatorname K_\bullet(S,f)=S[y],\partial )$ is a dg Poisson algebra. The dg Poisson structure  characterized by $\partial:y\mapsto f$, $\{x_i,x_j\}=\Lambda_{ij}$ and $\{x_i,y\}=Z_i$. Here the $Z_i\in S$ are uniquely determined up to $I$ by the requirement  $\{x_i,f\}=Z_if$.
\end{theorem}

In general, when $\bs f=(f_1,\dots, f_k)$ is a complete intersection, the Koszul complex $\operatorname K_\bullet(S,\bs f)$ merely carries the structure of a $P_\infty$-algebra. One can find this statement and a sketch of a proof in a paper of Fresse \cite{FresseCI}.
We will show this in Section \ref{sec:homotopystuff}, dropping the hypothesis of $\bs f$ being a complete intersection.


\section{Poisson connection on the Koszul complex}\label{sec:connection}

To put the calculations of the previous section into context, we offer here an explanation of the tensor $\dP Z-[Z,Z]$ (defined in  \eqref{eq:MC}) as $-1$ times the curvature of a Poisson connection (see, e.g., \cite{Bursztyn}) on the Koszul complex. Let us recall that a \emph{Poisson module} over the Poisson algebra $(S,\{\:,\:\})$ is an $S$-module $M$ such that $S\oplus M$
\begin{enumerate}
\item is understood as the unique  $\bs k$-algebra extending the $\bs k$-algebra structure on $S$ and the $S$-(bi)module structure on $M$ with the property that $M\cdot M=0$,
\item carries a Poisson bracket $\{\:,\:\}$ that extends the bracket on $S$ such that $\{M,M\}=0$.
\end{enumerate}
This translates into the condition that $\{\:,\:\}:S\times M\to M$ satisfies
\begin{align}
\{\{f_1,f_2\},m\}&=-\{f_2,\{f_1,m\}\}+\{f_1,\{f_2,m\}\},\\
\{f_1,f_2 m\}&=\{f_1,f_2\} m+f_2\{f_1, m\},\\
\{f_1f_2, m\}&=f_2\{f_1,m\}+f_1\{f_2, m\}
\end{align}
for all $f_1,f_2\in S$ and $m\in M$. Conversely, a bracket $\{\:,\:\}:S\times M\to M$ with the above properties extends uniquely to a Poisson bracket on $S\oplus M$ that defines a Poisson module. Note that a free $S$-module is in an obvious way a Poisson module by taking brackets coordinate-wise. In particular, this applies to the $S$-module underlying the Koszul complex $\operatorname K_\bullet(S,\bs f)$.

By a \emph{Poisson connection} on the Poisson module $M$ we mean a $\bs k$-linear map $\nabla:\Omega_{S|\bs k}\otimes_{\bs k} M\to M$, $\alpha\otimes m\mapsto \nabla_\alpha m$, that satisfies
\begin{align}
\nabla_{f\alpha} m&=f\nabla_{\alpha} m, \\
\nabla_{\alpha} (f m)&=f\nabla_{\alpha} m+\langle \alpha,\{f,\:\}\rangle m
\end{align}
for all $f\in S$, $m\in M$, and $\alpha\in\Omega_{S|\bs k}$. Here $\langle\:,\:\rangle$ is used for the pairing of K\"{a}hler forms with derivations of $S$. The connection is uniquely determined by its values $\nabla_{\dR x_i}$ on the generators $\dR x_i$. The \emph{curvature tensor} of the Poisson connection is defined to be
\begin{align}
\mathcal R(\alpha,\beta):=[\nabla_{\alpha},\nabla_{\beta}]-\nabla_{[\alpha,\beta]},
\end{align}
for $\alpha,\beta\in\Omega_{S|\bs k}$. Here we use the Koszul bracket $[\alpha,\beta]$ between $\alpha$ and $\beta$.
It turns out that $\alpha\wedge\beta\mapsto\mathcal  R(\alpha,\beta)$ is an $S$-linear map $ \Omega_{S|\bs k}\wedge_S\Omega_{S|\bs k}\to \operatorname{End}_S(M)$. The Poisson connection $\nabla$ is called \emph{flat} if its curvature tensor vanishes identically.

\begin{proposition} For each $t\in\bs k$ the expression
$\nabla_{\dR x_i}^t:=\{x_i,\:\}+t\sum_{\mu,\nu} Z_{i\mu}^\nu y_\nu\partial/\partial{y_\mu}$ defines a Poisson connection $\nabla^t$ on the $S$-module underlying the Koszul complex $(\operatorname K_\bullet(S,\bs f),\partial)$. It has the property that $\nabla_\alpha$ is a derivation of the product for all $\alpha\in \Omega_{S|\bs k}$.
When we specialize to $t=1$, $\nabla^{1}$ is the unique Poisson connection having this property such that $[\partial,\nabla^{1}_\alpha]=0$ for all $\alpha\in \Omega_{S|\bs k}$. The curvature tensor of $\nabla^{t}$ equals $t\dP Z+t^2[Z,Z]$. This means in particular that  the curvature tensor of $\nabla^{-1}$ is $-\dP Z+[Z,Z]$.
\end{proposition}

\begin{proof} The most general Poisson connection acting as a derivation takes the form
\begin{align*}
\nabla_{\dR x_i}=\{x_i,\:\}+\sum_{\mu,\nu} \Gamma_{i\mu}^\nu y_\nu\partial/\partial{y_\mu}
\end{align*}
 for some coefficients $\Gamma_{i\mu}^\nu
\in S$ (the so-called Christoffel symbols of the connection). It is easy to verify that the commutator $[\nabla_{\dR x_i},\partial]$ vanishes precisely when $Z_{i\mu}^\nu=\Gamma_{i\mu}^\nu$ for all $\mu,\nu=1,\dots,k$. In the formula for the curvature we implicitly use the canonical identification of $\mathfrak{gl}_k(S)$ with linear vector fields in the odd coordinates $y_\mu$ sending the $k\times k$-matrix $(a_\mu^\nu)_{\mu,\nu=\{1,\dots,k\}}$ to $\sum_{\mu,\nu}a_{\mu}^{\nu}y_\nu\partial/\partial{y_\mu}$. For the curvature tensor of $\nabla^t$ acting on $m\in\operatorname K_\bullet(S,\bs f),\partial)$ we find
\begin{align*}
&[\nabla^t_{\dR x_i},\nabla^t_{\dR x_j}]m-\nabla^t_{[\dR x_i,\dR x_j]}m=[\nabla^t_{\dR x_i},\nabla^t_{\dR x_j}]m-\nabla^t_{\dR \Lambda_{ij}}m\\
&=\left[\{x_i,\:\}+t\sum_{\mu,\nu} Z_{i\mu}^\nu y_\nu{\partial\over\partial{y_\mu}},\{x_j,\:\}+t\sum_{\lambda,\rho} Z_{j\lambda}^\rho y_\rho{\partial\over\partial{y_\lambda}}\right]m\\
&= \{x_i,\{x_j,m\}\}-\{x_j,\{x_i,m\}\}+t\left[\{x_i,\:\},\sum_{\lambda,\rho} Z_{j\lambda}^\rho y_\rho{\partial\over\partial{y_\lambda}}\right]m-t\left[\{x_j,\:\},\sum_{\mu,\nu} Z_{i\mu}^\nu y_\nu{\partial\over\partial{y_\mu}}\right]m\\
&\qquad +\sum_{\mu,\nu}t^2[Z_i,Z_j]_{\mu}^{\nu} y_\nu{\partial m\over\partial{y_\mu}}-\{\{x_i,x_j\},m\}-t\sum_{k,\mu,\nu}{\partial \Lambda_{ij}\over \partial x_k}Z_{k\mu}^\nu y_\nu{\partial m\over\partial{y_\mu}}=\left(t\dP Z+t^2[Z,Z]\right)m. \qedhere
\end{align*}
\end{proof}


\section{The resolvent and the cotangent complex}\label{sec:cotangent}

In this section we recall material concerning the resolvent and the cotangent complex. We essentially follow \cite{Manetti, AvramovInfFree} but adapt the notation to our needs.

By a graded set we mean a countable set $\mathcal I$ with a function $\phi:\mathcal I\to \Z_{>0}$ such that for each $m\ge 0$ the cardinality of $\mathcal I_m:=\phi^{-1}(m)$ is finite. Accordingly $\mathcal I$ decomposes as a disjoint union of finite sets $\mathcal I=\sqcup_{m>0}\mathcal I_m$. To each $i\in \mathcal I_m$ we attach a variable $x_i^{(m)}$ whose parity coincides with the parity of $m$. This means that in the graded polynomial ring over the commutative $\bs k$-algebra $S$
\begin{align*}
S[\bs x]:=S\mleft[x_i^{(m)}\;\middle |\;m \ge 1,i\in \mathcal I_m\mright]
\end{align*}
we have relation $x_i^{(m)}x_j^{(n)}=(-1)^{mn}x_j^{(n)}x_i^{(m)}$.
Let us assign the cohomological degree $|x_i^{(m)}|:=-m$ to the variables $x_i^{(m)}$. By considering only variables up to level $r\ge 0$ we also have the graded polynomial ring in finitely many variables
\begin{align*}
S[\bs x_{\le r}]:=S\mleft[x_i^{(m)}\;\middle |\;1\le m \le r,i\in \mathcal I_m\mright],
\end{align*}
with the convention that $S=S[\bs x_{\le 0}]$.

By a dg $S$-algebra mean a cochain complex $(R,\partial)$ who is at the same time a supercommutative algebra such that $\partial$ is a derivation of the product. This means that for $a,b\in R$ of cohomological degree $|a|$ and $|b|$ we have $ab=(-1)^{|a||b|}ba$ and $\partial(ab)=(\partial a)b+(-1)^{|a|}a\partial b$.

\begin{definition} A dg $S$-algebra $(R,\partial)$ is called \emph{semifree} if the following two conditions hold true.
\begin{enumerate}
\item As an $S$-algebra $R$ is a graded polynomial algebra $S[\bs x]$ over the graded set $\phi:\mathcal I\to \Z_{>0}$.
\item For each $m>0$ and $i\in \mathcal I_m$ we have $\partial( x_i^{(m)})\in S[\bs x_{\le m-1}]$.
\end{enumerate}
Clearly for each $r\ge 0$ $S[\bs x_{\le r}]$ forms a semifree dg subalgebra which we denote by $(R_{\le r},\partial_{\le r})$. We will refer to it as the \emph{approximation of $(R,\partial)$ of level $\le r$}. The canonical algebra map $R\to S$ that sends each variable to zero is denoted by $\kappa$.
\end{definition}

We follow the habit of physics and interpret $\Spec(R_{\le r})$ as an affine supervariety with the homological vector field $\partial$. (If one takes the grading into account people refer to the setting as an NQ-manifold, see e.g. \cite{Strobl}.) $\Spec(R)$ however is merely a projective limit of affine supervarieties and one has to pay attention to issues related to the infinitely many degrees of freedom.
Denoting the image of $x_i^{(m)}$ under $\partial$ by $F_i(\bs x_{\le {m-1}})$ we will find it convenient to write
\begin{align}\label{eq:KoszulTate}
\partial_{\le r}=\sum_{m=1}^r\sum_{j\in \mathcal I_m}
 F_{j}(\bs x_{\le m-1}){\partial\over\partial x_{j}^{(m)}}\quad\mbox{and }\quad \partial=\sum_{m=1}^{\infty}\sum_{j\in  \mathcal I_m} F_{j}(\bs x_{\le m-1}){\partial\over \partial x_{j}^{(m)}}.
\end{align}
The fact that $\partial^2=0$ translates into the condition that for each $n\ge 1$ and $k\in \mathcal I_n$ we have
\begin{align*}
\sum_{m=1}^{n-1}\sum_{j\in \mathcal  I_m} F_j(\bs x_{\le m-1}){\partial F_k(\bs x_{\le n-1})\over\partial x_{j}^{(m)}}=0.
\end{align*}

\begin{definition}
Let $I$ be an ideal in $S$. We say that the semifree $S$-algebra $(R,\partial)$ is a \emph{resolvent} of $S/I$ if the composition of the algebra morphisms $\kappa: R\to S$ and $S\to S/I$ is a quasi-isomorphism, i.e., induces an isomorphism in cohomology. (Here $S/I$ is seen a cochain complex concentrated in degree $0$).
\end{definition}

Let us specialize to the situation when $S=\bs k[x_1,\dots,x_n]$. In this case $A$ admits a resolvent (cf., e.g., \cite[Prop. 2.1.10.]{AvramovInfFree}). It is constructed as follows. Let $f_1,\dots, f_k$ be generators for the ideal $I$. Put $\mathcal I_1:=\{1,\dots ,k\}$ and define $\partial_{\le 1}$ by $x_\mu^{(1)}\mapsto f_\mu$. Notice that $R_{\le 1}$ is nothing but the Koszul complex seen as a cochain complex. It has finite dimensional cohomology in degree $-2$. For each cohomology class pick a representative $F_j(\bs x_{\le 1})\in R_{\le 1}^{-2}$ and define  $\partial_{\le 2}$ by $x_j^{(2)}\mapsto  F_j(\bs x_{\le 1})$. The process continues by induction. That is, for each cohomology class in $H^{-k}R_{\le k-1}$ we a pick representative $F_j(\bs x_{\le k-1})\in R_{\le k-1}^{-k}$ and put $\partial_{\le k}:x_j^{(k)}\mapsto  F_j(\bs x_{\le k-1})$. The resulting resolvent $(R,\partial)$ is unique up to $S$-linear homotopy.

In the situation when $I=(f_1,\dots,f_k)$ is  a homogeneous ideal in $S=\bs k[x_1,\dots,x_n]$ with $\deg(x_i)\ge 1$ for $i=1,\dots,n$ such that $I\subseteq\mathfrak m=(x_1,\dots, x_n)$, we assign to the variable $x_j^{(m)}$'s internal degrees such that $\deg({\partial})=0$. In this way the resolvent $(R,\partial)$ of $S/I$ becomes a bigraded dg algebra.

\begin{definition}
With the above assumptions the resolvent $(R,\partial)$ of $A=S/I$ is called a \emph{minimal model} if
\begin{align*}
\partial(x_j^{(1)})\in \mathfrak m \mbox{ for all }j=1,2,\dots,k, \quad \partial(x_j^{(r)})\in \mathfrak n^2 \mbox{ for all }r\ge 2,\:j\in \mathcal I_r,
\end{align*}
where $\mathfrak n$ is the kernel of the composition of $\kappa: R\to S$ and $S\to\bs k$.
\end{definition}

By \cite[Subsection 7.2]{AvramovInfFree} (see also \cite[Section 4.3]{ACI}) a minimal model exists and is unique up to isomorphism of bigraded dg algebras. All examples of resolvents in this paper are minimal models. Algorithms for computing minimal models have been incorporated into \emph{Macaulay2} \cite{M2} by Frank Moore as the package \emph{dgalgebras} \cite{dgAlgsM2}.

\begin{definition}
Let $(R,\partial)$ be a resolvent of the $S$-algebra $A=S/I$. The \emph{cotangent complex} of $A$ over $\bs k$ is $\mathbb L_{A|\bs k}=\Omega_{R|\bs k}\otimes_R  A$, where $\Omega_{R|\bs k}$ are the K\"{a}hler differentials of $R$ and $\otimes_R$ is the tensor product in the category of complexes of $R$-modules.
\end{definition}

If the ideal $I$ is homogeneous the cotangent complex $\mathbb L_{A|\bs k}$ is bigraded in the obvious way. If $I$ is a complete intersection $\mathbb L_{A|\bs k}$ is actually  isomorphic the complex of Equation \eqref{eq:cicotangent}, where the isomorphism is given by  $\dR x_j\mapsto e_j$ and $\dR y_\mu\mapsto \varepsilon_\mu$ writing  $y_\mu$ for $x^{(1)}_j$.


\section{The $P_\infty$-algebra and the  $L_\infty$-algebroid}\label{sec:homotopystuff}

The aim of this section is to construct a $P_\infty$-algebra structure on the resolvent $(R,\partial)$ of a Poisson ideal and thereby prove Theorem \ref{thm:homotopyPoisson} and Corollary \ref{cor:homotopyLiealgebroid}. As the principal tool we need to recall the Schouten bracket on the multiderivations of the resolvent. Here we follow Cattaneo and Felder \cite{CF} who defined a Schouten bracket for graded supermanifolds. The adaptation from the setup of smooth supermanifolds to affine supervarieties is straightforward.
The only catch here is to make sure that their formula also makes sense in the case when there are infinitely many generators.

Let $V[m]=\oplus_k V[m]^k$  be the \emph{shift} of a graded vector space $V=\oplus_k V^k$ with components $V[m]^k=V^{m+k}$. The identity map gives rise to a map $\downarrow^m:=[m]:V\to V[m]$ of degree $-m$.
Let us fix $r\ge 0$ and consider, using the notation of Section \ref{sec:cotangent},
the graded symmetric algebra $\mathfrak h_{\le r}:=\operatorname S_{R_{\le r}}(\operatorname{Der}_{R_{\le r}}[-1])$ of the shifted module $\operatorname{Der}_{R_{\le r}}[-1]$ of derivations of the algebra $R_{\le r}$.  Accordingly the coordinate derivations of Section \ref{sec:cotangent} give rise to generators
 \begin{align*}
\xi^i_{(m)}:={\partial\over\partial x_i^{(m)}}[-1]\in\operatorname{Der}_{R_{\le r}}[-1]\subset \mathfrak h_{\le r}, \quad r\ge m\ge 0, i\in \mathcal I_m
\end{align*}
of cohomological degree $m+1$. We can interpret $\mathfrak h_{\le r}$ as a graded polynomial ring
\begin{align*}
\mathfrak h_{\le r}=\bs k\left[x_i^{(m)},\xi^i_{(m)}|r\ge m\ge 0, i\in \mathcal I_m\right].
\end{align*}
The cohomological degree $|\xi^{i_1}_{(m_1)}\cdots\xi^{i_\ell}_{(m_\ell)}|$ of $\xi^{i_1}_{(m_1)}\cdots\xi^{i_\ell}_{(m_\ell)}$ is evidently $m_1+\dots+m_\ell+\ell$.  We extend the differential $\partial_{\le r}$ to $\mathfrak h_{\le r}$ by declaring $\partial_{\le r}\xi^i_{(m)}=0$. Furthermore, we introduce on $\mathfrak h_{\le r}$ the multiplicative \emph{filtration degree} $\fd$ by putting
\begin{align*}
\fd(\xi^{i}_{(m)}):=|\xi^{i}_{(m)}|,\quad \fd(x_{i}^{(m)}):=0.
\end{align*}
Let $\mathcal F^p\mathfrak h_{\le r}$ be the $R_{\le r}$-span of $\{X\in \mathfrak h_{\le r}\mid \fd(X)\ge p\}$. The collection $(\mathcal F^p\mathfrak h_{\le r})_{p\ge 0}$ forms a descending Hausdorff filtration such that $\partial_{\le r}(\mathcal F^p \mathfrak h_{\le r})\subseteq\mathcal F^p \mathfrak h_{\le r}$, i.e, $\mathfrak h_{\le r}$ is a filtered complex. We use the convention that if $p<0$ then $\mathcal F^{p}\mathfrak h_{\le r}=\mathcal F^{0}\mathfrak h_{\le r}$.

From \cite{CF} we know that there is a unique bracket (of cohomological degree $-1$)
$\llbracket\:,\:\rrbracket$, the so-called \emph{Schouten bracket}, on $\mathfrak h_{\le r}$ such that
\begin{enumerate}
\item $XY=(-1)^{|X||Y|}YX$,
\item $\llbracket X,Y\rrbracket=-(-1)^{(|X|-1)(|Y|-1)}\llbracket Y,X\rrbracket$,
\item $\llbracket X,YZ\rrbracket=\llbracket X,Y\rrbracket Z+(-1)^{|Y|(|X|-1)}Y\llbracket X,Z\rrbracket $,
\item $\llbracket X,\llbracket Y,Z\rrbracket\rrbracket=\llbracket \llbracket X,Y\rrbracket,Z\rrbracket+(-1)^{(|X|-1)(|Y|-1)}\llbracket Y,\llbracket X,Z\rrbracket \rrbracket$.
\end{enumerate}
It makes $\mathfrak h_{\le r}$ a \emph{Gerstenhaber algebra} (confer, e.g., \cite{PingXu}) which entails in particular that $\mathfrak h_{\le r}[1]$ is a Lie superalgebra. A more convenient way to work with $\llbracket\:,\:\rrbracket$ is provided by the following formula from \cite{CF}
\begin{align}\label{eq:CFformula}
\llbracket X,Y\rrbracket=\sum_{m= 0}^r\sum_{i\in \mathcal I_m}X{\overleftarrow\partial \over \partial \xi^i_{(m)}}{\overrightarrow\partial \over \partial x_i^{(m)}}Y-X{\overleftarrow\partial \over \partial x_i^{(m)} }{\overrightarrow\partial \over \partial \xi^i_{(m)}}Y.
\end{align}
Here the arrows indicate in which direction the coordinate derivation in question acts. The collection $(\mathfrak h_{\le r})_{r\ge 0}$ forms a directed system of  Gerstenhaber algebras. The Schouten bracket on the direct limit $\mathfrak h:=\cup_{r\ge 0}\mathfrak h_{\le r}$ is given by the formula
\begin{align}\label{eq:CFformulainf}
\llbracket X,Y\rrbracket=\sum_{m= 0}^\infty\sum_{i\in \mathcal I_m}X{\overleftarrow\partial \over \partial \xi^i_{(m)}}{\overrightarrow\partial \over \partial x_i^{(m)}}Y-X{\overleftarrow\partial \over \partial x_i^{(m)} }{\overrightarrow\partial \over \partial \xi^i_{(m)}}Y.
\end{align}

\begin{proposition}\label{prop:Fadic}  The canonical isomorphism $\left(\operatorname{Der}_{R_{\le r}} R_{\le r}\right)[-1]\cong\Hom_{R_{\le r}}(\Omega_{R_{\le r}|\bs k}[1],R_{\le r})$ extents to an injective morphism of $R$-modules $\mathfrak h=\cup_{r\ge 0}\mathfrak h_{\le r}\to \mathfrak g:=\operatorname{Sym}_R(\Omega_{R|\bs k}[1],R)$.
The $\bs k$-vector space $\mathfrak g$ is the completion of $\mathfrak h$ in the $\mathcal F$-adic topology.
There is a unique structure of a Gerstenhaber algebra on $\mathfrak g$
such that $\mathfrak h\to \mathfrak g$ is a morphism of Gerstenhaber algebras.  Likewise, the $\mathcal F$-adic completion of $\mathfrak h_{\le r}$ is $\mathfrak g_{\le r}:=\operatorname{Sym}_{R_{\le r}}(\Omega_{R_{\le r}|\bs k}[1],R_{\le r})$ and the collection $(\mathfrak g_{\le r})_{r\ge 0}$ forms a directed system of $\mathcal F$-adically complete Gerstenhaber algebras.
\end{proposition}
\begin{proof}
Let us convince ourselves that $\mathfrak g=\operatorname{Sym}_R(\Omega_{R|\bs k}[1],R)$ and the completion of $\cup_{r\ge 0}\mathfrak h_{\le r}$ are isomorphic as graded $\bs k$- vector spaces. Recall that $\Omega_{R|\bs k}[1]$ is a free $R$-module generated by
\begin{align}\label{eq:OmegaGenerators}
\dR x_j^{(r)}[1]\quad r\ge 0,\:j\in\mathcal I_r
\end{align}
of cohomological degree $-r$.
For the sake of notational simplicity we will use the graded set $\mathcal I':=I\sqcup\{1,2,\dots, n\}$ with $\phi':\mathcal I'\to \Z_{\ge 0}$ where $\phi'_{|\mathcal I}=\phi$ and $\phi'_{|\{1,2,\dots, n\}}=0$. For the generators in Equation \eqref{eq:OmegaGenerators} we write $\dR X_j$, where $j\in \mathcal I'$.
An element $\Psi\in \mathfrak g$ of cohomological degree $\ell$ is nothing other than an assignment
\begin{align}\label{eq:Gerstiso}
(\dR X_{j_1},\dR X_{j_2},\dots, \dR X_{j_l}) \mapsto \psi_{j_1j_2\dots j_l}\in R,\quad \mbox{ for }j_1,j_2\dots, j_l\in \mathcal I'
\end{align}
with $\psi_{j_1j_2\dots j_l}=\epsilon(\sigma,\bs j)\psi_{j_{\sigma(1)}j_{\sigma(2)}\dots j_{\sigma (l)}}$ for any $\sigma \in \mathrm S_l$. Here $\bs j=(\phi'(j_1),\dots,\phi'(j_l))\in \Z^l$ and $\epsilon(\sigma,\bs j)$ is the Koszul sign of $\sigma$. The cohomological degree of $\psi_{j_1j_2\dots j_l}$ is $|\psi_{j_1j_2\dots j_l}|=\ell+\sum_{m=1}^l|X_{j_m}|=\ell-\sum_{m=1}^l\phi'(j_m)$. The corresponding element in the completion of $\cup_{r\ge 0}\mathfrak h_{\le r}$ is
\begin{align}\label{eq:Fseries}
\psi:=\sum_{l=0}^\infty \sum_{j_1,j_2,\dots ,j_l\in \mathcal I'} \psi_{j_1j_2\dots j_l}\underbrace{\xi^{j_1}_{(\phi'(j_1))}\xi^{j_2}_{(\phi'(j_2))}\cdots \xi^{j_l}_{(\phi'(j_l))}}_{\in \mathcal F^{l+\sum_{m=1}^l\phi'(j_m)}(\cup_{r\ge 0}\mathfrak h_{\le r})}.
\end{align}
Conversely, the Taylor coefficients $\psi_{j_1j_2\dots j_l}$ of any such series of cohomological degree $\ell$ give rise to an element in $\mathfrak g=\operatorname{Sym}_R(\Omega_{R|\bs k}[1],R)$ by Equation \eqref{eq:Gerstiso}.

This isomorphism is actually an isomorphism of graded commutative algebras. To sketch an argument showing this we use the canonical isomorphism $\operatorname{Sym}_R(\Omega_{R|\bs k}[1],R)\cong \operatorname{Hom}_R(\mathrm S_R(\Omega_{R|\bs k}[1]),R)$. Now $H:=\mathrm S_R(\Omega_{R|\bs k}[1])$ is a graded Hopf algebra over $R$ with the obvious multiplication $\mu:H\otimes_R H\to H$ and   the unique comultiplication $\Delta:H\to H\otimes_R H$ such that $\Omega_{R|\bs k}[1]\subset H$ is primitive. The product of $\Psi$ and $\Phi$ in $\mathfrak g$ is given by convolution $\mu\circ\Psi\otimes\Phi\circ \Delta$. It is well known that this restricts to the usual product for the subspace $\mathfrak g_{\le r}$ as the dual Hopf algebra of a graded symmetric algebra over a finitely generated free $R_{\le r}$-module is the graded symmetric Hopf algebra of the dual module. Clearly this product uniquely extends to a product on the $\mathcal F$-adic completion.

It remains to check that the Schouten bracket $\llbracket\psi,\varphi\rrbracket$ of two series $\psi$, $\varphi$ as in Equation \eqref{eq:Fseries} of cohomological degrees $|\psi|, |\varphi|\in\Z$  is well-defined. Let us inspect if
\begin{align*}
\sum_{m\ge 0}\sum_{i\in \mathcal I_m}\psi{\overleftarrow\partial \over \partial \xi^i_{(m)}}{\overrightarrow\partial \over \partial x_i^{(m)}}\varphi
\end{align*}
is well defined. This expression can be spelled out as
\begin{align*}
\sum_{l,m\ge 0}\sum_{j_1,\dots,j_l,k_1,\dots,k_m\in \mathcal I'}\sum_{u=1}^l\pm\psi_{j_1\dots j_l}{\partial\varphi_{k_1\dots k_m}\over \partial X_{j_u}}\xi^{j_1}_{(\phi'(j_1))}\cdots\widehat{\xi^{j_u}_{(\phi'(j_u))}}\cdots \xi^{j_l}_{(\phi'(j_l))}\xi^{k_1}_{(\phi'(k_1))}\cdots \xi^{k_m}_{(\phi'(k_m))},
\end{align*}
where we took the liberty not to specify the signs and used $\:\widehat{\:}\:$ to indicate omission. The issue is now that in principle in this sum the same monomial in the $\xi$'s can occur for infinitely many $j_u$. This is not a problem, however, since $|\varphi_{k_1\dots k_m}|=|\varphi|-\sum_v \phi'(k_v)$. In fact,
 $\varphi_{k_1\dots k_m}$ cannot depend on $X_{j_u}$ if $\phi'(j_u)$ exceeds $|\varphi_{k_1\dots k_m}|$. The second term in \eqref{eq:CFformula} can be taken care of in an analogous manner.
\end{proof}

By a \emph{multiderivation} $X$ we mean an element of $\mathfrak g$. The degree of $X[1]\in\mathfrak g[1]$ is referred to a the \emph{total degree} $\overline X:=|X|-1$.

Next we need to involve the differentials. The differential $\partial_{\le r}$ (see Equation \eqref{eq:KoszulTate}) gives rise to an element
\begin{align}
\pi_0^{\le r}:=\sum_{m=1}^r\sum_{j\in  \mathcal I_m} F_{j}(\bs x_{\le m-1})\xi^{j}_{(m)}\in \mathfrak h_{\le r}.
\end{align}
Likewise, we define an element of $\mathfrak g$ analogous to the differential $\partial$:
\begin{align}
\pi_0:=\sum_{m=1}^\infty\sum_{j\in \mathcal I_m}
 F_{j}(\bs x_{\le m-1})\xi^{j}_{(m)}\in\mathfrak g,
\end{align}
noting that the infinite sum above converges in the $\mathcal F$-adic topology.

With this we can formulate the following lemma whose proof is straightforward and left to the reader.

\begin{lemma}\label{lem:filtration} Let $p,q,r,s$ be integers $\ge 0$. Then
\begin{enumerate}
\item for each $X\in \mathcal F^p\mathfrak h_{\le r}$ we have $\llbracket\pi_0^{\le r},X\rrbracket  \in\partial_{\le r}X+\mathcal F^{p+1}\mathfrak h_{\le r}$, and
\item $\left\llbracket\mathcal F^p\mathfrak h_{\le r},\mathcal F^q\mathfrak h_{\le s}\right\rrbracket\subseteq \mathcal F^{p+q-1-\min(r,s)}\mathfrak h_{\le \max(r,s)}$.
\end{enumerate}
\end{lemma}

Our objective is to find a Maurer-Cartan element $\pi$, i.e., a multiderivation $\pi$ with total degree $\overline{\pi}=1$ that satisfies $\llbracket\pi,\pi\rrbracket  =0$, such that for all $p$ and $X\in\mathcal F^p\mathfrak g$
\begin{align}
\llbracket\pi,X\rrbracket\in\partial X+\dP X+\mathcal F^{p+2}\mathfrak g.
\end{align}
The recipe used is what is sometimes referred to as \emph{homological perturbation theory}. It has been employed, for example, in Fedosov quantization \cite{Fedosov} and in the construction of the BFV-charge \cite{Stasheff}.

For better readability we introduce a special notation in low cohomological degrees. We write $y_\mu$ instead  of $x_i^{(1)}$ and $\mu$ for $i\in \mathcal I_1$, $z_\alpha$ instead  of $x_i^{(2)}$ and $\alpha$ for $i\in \mathcal I_2$ and $w_t$ for $x_i^{(2)}$ using index $t$ for $i\in \mathcal I_3$. Similarly we use $\xi^i$ for $\xi^i_{(0)}$, $\eta^\mu$ for $\xi^i_{(1)}$, $\zeta^\alpha$ for $\xi^i_{(2)}$, and  $\theta^s$ instead of  $\xi^i_{(2)}$. The  first three coefficients of $\partial_{\le r}$ are denoted by $f_\mu$, $g_\alpha$ and $h_t$, i.e.,
\begin{align}
&\partial_{\le 1}=\sum_\mu f_\mu \partial/\partial y_\mu, \quad \partial_{\le 2}=\sum_\mu f_\mu \partial/\partial y_\mu+\sum_\alpha g_\alpha \partial/\partial z_\alpha,\\
\nonumber &\partial_{\le 3}=\sum_\mu f_\mu \partial/\partial y_\mu+\sum_\alpha g_\alpha \partial/\partial z_\alpha+\sum_t h_t\partial/\partial w_t.
\end{align}
For convenience of the reader we record the relevant degrees in a table.
\begin{align*}
\begin{tabular}{c||c|c|c|c|c|c|c|c}
                & $x_i$ &   $y_\mu$ & $z_\alpha$  &  $w_t$ & $\xi^i$ & $\eta^\mu$ & $\zeta^\alpha$ & $\theta^t$  \\\hline\hline
$|\:|$          & $0$ & $-1$  & $-2$ & $-3$ &   $1$ &    $2$ &     $3$ &      $4$\\\hline
$^-$            &$-1$ & $-2$  & $-3$ & $-4$ &   $0$ &    $1$ &     $2$ &      $3$\\\hline
$\fd$           & $0$ & $0$  & $0$ & $0$    &   $1$ &    $2$ &     $3$ &      $4$
\end{tabular}
\end{align*}

Our first approximation of $\pi$ is $\pi_0^{\le 1}+\pi_1$, where
\begin{align*}
\pi_0^{\le 1}:=\sum_{\mu=1}^k f_\mu\eta^\mu\quad\mbox{and }\quad \pi_1:=\sum_{i,j=1}^n {1\over 2}\Lambda_{ij}\xi^i\xi^j.
\end{align*}
When evaluating the bracket
\begin{align*}
\llbracket\pi_0^{\le 1}+\pi_1,\pi_0^{\le 1}+\pi_1\rrbracket  =\llbracket\pi_0^{\le 1},\pi_0^{\le 1}\rrbracket  +2\llbracket\pi_0^{\le 1},\pi_1\rrbracket  +\llbracket\pi_1,\pi_1\rrbracket  =2\llbracket\pi_0^{\le 1},\pi_1\rrbracket  ,
\end{align*}
we notice that $\llbracket\pi_0^{\le 1},\pi_0^{\le 1}\rrbracket  $ and $\llbracket\pi_1,\pi_1\rrbracket  $ vanish. In fact, the vanishing of $\llbracket\pi_1,\pi_1\rrbracket  $ is equivalent to the Jacobi identity of the Poisson bracket on $S$.
On the other hand
\begin{align*}
\llbracket\pi_0^{\le 1}, \pi_1\rrbracket  &=\sum_{i,j,\mu}{\Lambda_{ij}\over 2}\left\llbracket f_\mu\eta^\mu,\xi^i\right\rrbracket\xi^j-{\Lambda_{ij}\over 2}\xi^i\left\llbracket f_\mu\eta^\mu,\xi^j\right\rrbracket
=-\sum_{i,j,\mu}{\Lambda_{ij}\over 2}{\partial f_\mu\over \partial x_i}\eta^\mu\xi^j+{\Lambda_{ij}\over 2}{\partial f_\mu\over \partial x_j}\xi^i\eta^\mu
\\
&=\sum_{i,j,\mu}\Lambda_{ij}{\partial f_\mu\over \partial x_j}\xi^i \eta^\mu=\sum_{i,\mu,\nu}Z_{i\mu}^\nu f_\nu\xi^i \eta^\mu
\end{align*}
is a coboundary for $\partial_{\le 1}$ in cohomological degree $-1$.
The idea is to compensate for this by  introducing $\pi_2:=-\sum_{i,\mu,\nu}Z_{i\mu}^\nu y_\nu\xi^i\eta^\mu=:-\{x_i,y_\mu\}_2\xi^i\eta^\mu$.

At the second step of the recursion we put
\begin{align*}
&\pi_0^{\le 2}:=\sum_{\mu}f_\mu\eta^\mu+\sum_\alpha g_\alpha \zeta^\alpha\in \mathfrak h_{\le 2},\\
&\pi^{\le 2}:=\pi_0^{\le 2}+\pi_1+\pi_2.
\end{align*}
Our task is to work out the individual terms in
\begin{align*}
&\llbracket\pi^{\le 2},\pi^{\le 2}\rrbracket=\overbrace{\llbracket\pi^{\le 2}_0,\pi^{\le 2}_0\rrbracket}^{=0}+\overbrace{\llbracket\pi_1,\pi_1\rrbracket}^{=0}+2\llbracket\pi^{\le 1}_0,\pi_1\rrbracket+ 2\llbracket g_\alpha\zeta^\alpha,\pi_1\rrbracket+2\llbracket\pi^{\le 2}_0,\pi_2\rrbracket+2\llbracket\pi_1,\pi_2\rrbracket+\llbracket\pi_2,\pi_2\rrbracket.
\end{align*}
The result of this calculation is
\begin{align*}
&2\llbracket\pi^{\le 1}_0,\pi_1\rrbracket+2\llbracket\pi^{\le 2}_0,\pi_2\rrbracket=
-\sum_{\mu,\nu,\lambda}\mathcal A_{\mu\nu}^\lambda y_\lambda\eta^\mu\eta^\nu+\sum_{i,\mu,\nu,\alpha}2\left({\partial g_\alpha\over \partial x_i}Z_{i\mu}^\nu y_\nu\eta^\mu \zeta^\alpha-{\partial g_\alpha\over \partial y_\mu}Z_{i\mu}^\nu y_\nu\xi^i\zeta^\alpha\right)\\
&2\llbracket g_\alpha\zeta^\alpha,\pi_1\rrbracket=-2\sum_{i,j,\alpha}\Lambda_{ij}{\partial g_\alpha\over \partial x_j}\xi^i\zeta^\alpha=-2\sum_{i,j,\alpha}\{x_i,g_\alpha\}\xi^i\zeta^\alpha,\\
&2\llbracket\pi_1,\pi_2\rrbracket=\sum_{i,j,\mu,\nu}\left(\{x_i,Z_{j\mu}^\nu\}-\{x_j,Z_{i\mu}^\nu\}-\sum_k
{\partial\Lambda_{ij}\over\partial x_k}Z_{k\mu}^\nu \right)y_\nu\xi^i\xi^j\eta^\mu\\
&\llbracket\pi_2,\pi_2\rrbracket=\sum_{i,j,\mu,\nu,\lambda,\rho}2Z_{i\mu}^\nu{\partial Z_{j\lambda}^\rho\over \partial x_i}y_\nu y_\rho\xi^j\eta^\mu\eta^\lambda -\sum_{i,j,\mu,\nu}(Z_{i\mu}^\nu Z_{j\nu}^\lambda-Z_{j\mu}^\nu Z_{i\nu}^\lambda)y_\lambda\xi^i\xi^j\eta^\mu.
\end{align*}
It turns out that at this step of the iteration we can ignore all terms in $\mathcal F^5\mathfrak h_{\le 2}$ and deduce
\begin{align*}
\llbracket\pi^{\le 2},\pi^{\le 2}\rrbracket+\mathcal F^5\mathfrak h_{\le 2}&=\sum_{i,j,\mu,\nu,}\left(\{x_i,Z_{j\mu}^\nu\}-\{x_j,Z_{i\mu}^\nu\}-\sum_k
{\partial\Lambda_{ij}\over\partial x_k}Z_{k\mu}^\nu \right)y_\nu\xi^i\xi^j\eta^\mu\\
&\quad+
\sum_{i,j,\mu,\nu,\lambda}(Z_{i\mu}^\nu Z_{j\nu}^\lambda-Z_{j\mu}^\nu Z_{i\nu}^\lambda)y_\lambda\xi^i\xi^j\eta^\mu+\sum_{\mu,\nu,\lambda}-\mathcal A_{\mu\nu}^\lambda y_\lambda\eta^\mu\eta^\nu\\
&\quad-2\sum_{i,\alpha}\left(\{x_i,g_\alpha\}+\sum_{\mu,\nu}{\partial g_\alpha\over \partial y_\mu}Z_{i\mu}^\nu y_\nu\right)\xi^i\zeta^\alpha+\mathcal F^5\mathfrak h_{\le 2}\\
&=\sum_{\mu,\nu}2(\dP Z-[Z,Z])_\mu^\nu y_\nu\eta^\mu-\sum_{\mu,\nu,\lambda}\mathcal A_{\mu\nu}^\lambda y_\lambda\eta^\mu\eta^\nu\\
&\quad-2\sum_{i,\alpha}\left(\{x_i,g_\alpha\}+\sum_{\mu,\nu}{\partial g_\alpha\over \partial y_\mu}Z_{i\mu}^\nu y_\nu\right)\xi^i\zeta^\alpha+\mathcal F^5\mathfrak h_{\le 2}.
\end{align*}
We already know that $\partial_{\le 1}$ applied to the first two expressions gives zero (cf. the proof of Theorem \ref{thm:MC} and Equation \eqref{eq:alphamunu}). Let us verify that the last term is also a 1-cocycle:
\begin{align*}
\partial_{\le 1}\{x_i,g_\alpha\}&=\partial_{\le 1}\left(\sum_j\Lambda_{ij}{\partial g_\alpha\over \partial x_j}\right)=\sum_{\mu,j}f_\mu{\partial \over \partial y_\mu}\left(\Lambda_{ij}{\partial g_\alpha\over \partial x_j}\right)=\sum_{\mu,j}\Lambda_{ij}f_\mu{\partial \over \partial x_j}\left({\partial g_\alpha\over \partial y_\mu}\right)\\
&=\sum_{\mu,j}\left(\Lambda_{ij}{\partial \over \partial x_j}\left(f_\mu{\partial g_\alpha\over \partial y_\mu}\right)-\Lambda_{ij}{\partial g_\alpha\over \partial y_\mu}{\partial f_\mu\over \partial x_j}\right)=-\sum_{\mu,\nu}Z_{i\mu}^\nu f_\nu{\partial g_\alpha\over \partial y_\mu}=-\partial_{\le 1}\left(\sum_{\mu,\nu}Z_{i\mu}^\nu y_\nu {\partial g_\alpha\over \partial y_\mu}\right),
\end{align*}
where we have used the fact that $\sum_{\mu}f_\mu \partial g_\alpha/\partial y_\mu=0$.
As the resolvent is exact in cohomological degree $-1$ we can solve the system
\begin{align*}
&-2\partial_{\le 2}\pi_3^{001}=\sum_{\mu,\nu}2(\dP Z-[Z,Z])_\mu^\nu y_\nu\eta^\mu,\\
&-2\partial_{\le 2}\pi_3^{11}=-\sum_{\mu,\nu,\lambda}\mathcal A_{\mu\nu}^\lambda y_\lambda\eta^\mu\eta^\nu,\\
&-2\partial_{\le 2}\pi_3^{02}=-2\sum_{i,\alpha}\left(\{x_i,g_\alpha\}+\sum_{\mu,\nu}{\partial g_\alpha\over \partial y_\mu}Z_{i\mu}^\nu y_\nu\right)\xi^i\zeta^\alpha
\end{align*}
and put $\pi_3=\pi_3^{001}+\pi_3^{11}+\pi_3^{02}$. For the coefficients we use the notation $\pi_3^{001}=\sum_{i,j,\mu}{1\over 2}\{x_i,x_j,y_\mu\}_3\xi^i\xi^j\eta^\mu$, $\pi_3^{11}=\sum_{\mu,\nu}{1\over 2}\{y_\mu,y_\nu\}_2\eta^\mu\eta^\nu$ and $\pi_3^{02}=-\sum_{i,\alpha}\{x_i,z_a\}_2\xi^i\zeta^a$.

The proceeding calculations can be made systematic by the following argument.

\begin{lemma} \label{lem:recursion}
For $m\ge 2$ we can recursively find $\pi_m\in \mathcal F^{m+1}\mathfrak h_{\le m-1}$
such that
\begin{align}\label{eq:recursion level m}
\left\llbracket\pi_0^{\le m-1}+\sum_{i=1}^{m-1}\pi_i,\pi_0^{\le m-1}+\sum_{i=1}^{m-1}\pi_i\right\rrbracket+\mathcal F^{m+2}\mathfrak h_{\le m-1}=-2\partial_{\le m-1}\pi_m+\mathcal F^{m+2}\mathfrak h_{\le m-1}.
\end{align}
\end{lemma}
\begin{proof} We proceed by induction on $m\ge 3$. Let us assume Equation \eqref{eq:recursion level m}
holds for $m$. We need to construct $\pi_{m+1}$ such that \eqref{eq:recursion level m} holds after substituting $m\mapsto m+1$. Let us write $X_m:=\pi_0^{\le m}-\pi_0^{\le m-1}$ and decompose
\begin{align}\label{eq:recursion}
&\left\llbracket\pi_0^{\le m}+\sum_{i=1}^{m}\pi_i,\pi_0^{\le m}+\sum_{j=1}^{m}\pi_j\right\rrbracket
=\left\llbracket X_m+\pi_m+\pi_0^{\le m-1}+\sum_{i=1}^{m-1}\pi_i,X_m+\pi_m+\pi_0^{\le m-1}+\sum_{j=1}^{m-1}\pi_j\right\rrbracket\\
\nonumber
&=\left\llbracket\pi_0^{\le m-1}+\sum_{i=1}^{m-1}\pi_i,\pi_0^{\le m-1}+\sum_{j=1}^{m-1}\pi_j\right\rrbracket+\llbracket X_m,X_m\rrbracket  +\llbracket\pi_m,\pi_m\rrbracket
\\\nonumber
&\hspace{5cm}
+ 2\left\llbracket X_m,\pi_0^{\le m-1}+\sum_{i=1}^{m-1}\pi_i\right\rrbracket
+ 2\left\llbracket\pi_m,\pi_0^{\le m}+\sum_{i=1}^{m-1}\pi_i\right\rrbracket.
\end{align}
We wish to show that this is in $\mathcal F^{m+2}\mathfrak h_{\le m}$.
We have $ \llbracket X_m,X_m\rrbracket  =0$ and using Lemma \ref{lem:filtration} we see that for $m\ge j$ we get $\llbracket\pi_m,\pi_j\rrbracket  \in \mathcal F^{m+2}\mathfrak h_{\le m-1}$.
Moreover, $\llbracket X_m,\pi_0^{\le m-1}\rrbracket  =0$ since
\begin{align*}
0=2\left(\partial_{\le m}\right)^2=\left\llbracket\pi_0^{\le m},\pi_0^{\le m}\right\rrbracket=\llbracket X_m,X_m\rrbracket  +2\left\llbracket X_m,\pi_0^{\le m-1}\right\rrbracket+2\left(\partial_{\le m-1}\right)^2.
\end{align*}
Also, if $m>j$ it follows $\llbracket X_m,\pi_j\rrbracket  \in \mathcal F^{j+1+m-(j-1)}\mathfrak h_{\le m}=\mathcal F^{m+2}\mathfrak h_{\le m}$. Finally, we have $\llbracket \pi_m,\pi_0^{\le m}\rrbracket  =\partial_{\le m-1}\pi_m+\mathcal F^{m+2}\mathfrak h_{\le m-1}$, establishing the claim. It turns out that the only terms in Equation \eqref{eq:recursion} that are not necessarily in $\mathfrak h_{\le m-1}$ are the $\llbracket X_m,\pi_j\rrbracket  $. For the coefficients (i.e. the resolvent part) are actually in $\mathfrak h_{\le m-1}$ while the derivation part contains possibly one derivation from $\mathfrak h_{\le m}$.

Let now $A_m\in\mathcal F^{m+2}\mathfrak h_{\le m}\backslash \mathcal F^{m+3}\mathfrak h_{\le m}$ such that $A_m+\mathcal F^{m+3}\mathfrak h_{\le m}=\llbracket \pi_0^{\le m}+\sum_{i=1}^{m}\pi_i,\pi_0^{\le m}+\sum_{j=1}^{m}\pi_j\rrbracket  +\mathcal F^{m+3}\mathfrak h_{\le m}$. The Jacobi identity for the Schouten bracket and Lemma \ref{lem:filtration} allows us to conclude that $\partial_{\le m}A_m=0$. The argument goes as follows:
\begin{align*}
0=\left\llbracket\pi_0^{\le m}+\sum_{i=1}^{m}\pi_i,\left\llbracket\pi_0^{\le m}+\sum_{i=1}^{m}\pi_i,\pi_0^{\le m}+\sum_{j=1}^{m}\pi_j\right\rrbracket\right\rrbracket=\left\llbracket\pi_0^{\le m}+\sum_{i=1}^{m}\pi_i,A_m\right\rrbracket.
\end{align*}
But $\llbracket \pi_i,A_m\rrbracket  \in \mathcal F^{m+4}\mathfrak h_{\le m}$ and $\llbracket \pi_0^{\le m},A_m\rrbracket  \in\partial_{\le m} (A_m)+\mathcal F^{m+3}\mathfrak h_{\le m}$. We observe that, by construction, $\partial_{\le m}A_m=
\partial_{\le m-1}A_m$.
We choose $\pi_{m+1}$ such that $\partial_{\le m}\pi_{m+1}=-A_m/2$.
\end{proof}

\begin{theorem}\label{thm:derivedbr}
With the notation of Lemma \ref{lem:recursion}
$\pi:=\sum_{m=0}^\infty \pi_m$
defines an element in $\mathfrak g$ such that $\llbracket \pi,\pi\rrbracket  =0$.
Let $\downarrow:\mathfrak g\to \mathfrak g[1]$ denote the desuspension (i.e. the identity seen as a degree $-1$ map) and let $\uparrow$ be its inverse. Let us write $[\:,\:]=\downarrow\circ \llbracket\:,\:\rrbracket\circ \uparrow\otimes \uparrow$ for the Lie bracket on $\mathfrak g[1]$.
Let $\operatorname{pr}:\mathfrak g[1]\to R[1]$ be the canonical projection onto arity zero.  The operations $\left(l_m\right)_{n\ge 0}$
\begin{align}\label{eq:derived brackets}
l_m(x_1,x_2,\dots,x_m):=-\operatorname{pr}\left(\left[\dots[[\pi[1],x_1],x_2],\dots,x_m\right]\right)
\end{align}
define an $L_\infty[1]$-algebra on $R[1]$. Using \eqref{eq:decalage} this induces a $P_\infty$-algebra structure $\left(\{\:,\dots,\:\}_m\right)_{m\ge 0}$,
\begin{align}
\nonumber
\{a_1,\dots,a_m\}_m&:=(-1)^{\sum_{i=1}^m(m-i)|a_i|}l_m(a_1[1],\dots,a_m[1])[-1] \\
&=-(-1)^{\sum_{i=1}^m(m-i)|a_i|}\epsilon\left(\left\llbracket\dots\llbracket\llbracket \pi, a_1\rrbracket,a_2 \rrbracket,\dots,a_m\right\rrbracket\right),
\end{align}
on $R$, where $\epsilon:\mathfrak g\to R$ is the augmentation. The canonical map $p:R\to A$ is an $L_\infty$-quasiisomorphism.
\end{theorem}
\begin{proof} The series $\pi:=\sum_{m=0}^\infty \pi_m$ converges in the $\mathcal F$-adic topology. By completeness of $\mathfrak g$ (cf. Proposition \ref{prop:Fadic}) it follows that $\pi\in\mathfrak g$.
The construction of $l_m$ is a special case of the higher derived brackets of \cite{VoronovHigherDerived}. It is straightforward to check that the Leibniz rule holds for the higher Poisson brackets $\{\:,\dots,\:\}_m$.

To prove that $p$ is compatible with the brackets we show that the only contributions to the brackets
$\{a_1,\dots,a_m\}_m$ that are not annihilated by $p$ come from terms involving $\pi_1$. Let $\mathfrak a$ be the kernel of the canonical map $\kappa:R\to S$, i.e., the ideal in $R$ generated by the $x_j^{(m)}$ with $m\ge 1,j\in \mathcal I_m$. Recall that for $m\ge 1$ the tensor $\pi_m$ is of filtration degree $m+1$ and of cohomological degree $2$. Hence for $m\ge 2$ we have
\begin{align*}
\left\llbracket\dots\llbracket\llbracket \pi_m, \mathfrak a,\rrbracket,R \rrbracket,\dots,R\right\rrbracket\subseteq \mathfrak a\subseteq \ker(p).
\end{align*}
Moreover, $\llbracket \pi_0,R\rrbracket=\partial R\subseteq\ker(p)$ and the claim follows..
\end{proof}

If our ideal $I$ is generated by a complete intersection $f_1,\dots,f_k$ then for degree reasons $\pi_m=0$ for all $m\ge k+2$ and the sum in the theorem above is actually finite.

\begin{corollary} Let $\left(\{\:,\dots,\:\}_m\right)_{m\ge 1}$ be a $P_\infty$-algebra structure on the resolvent $R$ with $\partial=\{\:\}_1$.
Then the brackets $[\dR a_1,\dR a_2, \dots, \dR a_m]_m:=\dR\{a_1,a_2,\dots,a_m\}_m$, $m\ge 2$, define the structure of a  $L_\infty$-algebroid over $\Spec(A)$ on  $\mathbb L_{A|\bs k}=A\otimes _R\Omega_{R|\bs k}$ with anchor $\rho_m(\dR a_1,\dR a_2, \dots, \dR a_{m-1},b)_m:=\{a_1,a_2,\dots,a_{m-1},b\}_m$, $m\ge 2$. Here $[\:]_1$ is understood to be the differential  $A\otimes _R\partial$ of the cotangent complex.
\end{corollary}

\begin{proof}
The higher Jacobi identities \eqref{eqn:Linftyalgebra} and \eqref{eqn:Linftymodule} follow directly from the higher Jacobi identities of the $P_\infty$-structure. The identities (3) of Definition \ref{def:Linftyalgebroid} follow from the Leibniz rule of the $P_\infty$-structure.
\end{proof}

We would like to emphasize the following phenomenon. Let again $\mathfrak a$ be the ideal $\ker(\kappa)$ with  $\kappa:R\to S$. Then any bracket operation $[\:,\dots,\:]_m$, as well as the anchor $\rho_m$, on the cotangent complex $\mathbb L_{A|\bs k}$ that corresponds to a  $\pi_m\in\mathfrak a^2$ vanish. In the case of a complete intersection for $m\ge 3$ we have $\pi_m\in\mathfrak a^2$ for degree reasons. This means that in the case of a complete intersection we get a dg Lie algebroid and we to recover Theorem \ref{thm:ciLiealgebroid}.

Finally we would like to address a peculiarity of the Koszul case. Let us assume that in the constructions of Theorem \ref{thm:derivedbr} we chose the resolvent $R$ to be a minimal model. Let us introduce the \emph{Euler derivation}
\begin{align*}
\mathcal E:=\sum_{m\ge 0}\sum_{i\in \mathcal I_m}x_i^{(m)}{\partial \over \partial x_i^{(m)}}.
\end{align*}
We say that $X\in\mathfrak g$ is \emph{homogeneous of $x$-degree $k$} if $\mathcal E X=kX$. The $x$-degree can be interpreted as counting the number occurrences of the variables of the minimal model in a monomial of $X$, including the original variables $x_1,\dots,x_n$.

\begin{theorem}\label{thm:Koszul} Let $S=\bs k[x_1,\dots, x_n]$ be standard graded, assume that $I$ is generated by quadrics such that $S/I$ is a Koszul algebra, and suppose that $p:=\deg\{\:,\:\}\le 0$. Then all $\pi_m$ of Theorem \ref{thm:derivedbr} with $m\ge 1$ can be chosen homogeneous of the same $x$-degree $p+2$.
\end{theorem}

\begin{proof}
Due to Koszulness $\pi_0$ is homogeneous of $x$-degree $2$. With $\pi_1$ being of $x$-degree $k=p+2$ we can easily see that $\pi_2$ can be chosen of the same $x$-degree. In the inductive construction of $\pi_m$ of Lemma \ref{lem:recursion} we have to consider two cases. The $x$-degrees of
\begin{align*}
\left\llbracket\pi_0^{\le m-1},\sum_{i=1}^{m-1}\pi_i\right\rrbracket\quad \mbox{and }\quad
\left\llbracket\sum_{i=1}^{m-1}\pi_i,\sum_{i=1}^{m-1}\pi_i\right\rrbracket
\end{align*}
are $k+1$ and $2k-1$, respectively. If $k\le 1$ these degrees differ and the second term has to be zero on the nose since a nontrivial boundary has to be of $x$-degree $\ge 2$. So only the first term contributes to the $\pi_m$, which can be chosen to be homogeneous of $x$-degree $k$. If $k=2$ the $x$-degrees of both terms are $k+1=3=2k-1$ and $\pi_m$ can be chosen to be homogeneous of $x$-degree $2$.
\end{proof}

It seems that there are Koszul examples with Poisson brackets of degree $\ge 1$ where it is unlikely to find $\pi_m$ of homogeneous $x$-degree for all $m\ge 3$. To be more specific, we looked into the ideal generated by $x_1^2$ and $x_1x_2$ in $\bs k[x_1,x_2]$ and bracket $\{x_1,x_2\}=x_1x_2(x_1+x_2)$. We were not able to make choices such that $\pi_3$ and $\pi_4$ are of homogeneous $x$-degree.

\section{Sample calculations}\label{sec:examples}

In this section, we give the results of computations of the structures indicated by Theorem 1.4 in various examples. In each cases, the resolvent was computed using \emph{Macaulay2} and the \emph{dgalgebras} package \cite{M2,dgAlgsM2}; the $\pi_m$ were computed using \emph{Mathematica} \cite{Mathematica}. Details of these computations are available from the authors by request.

\subsection{A non-complete intersection generated by two monomials in dimension two} We consider the ideal in $\bs k[x_1,x_2]$ generated by the two quadratic  monomials $f_1 = x_1^2$ and $f_2 = x_1 x_2$ with diagonal Poisson bracket $\{x_1,x_2\}=x_1x_2$. The differential $\partial_{\le 7}$ can be read off
\begingroup
\allowdisplaybreaks
\begin{align*}
        \pi_0^{\leq 7}
        &=  x_1 x_1 \xi_{(1)}^1 + x_1 x_2 \xi_{(1)}^2 + x_1 x^{(1)}_2 \xi_{(2)}^1 - x_2 x^{(1)}_1 \xi_{(2)}^1- x^{(1)}_1 x^{(1)}_2 \xi_{(3)}^1+ x_1 x^{(2)}_1 \xi_{(3)}^1+ x_1 x^{(3)}_1 \xi_{(4)}^2 + x_2 x^{(3)}_1 \xi_{(4)}^1\\
         &\quad
            - x^{(1)}_1 x^{(2)}_1 \xi_{(4)}^2 - x^{(1)}_2 x^{(2)}_1 \xi_{(4)}^1 + x_1 x^{(4)}_1 \xi_{(5)}^1
            + x_1 x^{(4)}_2 \xi_{(5)}^3
            + 2 x_2 x^{(4)}_2 \xi_{(5)}^2 - x^{(1)}_1 x^{(3)}_1 \xi_{(5)}^3 - x^{(1)}_2 x^{(3)}_1 \xi_{(5)}^1
            \\&\quad
            - 2 x^{(1)}_2 x^{(3)}_1 \xi_{(5)}^2 - x^{(2)}_1 x^{(2)}_1 \xi_{(5)}^2 \
        + x_1 x_1^{(5)} \xi_{(6)}^3 + x_1 x_2^{(5)} \xi_{(6)}^2 + x_1 x_3^{(5)} \xi_{(6)}^4 - x_1^{(1)} x_1^{(4)} \xi_{(6)}^3 - x_1^{(1)} x_2^{(4)} \xi_{(6)}^4 \\&\quad
        + x_1^{(2)} x_1^{(3)} \xi_{(6)}^2 + x_2 x_1^{(5)} \xi_{(6)}^1
        - x_2 x_3^{(5)} \xi_{(6)}^2 + x_2 x_3^{(5)} \xi_{(6)}^3 - x_2^{(1)} x_1^{(4)} \xi_{(6)}^1 - x_2^{(1)} x_2^{(4)} \xi_{(6)}^2 - x_2^{(1)} x_2^{(4)} \xi_{(6)}
    \\
        &\quad+ x_1 x_1^{(6)} \xi_{(7)}^1 + x_1 x_2^{(6)} \xi_{(7)}^4 + x_1 x_3^{(6)} \xi_{(7)}^3 + x_1 x_4^{(6)} \xi_{(7)}^5 - x_1^{(1)} x_1^{(5)} \xi_{(7)}^3 - x_1^{(1)} x_2^{(5)} \xi_{(7)}^4 - x_1^{(1)} x_3^{(5)} \xi_{(7)}^5
    \\&\quad
        - x_1^{(2)} x_1^{(4)} \xi_{(7)}^2 - x_1^{(2)} x_2^{(4)} \xi_{(7)}^4 + x_2 x_2^{(6)} \xi_{(7)}^2 + x_2 x_3^{(6)} \xi_{(7)}^2 + 3 x_2 x_4^{(6)} \xi_{(7)}^4 - x_2^{(1)} x_1^{(5)} \xi_{(7)}^1 - x_2^{(1)} x_1^{(5)} \xi_{(7)}^2
    \\&\quad
        - x_2^{(1)} x_2^{(5)} \xi_{(7)}^2 - x_2^{(1)} x_3^{(5)} \xi_{(7)}^3 - 2 x_2^{(1)} x_3^{(5)} \xi_{(7)}^4
\end{align*}
by substituting $\xi_{(m)}^i$ with $\partial/\partial x^{(m)}_i$. In this example the exponential growth of the number of terms in $\partial_{\le m}$ sets in relatively late so that we can calculate $\pi_m$ up to $m=8$. We present the results for two different choices of $Z_{i\mu}^\nu$.
\subsubsection{A non-diagonal choice of $Z_{i\mu}^\nu$}\label{subsec:nondiag} Here we consider the $Z_{i\mu}^\nu$ with nonzero components $Z_{2,1}^2 =-2x_1$, $Z_{1,2}^2=x_1$, and $Z_{2,2}^2=-x_2$. Note that the tensor $Z_{i\mu}^\nu$ can in principal be read off $\pi_2$ so that in later examples we will not spell it out to avoid redundancy. We were able to calculate
\begin{align*}
	\pi_1
        &=  x_1 x_2 \xi^1 \xi^2,
    \\
	\pi_2
        &=   - x_1 x_2^{(1)} \xi^1 \xi_{(1)}^2 + 2   x_1 x_2^{(1)} \xi^2 \xi_{(1)}^1 + x_2 x_2^{(1)} \xi^2 \xi_{(1)}^2,
    \\
    \pi_3
        &=  x_1 x_1^{(2)} \xi^1 \xi_{(2)}^1,
    \\
    \pi_4
        &=   - x_1 x_1^{(3)} \xi^1 \xi_{(3)}^1 + 2   x_1 x_1^{(3)} \xi_{(1)}^1 \xi_{(2)}^1 + x_2 x_1^{(3)} \xi^2 \xi_{(3)}^1 + x_2 x_1^{(3)} \xi_{(1)}^2 \xi_{(2)}^1,
    \\
    \pi_5
        &=  2 x_1 x_1^{(4)} \xi^1 \xi_{(4)}^1 - 2   x_1 x_1^{(4)} \xi^2 \xi_{(4)}^2 - x_1 x_1^{(4)} \xi_{(1)}^2 \xi_{(3)}^1 + x_1 x_2^{(4)} \xi^1 \xi_{(4)}^2 - x_2 x_1^{(4)} \xi^2 \xi_{(4)}^1 - 2   x_1 x_1^{(4)} \xi^2 \xi_{(1)}^1 \xi_{(2)}^1,
    \\
    \pi_6
        &=   - 2 x_1 x_1^{(5)} \xi^1 \xi_{(5)}^1 + 2   x_1 x_1^{(5)} \xi^2 \xi_{(5)}^3 + 4   x_1 x_1^{(5)} \xi_{(1)}^1 \xi_{(4)}^1 - x_1 x_1^{(5)} \xi_{(1)}^2 \xi_{(4)}^2 + x_1 x_1^{(5)} \xi_{(2)}^1 \xi_{(3)}^1 - 2   x_1 x_2^{(5)} \xi^1 \xi_{(5)}^2
    \\&\quad
        - x_1 x_2^{(5)} \xi_{(2)}^1 \xi_{(3)}^1 - x_1 x_3^{(5)} \xi^1 \xi_{(5)}^3 + 2   x_1 x_3^{(5)} \xi_{(1)}^1 \xi_{(4)}^2 + 2   x_2 x_1^{(5)} \xi^2 \xi_{(5)}^1 + 4   x_2 x_1^{(5)} \xi^2 \xi_{(5)}^2 + 2   x_2 x_1^{(5)} \xi_{(1)}^2 \xi_{(4)}^1
    \\&\quad
        + x_2 x_3^{(5)} \xi^2 \xi_{(5)}^3 + x_2 x_3^{(5)} \xi_{(1)}^2 \xi_{(4)}^2 + 2   x_2 x_3^{(5)} \xi_{(2)}^1 \xi_{(3)}^1 + 4   x_1 x_1^{(5)} \xi^2 \xi_{(1)}^1 \xi_{(3)}^1 + x_2 x_1^{(5)} \xi_{(1)}^2 \xi_{(1)}^2 \xi_{(2)}^1,
    \\
    \pi_7
        &=  3 x_1 x_1^{(6)} \xi^1 \xi_{(6)}^1 - 2   x_1 x_1^{(6)} \xi^2 \xi_{(6)}^2 - 4   x_1 x_1^{(6)} \xi^2 \xi_{(6)}^3 - 2   x_1 x_1^{(6)} \xi_{(1)}^2 \xi_{(5)}^1 - 2   x_1 x_1^{(6)} \xi_{(1)}^2 \xi_{(5)}^2 - 2   x_1 x_1^{(6)} \xi_{(2)}^1 \xi_{(4)}^1
    \\&\quad
        + 2   x_1 x_2^{(6)} \xi^1 \xi_{(6)}^2 + 4   x_1 x_2^{(6)} \xi_{(1)}^1 \xi_{(5)}^2 + x_1 x_2^{(6)} \xi_{(2)}^1 \xi_{(4)}^2 + x_1 x_2^{(6)} \xi_{(3)}^1 \xi_{(3)}^1 + 2   x_1 x_3^{(6)} \xi^1 \xi_{(6)}^3 - 2   x_1 x_3^{(6)} \xi^2 \xi_{(6)}^4
    \\&\quad
        - x_1 x_3^{(6)} \xi_{(1)}^2 \xi_{(5)}^3 - 2   x_1 x_3^{(6)} \xi_{(2)}^1 \xi_{(4)}^2 - x_1 x_3^{(6)} \xi_{(3)}^1 \xi_{(3)}^1 + x_1 x_4^{(6)} \xi^1 \xi_{(6)}^4 - 2   x_2 x_1^{(6)} \xi^2 \xi_{(6)}^1 - x_2 x_2^{(6)} \xi^2 \xi_{(6)}^2
    \\&\quad
        + 2   x_2 x_2^{(6)} \xi_{(1)}^2 \xi_{(5)}^2 + x_2 x_2^{(6)} \xi_{(2)}^1 \xi_{(4)}^1 - x_2 x_3^{(6)} \xi^2 \xi_{(6)}^3 - x_2 x_3^{(6)} \xi_{(2)}^1 \xi_{(4)}^1 - 8   x_1 x_1^{(6)} \xi^2 \xi_{(1)}^1 \xi_{(4)}^1
    \\&\quad
        + 4   x_1 x_1^{(6)} \xi^2 \xi_{(1)}^2 \xi_{(4)}^2 + 2   x_1 x_1^{(6)} \xi^2 \xi_{(2)}^1 \xi_{(3)}^1 + x_1 x_1^{(6)} \xi_{(1)}^2 \xi_{(1)}^2 \xi_{(3)}^1 - 2   x_1 x_3^{(6)} \xi^2 \xi_{(1)}^1 \xi_{(4)}^2 + 4   x_1 x_1^{(6)} \xi^2 \xi_{(1)}^1 \xi_{(1)}^2 \xi_{(2)}^1,
\end{align*}
\begin{align*}
    \pi_8
        &=   - 3 x_1 x_1^{(7)} \xi^1 \xi_{(7)}^1 + 4   x_1 x_1^{(7)} \xi^2 \xi_{(7)}^3 + 2   x_1 x_1^{(7)} \xi^2 \xi_{(7)}^4 + 6   x_1 x_1^{(7)} \xi_{(1)}^1 \xi_{(6)}^1 - 2   x_1 x_1^{(7)} \xi_{(1)}^2 \xi_{(6)}^3 + 4   x_1 x_1^{(7)} \xi_{(2)}^1 \xi_{(5)}^1
    \\&\quad
        + 2   x_1 x_1^{(7)} \xi_{(2)}^1 \xi_{(5)}^2 - 2   x_1 x_1^{(7)} \xi_{(3)}^1 \xi_{(4)}^1 - 3   x_1 x_2^{(7)} \xi^1 \xi_{(7)}^2 + 2   x_1 x_2^{(7)} \xi^2 \xi_{(7)}^4 - x_1 x_2^{(7)} \xi_{(1)}^2 \xi_{(6)}^2 - x_1 x_2^{(7)} \xi_{(2)}^1 \xi_{(5)}^1
    \\&\quad
        - x_1 x_2^{(7)} \xi_{(3)}^1 \xi_{(4)}^1 - 2   x_1 x_3^{(7)} \xi^1 \xi_{(7)}^3 + 2   x_1 x_3^{(7)} \xi^2 \xi_{(7)}^5 + 4   x_1 x_3^{(7)} \xi_{(1)}^1 \xi_{(6)}^3 - x_1 x_3^{(7)} \xi_{(1)}^2 \xi_{(6)}^4 + 3   x_1 x_3^{(7)} \xi_{(2)}^1 \xi_{(5)}^3
    \\&\quad
        - 2   x_1 x_4^{(7)} \xi^1 \xi_{(7)}^4 - x_1 x_4^{(7)} \xi_{(2)}^1 \xi_{(5)}^3 - x_1 x_4^{(7)} \xi_{(3)}^1 \xi_{(4)}^2 - x_1 x_5^{(7)} \xi^1 \xi_{(7)}^5 + 2   x_1 x_5^{(7)} \xi_{(1)}^1 \xi_{(6)}^4 + 3   x_2 x_1^{(7)} \xi^2 \xi_{(7)}^1
    \\&\quad
        + 6   x_2 x_1^{(7)} \xi^2 \xi_{(7)}^2 + 3   x_2 x_1^{(7)} \xi_{(1)}^2 \xi_{(6)}^1 + x_2 x_2^{(7)} \xi^2 \xi_{(7)}^2 + 2   x_2 x_3^{(7)} \xi^2 \xi_{(7)}^3 + 4   x_2 x_3^{(7)} \xi^2 \xi_{(7)}^4 + 2   x_2 x_3^{(7)} \xi_{(1)}^2 \xi_{(6)}^3
    \\&\quad
        + 2   x_2 x_3^{(7)} \xi_{(2)}^1 \xi_{(5)}^1 + 4   x_2 x_3^{(7)} \xi_{(2)}^1 \xi_{(5)}^2 + 2   x_2 x_3^{(7)} \xi_{(3)}^1 \xi_{(4)}^1 + x_2 x_5^{(7)} \xi^2 \xi_{(7)}^5 + x_2 x_5^{(7)} \xi_{(1)}^2 \xi_{(6)}^4
    \\&\quad
        + 3   x_2 x_5^{(7)} \xi_{(2)}^1 \xi_{(5)}^3 + 3   x_2 x_5^{(7)} \xi_{(3)}^1 \xi_{(4)}^2 + 12   x_1 x_1^{(7)} \xi^2 \xi_{(1)}^1 \xi_{(5)}^1 + 12   x_1 x_1^{(7)} \xi^2 \xi_{(1)}^1 \xi_{(5)}^2 - 4   x_1 x_1^{(7)} \xi^2 \xi_{(1)}^2 \xi_{(5)}^3
    \\&\quad
        - 4   x_1 x_1^{(7)} \xi^2 \xi_{(3)}^1 \xi_{(3)}^1 + x_1 x_1^{(7)} \xi_{(1)}^2 \xi_{(1)}^2 \xi_{(4)}^2 - 2   x_1 x_1^{(7)} \xi_{(1)}^2 \xi_{(2)}^1 \xi_{(3)}^1 + 4   x_1 x_2^{(7)} \xi^2 \xi_{(1)}^1 \xi_{(5)}^2
    \\&\quad
        + x_1 x_2^{(7)} \xi_{(1)}^2 \xi_{(2)}^1 \xi_{(3)}^1 + 4   x_1 x_3^{(7)} \xi^2 \xi_{(1)}^1 \xi_{(5)}^3 + 4   x_1 x_3^{(7)} \xi_{(1)}^1 \xi_{(2)}^1 \xi_{(3)}^1 + 3   x_2 x_1^{(7)} \xi_{(1)}^2 \xi_{(1)}^2 \xi_{(4)}^1
    \\&\quad
        + x_2 x_3^{(7)} \xi_{(1)}^2 \xi_{(1)}^2 \xi_{(4)}^2 + 2   x_2 x_3^{(7)} \xi_{(1)}^2 \xi_{(2)}^1 \xi_{(3)}^1 - 8   x_1 x_1^{(7)} \xi^2 \xi_{(1)}^1 \xi_{(1)}^2 \xi_{(3)}^1 + 2 x_1 x_1^{(7)} \xi_{(1)}^1 \xi_{(1)}^2 \xi_{(1)}^2 \xi_{(2)}^1
    \\&\quad
        + x_2 x_1^{(7)} \xi_{(1)}^2 \xi_{(1)}^2 \xi_{(1)}^2 \xi_{(2)}^1.
\end{align*}
Here $\dP Z-[Z,Z]$ as well as all $\mathcal A_{\mu\nu}^\lambda$ are zero.
We observe that none of the terms in the above list are in $\mathfrak a^2$. We do not know if that is a general feature of this example.

\subsubsection{Diagonal $Z_{i\mu}^\nu$} With the choices according to Proposition \ref{prop:diag} we get the following.
\begin{align*}
    \pi_1
        &=  x_1 x_2 \xi^1 \xi^2,
    \\
    \pi_2
        &=   - x_1 x_2^{(1)} \xi^1 \xi_{(1)}^2 + 2 x_2 x_1^{(1)} \xi^2 \xi_{(1)}^1 + x_2 x_2^{(1)} \xi^2 \xi_{(1)}^2,
    \\
    \pi_3
        &=  x_1 x_1^{(2)} \xi^1 \xi_{(2)}^1 + 2 x_1 x_1^{(2)} \xi_{(1)}^1 \xi_{(1)}^2 - 2 x_2 x_1^{(2)} \xi^2 \xi_{(2)}^1,
    \\
    \pi_4
        &=   - x_1 x_1^{(3)} \xi^1 \xi_{(3)}^1 + 3 x_2 x_1^{(3)} \xi^2 \xi_{(3)}^1 + x_2 x_1^{(3)} \xi_{(1)}^2 \xi_{(2)}^1,
    \\
    \pi_5
        &=  2 x_1 x_1^{(4)} \xi^1 \xi_{(4)}^1 - 3 x_1 x_1^{(4)} \xi_{(1)}^2 \xi_{(3)}^1 + x_1 x_2^{(4)} \xi^1 \xi_{(4)}^2 + 2 x_1 x_2^{(4)} \xi_{(1)}^1 \xi_{(3)}^1 - 3 x_2 x_1^{(4)} \xi^2 \xi_{(4)}^1 - 4 x_2 x_2^{(4)} \xi^2 \xi_{(4)}^2
    \\&\quad
        - 2 x_1 x_1^{(4)} \xi_{(1)}^1 \xi_{(1)}^2 \xi_{(1)}^2,
   \\
   \pi_6
        &=   - 2 x_1 x_1^{(5)} \xi^1 \xi_{(5)}^1 + 4 x_1 x_1^{(5)} \xi_{(1)}^1 \xi_{(4)}^1 + 3 x_1 x_1^{(5)} \xi_{(1)}^2 \xi_{(4)}^2 + 3 x_1 x_1^{(5)} \xi_{(2)}^1 \xi_{(3)}^1 - 2 x_1 x_2^{(5)} \xi^1 \xi_{(5)}^2
   \\&\quad
        - 2 x_1 x_2^{(5)} \xi_{(1)}^1 \xi_{(4)}^1 - 2 x_1 x_2^{(5)} \xi_{(1)}^2 \xi_{(4)}^2 - x_1 x_2^{(5)} \xi_{(2)}^1 \xi_{(3)}^1 - x_1 x_3^{(5)} \xi^1 \xi_{(5)}^3 + 4 x_2 x_1^{(5)} \xi^2 \xi_{(5)}^1
   \\&\quad
        + 2 x_2 x_1^{(5)} \xi_{(1)}^2 \xi_{(4)}^1 + 4 x_2 x_2^{(5)} \xi^2 \xi_{(5)}^2 + 5 x_2 x_3^{(5)} \xi^2 \xi_{(5)}^3 + 2 x_2 x_3^{(5)} \xi_{(1)}^1 \xi_{(4)}^1 + x_2 x_3^{(5)} \xi_{(1)}^2 \xi_{(4)}^2
   \\&\quad
        + 4 x_2 x_3^{(5)} \xi_{(2)}^1 \xi_{(3)}^1 + 4 x_1 x_1^{(5)} \xi_{(1)}^1 \xi_{(1)}^2 \xi_{(2)}^1 - 2 x_1 x_2^{(5)} \xi_{(1)}^1 \xi_{(1)}^2 \xi_{(2)}^1 + x_2 x_1^{(5)} \xi_{(1)}^2 \xi_{(1)}^2 \xi_{(2)}^1
   \\&\quad
        + 2 x_2 x_3^{(5)} \xi_{(1)}^1 \xi_{(1)}^2 \xi_{(2)}^1,
   \\
   \pi_7
        &=  3 x_1 x_1^{(6)} \xi^1 \xi_{(6)}^1 - 4 x_1 x_1^{(6)} \xi_{(1)}^2 \xi_{(5)}^1 - 6 x_1 x_1^{(6)} \xi_{(1)}^2 \xi_{(5)}^2 - 2 x_1 x_1^{(6)} \xi_{(2)}^1 \xi_{(4)}^1 + 2 x_1 x_2^{(6)} \xi^1 \xi_{(6)}^2
   \\&\quad
        + 2 x_1 x_2^{(6)} \xi_{(1)}^1 \xi_{(5)}^1 + 4 x_1 x_2^{(6)} \xi_{(1)}^1 \xi_{(5)}^2 + 2 x_1 x_2^{(6)} \xi_{(1)}^2 \xi_{(5)}^3 - x_1 x_2^{(6)} \xi_{(2)}^1 \xi_{(4)}^2 + x_1 x_2^{(6)} \xi_{(3)}^1 \xi_{(3)}^1
   \\&\quad
        + 2 x_1 x_3^{(6)} \xi^1 \xi_{(6)}^3 - 3 x_1 x_3^{(6)} \xi_{(1)}^2 \xi_{(5)}^3 - 3 x_1 x_3^{(6)} \xi_{(3)}^1 \xi_{(3)}^1 + x_1 x_4^{(6)} \xi^1 \xi_{(6)}^4 + 2 x_1 x_4^{(6)} \xi_{(1)}^1 \xi_{(5)}^3
   \\&\quad
        - 4 x_2 x_1^{(6)} \xi^2 \xi_{(6)}^1 - 5 x_2 x_2^{(6)} \xi^2 \xi_{(6)}^2 + 2  x_2 x_2^{(6)} \xi_{(1)}^2 \xi_{(5)}^2 + 3 x_2 x_2^{(6)} \xi_{(2)}^1 \xi_{(4)}^1 - 5 x_2 x_3^{(6)} \xi^2 \xi_{(6)}^3
   \\&\quad
        - 3  x_2 x_3^{(6)} \xi_{(2)}^1 \xi_{(4)}^1 - 6 x_2 x_4^{(6)} \xi^2 \xi_{(6)}^4 + 4 x_2 x_4^{(6)} \xi_{(1)}^1 \xi_{(5)}^2 - 4  x_2 x_4^{(6)} \xi_{(2)}^1 \xi_{(4)}^2 + 5 x_1 x_1^{(6)} \xi_{(1)}^2 \xi_{(1)}^2 \xi_{(3)}^1
   \\&\quad
        + 2 x_1 x_2^{(6)} \xi_{(1)}^1 \xi_{(1)}^2 \xi_{(3)}^1 - 4 x_1 x_3^{(6)} \xi_{(1)}^1 \xi_{(1)}^2 \xi_{(3)}^1 + 2 x_1 x_1^{(6)} \xi_{(1)}^1 \xi_{(1)}^2 \xi_{(1)}^2 \xi_{(1)}^2,
\end{align*}
\begin{align*}
   \pi_8
        &=   - 3 x_1 x_1^{(7)} \xi^1 \xi_{(7)}^1 + 6 x_1 x_1^{(7)} \xi_{(1)}^1 \xi_{(6)}^1 + 6 x_1 x_1^{(7)} \xi_{(1)}^2 \xi_{(6)}^2 + 4 x_1 x_1^{(7)} \xi_{(1)}^2 \xi_{(6)}^3 + 6 x_1 x_1^{(7)} \xi_{(2)}^1 \xi_{(5)}^1
   \\&\quad
        + 6 x_1 x_1^{(7)} \xi_{(2)}^1 \xi_{(5)}^2 - 2 x_1 x_1^{(7)} \xi_{(3)}^1 \xi_{(4)}^1 - 3 x_1 x_2^{(7)} \xi^1 \xi_{(7)}^2 - 2 x_1 x_2^{(7)} \xi_{(1)}^1 \xi_{(6)}^1 - 3 x_1 x_2^{(7)} \xi_{(1)}^2 \xi_{(6)}^2
   \\&\quad
        - 2 x_1 x_2^{(7)} \xi_{(1)}^2 \xi_{(6)}^3 - x_1 x_2^{(7)} \xi_{(2)}^1 \xi_{(5)}^1 + x_1 x_2^{(7)} \xi_{(3)}^1 \xi_{(4)}^1 - 2 x_1 x_3^{(7)} \xi^1 \xi_{(7)}^3 + 4 x_1 x_3^{(7)} \xi_{(1)}^1 \xi_{(6)}^3
   \\&\quad
        + 3 x_1 x_3^{(7)} \xi_{(1)}^2 \xi_{(6)}^4 + 3 x_1 x_3^{(7)} \xi_{(2)}^1 \xi_{(5)}^3 + 6 x_1 x_3^{(7)} \xi_{(3)}^1 \xi_{(4)}^2 - 2 x_1 x_4^{(7)} \xi^1 \xi_{(7)}^4 - 2 x_1 x_4^{(7)} \xi_{(1)}^1 \xi_{(6)}^3
   \\&\quad
        - 2 x_1 x_4^{(7)} \xi_{(1)}^2 \xi_{(6)}^4 - x_1 x_4^{(7)} \xi_{(2)}^1 \xi_{(5)}^3 - 3 x_1 x_4^{(7)} \xi_{(3)}^1 \xi_{(4)}^2 - x_1 x_5^{(7)} \xi^1 \xi_{(7)}^5 + 5 x_2 x_1^{(7)} \xi^2 \xi_{(7)}^1
   \\&\quad
        + 3 x_2 x_1^{(7)} \xi_{(1)}^2 \xi_{(6)}^1 + 5 x_2 x_2^{(7)} \xi^2 \xi_{(7)}^2 + 6 x_2 x_3^{(7)} \xi^2 \xi_{(7)}^3 + 2 x_2 x_3^{(7)} \xi_{(1)}^1 \xi_{(6)}^1 + 2 x_2 x_3^{(7)} \xi_{(1)}^2 \xi_{(6)}^3
   \\&\quad
        + 4 x_2 x_3^{(7)} \xi_{(2)}^1 \xi_{(5)}^1 + 2 x_2 x_3^{(7)} \xi_{(3)}^1 \xi_{(4)}^1 + 6 x_2 x_4^{(7)} \xi^2 \xi_{(7)}^4 + 4 x_2 x_4^{(7)} \xi_{(2)}^1 \xi_{(5)}^2 + 7 x_2 x_5^{(7)} \xi^2 \xi_{(7)}^5
   \\&\quad
        - 2 x_2 x_5^{(7)} \xi_{(1)}^1 \xi_{(6)}^2 + 4 x_2 x_5^{(7)} \xi_{(1)}^1 \xi_{(6)}^3 + x_2 x_5^{(7)} \xi_{(1)}^2 \xi_{(6)}^4 + 9 x_2 x_5^{(7)} \xi_{(2)}^1 \xi_{(5)}^3 + 5 x_2 x_5^{(7)} \xi_{(3)}^1 \xi_{(4)}^2
   \\&\quad
        - 5 x_1 x_1^{(7)} \xi_{(1)}^2 \xi_{(1)}^2 \xi_{(4)}^2 - 10 x_1 x_1^{(7)} \xi_{(1)}^2 \xi_{(2)}^1 \xi_{(3)}^1 + 2 x_1 x_2^{(7)} \xi_{(1)}^2 \xi_{(1)}^2 \xi_{(4)}^2 + 3 x_1 x_2^{(7)} \xi_{(1)}^2 \xi_{(2)}^1 \xi_{(3)}^1
   \\&\quad
        + 4 x_1 x_3^{(7)} \xi_{(1)}^1 \xi_{(1)}^2 \xi_{(4)}^2 + 4 x_1 x_3^{(7)} \xi_{(1)}^1 \xi_{(2)}^1 \xi_{(3)}^1 - 2 x_1 x_4^{(7)} \xi_{(1)}^1 \xi_{(1)}^2 \xi_{(4)}^2 - 2 x_1 x_4^{(7)} \xi_{(1)}^1 \xi_{(2)}^1 \xi_{(3)}^1
   \\&\quad
        + 3 x_2 x_1^{(7)} \xi_{(1)}^2 \xi_{(1)}^2 \xi_{(4)}^1 + 4 x_2 x_3^{(7)} \xi_{(1)}^1 \xi_{(1)}^2 \xi_{(4)}^1 + x_2 x_3^{(7)} \xi_{(1)}^2 \xi_{(1)}^2 \xi_{(4)}^2 + 2 x_2 x_3^{(7)} \xi_{(1)}^2 \xi_{(2)}^1 \xi_{(3)}^1
   \\&\quad
        + 4 x_2 x_5^{(7)} \xi_{(1)}^1 \xi_{(1)}^1 \xi_{(4)}^1 + 4 x_2 x_5^{(7)} \xi_{(1)}^1 \xi_{(1)}^2 \xi_{(4)}^2 + 10 x_2 x_5^{(7)} \xi_{(1)}^1 \xi_{(2)}^1 \xi_{(3)}^1 - 6 x_1 x_1^{(7)} \xi_{(1)}^1 \xi_{(1)}^2 \xi_{(1)}^2 \xi_{(2)}^1
   \\&\quad
        + 2 x_1 x_2^{(7)} \xi_{(1)}^1 \xi_{(1)}^2 \xi_{(1)}^2 \xi_{(2)}^1 + x_2 x_1^{(7)} \xi_{(1)}^2 \xi_{(1)}^2 \xi_{(1)}^2 \xi_{(2)}^1 + 2 x_2 x_3^{(7)} \xi_{(1)}^1 \xi_{(1)}^2 \xi_{(1)}^2 \xi_{(2)}^1
   \\&\quad
        + 4 x_2 x_5^{(7)} \xi_{(1)}^1 \xi_{(1)}^1 \xi_{(1)}^2 \xi_{(2)}^1.
\end{align*}
None of the terms above is in $\mathfrak a^2$. For $m\ge 3$ the tensor appears to have nondiagonal terms.

\subsection{A complete intersection given by two polynomials} Let us consider the ideal in $\bs k[x_1,x_2,x_3,x_4]$ generated by the two quadratic  monomials $f_1 = x_1x_2$ and $f_2 = x_3 x_4$ with diagonal Poisson bracket $\{x_i,x_j\}=x_ix_j$ for $1\le i<j\le 4$. It turns out that
\begin{align*}
    \pi_0
        &=  x_1 x_2 \xi_{(1)}^1 + x_3 x_4 \xi_{(1)}^2,\\
    \pi_1
        &=  x_1 x_2 \xi^1 \xi^2 + x_2 x_3 \xi^2 \xi^3 + x_3 x_4 \xi^3 \xi^4,
    \\
    \pi_2
        &=  - x_1 x^{(1)}_1 \xi^1 \xi_{(1)}^1 + x_2 x^{(1)}_1 \xi^2 \xi_{(1)}^1 - x_2 x^{(1)}_2 \xi^2 \xi_{(1)}^2 + x_3 x^{(1)}_1 \xi^3 \xi_{(1)}^1
            - x_3 x^{(1)}_2 \xi^3 \xi_{(1)}^2 + x_4 x^{(1)}_2 \xi^4 \xi_{(1)}^2,
    \\
    \pi_3
        &=  x^{(1)}_1 x^{(1)}_2 \xi_{(1)}^1 \xi_{(1)}^2,
    \\
    \pi_m
        &=  0\quad \mbox{for }m\ge 4.
\end{align*}
In other words, the Koszul complex here is a dg Poisson algebra. As there are only two indices $i=1,2$, any cubic expression in the $x^{(1)}_i$ has to vanish. This explains the vanishing of $\pi_m$ for $m\ge 4$. More generally, we observe that the Koszul complex of a regular sequence of length $2$ with vanishing $\dP Z-[Z,Z]$ has the structure of a dg Poisson algebra.  We note that the $\pi_m$ have a diagonal shape for $m\ge 1$. Up to now all examples of Poisson complete intersections that we are aware of are monomial, Casimir, or a hypersurface.

\subsection{An ideal generated two cubic Casimirs}
We consider the ideal of $\bs k[x_1, x_2, x_3, x_4]$ generated by the two monomials $f_1 = x_1 x_2 x_3$ and $f_2 = x_2 x_3 x_4$. As the Poisson bracket we take that of Theorem \ref{thm:detbracket} with derivations $\partial/\partial x_1$, $\partial/\partial x_2$, $\partial/\partial x_3$, $\partial/\partial x_4$ and $g=(x_2x_3)^{-1}$  (it can also be seen as a diagonal bracket). To calculate the first terms of $\pi_0$ we used \emph{Macaulay2}:
\begin{align*}
    \pi_0^{\le 11}
        &=  x_1 x_2 x_3 \xi_{(1)}^1 + x_2 x_3 x_4 \xi_{(1)}^2 + x_1 x^{(1)}_2 \xi_{(2)}^1 - x_4 x^{(1)}_1 \xi_{(2)}^1+  x_2 x_3 x^{(2)}_1 \xi_{(3)}^1 - x^{(1)}_1 x^{(1)}_2 \xi_{(3)}^1\\
        &\quad+  x_1 x^{(3)}_1 \xi_{(4)}^2 + x_4 x^{(3)}_1 \xi_{(4)}^1 - x^{(1)}_1 x^{(2)}_1 \xi_{(4)}^2 - x^{(1)}_2 x^{(2)}_1 \xi_{(4)}^1
    \\
        &\quad+  2 x_1 x^{(4)}_1 \xi_{(5)}^1 - 2 x_4 x^{(4)}_2 \xi_{(5)}^1 + x_2 x_3 x^{(4)}_1 \xi_{(5)}^2 + x_2 x_3 x^{(4)}_2 \xi_{(5)}^3 - x^{(1)}_1 x^{(3)}_1 \xi_{(5)}^3 - x^{(1)}_2 x^{(3)}_1 \xi_{(5)}^2 + x^{(2)}_1 x^{(2)}_1 \xi_{(5)}^1
    \\
        &\quad+ x_1 x_2^{(5)} \xi_{(6)}^2 + x_1 x_3^{(5)} \xi_{(6)}^4 - x_1^{(1)} x_1^{(4)} \xi_{(6)}^2
            - 2 x_1^{(1)} x_1^{(4)} \xi_{(6)}^3 - x_1^{(1)} x_2^{(4)} \xi_{(6)}^4 - x_1^{(2)} x_1^{(3)} \xi_{(6)}^3
            - x_2^{(1)} x_1^{(4)} \xi_{(6)}^1
        \\&\quad
            - x_2^{(1)} x_2^{(4)} \xi_{(6)}^2 - x_2^{(1)} x_2^{(4)} \xi_{(6)}^3 + x_4 x_2^{(5)} \xi_{(6)}^1
            + x_4 x_3^{(5)} \xi_{(6)}^2 + 3 x_4 x_3^{(5)} \xi_{(6)}^3 + x_2 x_3 x_1^{(5)} \xi_{(6)}^3
       \\
        &\quad + 3 x_1 x_1^{(6)} \xi_{(7)}^1 + x_1 x_3^{(6)} \xi_{(7)}^2 - x_1^{(1)} x_1^{(5)} \xi_{(7)}^2
            - x_1^{(1)} x_2^{(5)} \xi_{(7)}^4 - x_1^{(1)} x_3^{(5)} \xi_{(7)}^5 + x_1^{(2)} x_1^{(4)} \xi_{(7)}^1
            + x_1^{(2)} x_2^{(4)} \xi_{(7)}^2
        \\&\quad
            - x_2^{(1)} x_1^{(5)} \xi_{(7)}^1 - x_2^{(1)} x_2^{(5)} \xi_{(7)}^3 - x_2^{(1)} x_3^{(5)} \xi_{(7)}^4
            - 3 x_4 x_2^{(6)} \xi_{(7)}^1 + x_4 x_3^{(6)} \xi_{(7)}^1 - 3   x_4 x_4^{(6)} \xi_{(7)}^2
        \\&\quad
            + x_2 x_3 x_1^{(6)} \xi_{(7)}^3 + x_2 x_3 x_2^{(6)} \xi_{(7)}^4 + x_2 x_3 x_4^{(6)} \xi_{(7)}^5
        \\
        &\quad
        +x_1 x_3^{(7)} \xi_{(8)}^4 + x_1 x_4^{(7)} \xi_{(8)}^6 + x_1 x_5^{(7)} \xi_{(8)}^8
        - 3 x_1^{(1)} x_1^{(6)} \xi_{(8)}^1 - x_1^{(1)} x_1^{(6)} \xi_{(8)}^4 - x_1^{(1)} x_2^{(6)} \xi_{(8)}^6
        - x_1^{(1)} x_3^{(6)} \xi_{(8)}^2
    \\&\quad
        - x_1^{(1)} x_4^{(6)} \xi_{(8)}^8 + 3 x_1^{(2)} x_2^{(5)} \xi_{(8)}^1 + x_1^{(2)} x_2^{(5)} \xi_{(8)}^4
        - x_1^{(2)} x_2^{(5)} \xi_{(8)}^5 + 3 x_1^{(2)} x_3^{(5)} \xi_{(8)}^2 + x_1^{(2)} x_3^{(5)} \xi_{(8)}^6
        - x_1^{(2)} x_3^{(5)} \xi_{(8)}^7
    \\&\quad
        - 4 x_1^{(3)} x_1^{(4)} \xi_{(8)}^1 - x_1^{(3)} x_1^{(4)} \xi_{(8)}^4 + x_1^{(3)} x_1^{(4)} \xi_{(8)}^5
        - 4 x_1^{(3)} x_2^{(4)} \xi_{(8)}^2 - x_1^{(3)} x_2^{(4)} \xi_{(8)}^6 + x_1^{(3)} x_2^{(4)} \xi_{(8)}^7
        - x_2^{(1)} x_1^{(6)} \xi_{(8)}^3
    \\&\quad
        + 3 x_2^{(1)} x_2^{(6)} \xi_{(8)}^1 - x_2^{(1)} x_2^{(6)} \xi_{(8)}^5 - x_2^{(1)} x_3^{(6)} \xi_{(8)}^1
        + 3 x_2^{(1)} x_4^{(6)} \xi_{(8)}^2 - x_2^{(1)} x_4^{(6)} \xi_{(8)}^7 + x_4 x_3^{(7)} \xi_{(8)}^3
        + x_4 x_4^{(7)} \xi_{(8)}^5
    \\&\quad
        + x_4 x_5^{(7)} \xi_{(8)}^7 + x_2 x_3 x_1^{(7)} \xi_{(8)}^1 + x_2 x_3 x_2^{(7)} \xi_{(8)}^2\\
        &\quad
        +x_1 x_1^{(8)} \xi_{(9)}^2 + x_1 x_2^{(8)} \xi_{(9)}^4 + 2 x_1 x_3^{(8)} \xi_{(9)}^5 + x_1 x_5^{(8)} \xi_{(9)}^6
        + 2 x_1 x_7^{(8)} \xi_{(9)}^7 - x_1^{(1)} x_1^{(7)} \xi_{(9)}^2 - x_1^{(1)} x_2^{(7)} \xi_{(9)}^4
    \\&\quad
        - x_1^{(1)} x_3^{(7)} \xi_{(9)}^9 - x_1^{(1)} x_4^{(7)} \xi_{(9)}^{11} - x_1^{(1)} x_5^{(7)} \xi_{(9)}^{13}
        - 3 x_1^{(2)} x_1^{(6)} \xi_{(9)}^1 + 2 x_1^{(2)} x_1^{(6)} \xi_{(9)}^5 - 3 x_1^{(2)} x_2^{(6)} \xi_{(9)}^2
        + x_1^{(2)} x_2^{(6)} \xi_{(9)}^6
    \\&\quad
        + x_1^{(2)} x_3^{(6)} \xi_{(9)}^2 - x_1^{(2)} x_3^{(6)} \xi_{(9)}^3 - 3 x_1^{(2)} x_4^{(6)} \xi_{(9)}^4
        + 2 x_1^{(2)} x_4^{(6)} \xi_{(9)}^7 - x_1^{(3)} x_1^{(5)} \xi_{(9)}^2 + x_1^{(3)} x_1^{(5)} \xi_{(9)}^3
        - x_1^{(3)} x_2^{(5)} \xi_{(9)}^9
    \\&\quad
        + x_1^{(3)} x_2^{(5)} \xi_{(9)}^{10} - x_1^{(3)} x_3^{(5)} \xi_{(9)}^{11} + x_1^{(3)} x_3^{(5)} \xi_{(9)}^{12}
        + 2 x_1^{(4)} x_1^{(4)} \xi_{(9)}^1 - x_1^{(4)} x_1^{(4)} \xi_{(9)}^5 + 2 x_1^{(4)} x_2^{(4)} \xi_{(9)}^2
        + 2 x_1^{(4)} x_2^{(4)} \xi_{(9)}^3
    \\&\quad
        - x_1^{(4)} x_2^{(4)} \xi_{(9)}^6 - x_2^{(1)} x_1^{(7)} \xi_{(9)}^1 - x_2^{(1)} x_2^{(7)} \xi_{(9)}^3
        - x_2^{(1)} x_3^{(7)} \xi_{(9)}^8 - x_2^{(1)} x_4^{(7)} \xi_{(9)}^{10} - x_2^{(1)} x_5^{(7)} \xi_{(9)}^{12}
        + 2 x_2^{(4)} x_2^{(4)} \xi_{(9)}^4
    \\&\quad
        - x_2^{(4)} x_2^{(4)} \xi_{(9)}^7 + x_4 x_1^{(8)} \xi_{(9)}^1 + x_4 x_2^{(8)} \xi_{(9)}^3
        - 2 x_4 x_4^{(8)} \xi_{(9)}^5 - x_4 x_6^{(8)} \xi_{(9)}^6 - 2 x_4 x_8^{(8)} \xi_{(9)}^7
        + x_2 x_3 x_3^{(8)} \xi_{(9)}^8
    \\&\quad
        + x_2 x_3 x_4^{(8)} \xi_{(9)}^9 + x_2 x_3 x_5^{(8)} \xi_{(9)}^{10} + x_2 x_3 x_6^{(8)} \xi_{(9)}^{11}
        + x_2 x_3 x_7^{(8)} \xi_{(9)}^{12} + x_2 x_3 x_8^{(8)} \xi_{(9)}^{13}\\
    &\quad+
    x_1 x_1^{(9)} \xi_{(10)}^1 + x_1 x_3^{(9)} \xi_{(10)}^2 + 2   x_1 x_8^{(9)} \xi_{(10)}^5 + 3 x_1 x_9^{(9)} \xi_{(10)}^{10} + x_1 x_{10}^{(9)} \xi_{(10)}^{11} + x_1 x_{11}^{(9)} \xi_{(10)}^{15}
    \\&\quad
    + x_1 x_{12}^{(9)} \xi_{(10)}^{16} + x_1 x_{13}^{(9)} \xi_{(10)}^{18} - 3 x_1^{(1)} x_1^{(8)} \xi_{(10)}^8 - x_1^{(1)} x_2^{(8)} \xi_{(10)}^{14} - 2   x_1^{(1)} x_3^{(8)} \xi_{(10)}^5 - 4 x_1^{(1)} x_3^{(8)} \xi_{(10)}^6
    \\&\quad
    - 3   x_1^{(1)} x_4^{(8)} \xi_{(10)}^{10} - x_1^{(1)} x_5^{(8)} \xi_{(10)}^{11} - 3   x_1^{(1)} x_5^{(8)} \xi_{(10)}^{12} - x_1^{(1)} x_6^{(8)} \xi_{(10)}^{15} - x_1^{(1)} x_7^{(8)} \xi_{(10)}^{16}
    \\&\quad
    - 2 x_1^{(1)} x_7^{(8)} \xi_{(10)}^{17} - x_1^{(1)} x_8^{(8)} \xi_{(10)}^{18} + x_1^{(2)} x_1^{(7)} \xi_{(10)}^1 + x_1^{(2)} x_2^{(7)} \xi_{(10)}^2 + 5 x_1^{(2)} x_3^{(7)} \xi_{(10)}^4 + x_1^{(2)} x_3^{(7)} \xi_{(10)}^5
    \\&\quad
    - x_1^{(2)} x_3^{(7)} \xi_{(10)}^6 - x_1^{(2)} x_3^{(7)} \xi_{(10)}^7 + 5 x_1^{(2)} x_4^{(7)} \xi_{(10)}^8 + 5   x_1^{(2)} x_4^{(7)} \xi_{(10)}^9 + x_1^{(2)} x_4^{(7)} \xi_{(10)}^{10} - 2   x_1^{(2)} x_4^{(7)} \xi_{(10)}^{12}
    \\&\quad
    - x_1^{(2)} x_4^{(7)} \xi_{(10)}^{13} + 5   x_1^{(2)} x_5^{(7)} \xi_{(10)}^{14} + x_1^{(2)} x_5^{(7)} \xi_{(10)}^{15} - x_1^{(2)} x_5^{(7)} \xi_{(10)}^{16} - 3 x_1^{(2)} x_5^{(7)} \xi_{(10)}^{17} + x_1^{(3)} x_1^{(6)} \xi_{(10)}^4
    \\&\quad
    - x_1^{(3)} x_1^{(6)} \xi_{(10)}^5 - 3   x_1^{(3)} x_1^{(6)} \xi_{(10)}^6 + x_1^{(3)} x_1^{(6)} \xi_{(10)}^7 + 4   x_1^{(3)} x_2^{(6)} \xi_{(10)}^8 - 5 x_1^{(3)} x_2^{(6)} \xi_{(10)}^9 - x_1^{(3)} x_2^{(6)} \xi_{(10)}^{10}
    \\&\quad
    - x_1^{(3)} x_2^{(6)} \xi_{(10)}^{12} + x_1^{(3)} x_2^{(6)} \xi_{(10)}^{13} - 3 x_1^{(3)} x_3^{(6)} \xi_{(10)}^8 + 3   x_1^{(3)} x_3^{(6)} \xi_{(10)}^9 - 2 x_1^{(3)} x_4^{(6)} \xi_{(10)}^{14} - x_1^{(3)} x_4^{(6)} \xi_{(10)}^{15}
    \\&\quad
    + x_1^{(3)} x_4^{(6)} \xi_{(10)}^{16} + x_1^{(3)} x_4^{(6)} \xi_{(10)}^{17} - x_1^{(4)} x_1^{(5)} \xi_{(10)}^1 - 4   x_1^{(4)} x_2^{(5)} \xi_{(10)}^4 + 2 x_1^{(4)} x_2^{(5)} \xi_{(10)}^6 - 10   x_1^{(4)} x_3^{(5)} \xi_{(10)}^8
    \\&\quad
    - x_1^{(4)} x_3^{(5)} \xi_{(10)}^9 - 2   x_1^{(4)} x_3^{(5)} \xi_{(10)}^{10} + x_1^{(4)} x_3^{(5)} \xi_{(10)}^{11} + 4   x_1^{(4)} x_3^{(5)} \xi_{(10)}^{12} - x_1^{(4)} x_3^{(5)} \xi_{(10)}^{13} - 2   x_2^{(1)} x_1^{(8)} \xi_{(10)}^4
    \\&\quad
    - 3 x_2^{(1)} x_2^{(8)} \xi_{(10)}^9 - x_2^{(1)} x_3^{(8)} \xi_{(10)}^3 + 5 x_2^{(1)} x_4^{(8)} \xi_{(10)}^4 - x_2^{(1)} x_4^{(8)} \xi_{(10)}^5 - x_2^{(1)} x_4^{(8)} \xi_{(10)}^6 - x_2^{(1)} x_4^{(8)} \xi_{(10)}^7
    \\&\quad
    - 2 x_2^{(1)} x_5^{(8)} \xi_{(10)}^7 + 5   x_2^{(1)} x_6^{(8)} \xi_{(10)}^8 + 5 x_2^{(1)} x_6^{(8)} \xi_{(10)}^9 + x_2^{(1)} x_6^{(8)} \xi_{(10)}^{10} - x_2^{(1)} x_6^{(8)} \xi_{(10)}^{11} - 2   x_2^{(1)} x_6^{(8)} \xi_{(10)}^{12}
    \\&\quad
    - x_2^{(1)} x_6^{(8)} \xi_{(10)}^{13} - 3   x_2^{(1)} x_7^{(8)} \xi_{(10)}^{13} + 5 x_2^{(1)} x_8^{(8)} \xi_{(10)}^{14} + x_2^{(1)} x_8^{(8)} \xi_{(10)}^{15} - 2 x_2^{(1)} x_8^{(8)} \xi_{(10)}^{16} - 3   x_2^{(1)} x_8^{(8)} \xi_{(10)}^{17}
    \\&\quad
    - x_2^{(4)} x_1^{(5)} \xi_{(10)}^2 + 4   x_2^{(4)} x_2^{(5)} \xi_{(10)}^8 - 5 x_2^{(4)} x_2^{(5)} \xi_{(10)}^9 + 2   x_2^{(4)} x_2^{(5)} \xi_{(10)}^{10} - x_2^{(4)} x_2^{(5)} \xi_{(10)}^{11} - x_2^{(4)} x_2^{(5)} \xi_{(10)}^{12}
    \\&\quad
    + x_2^{(4)} x_2^{(5)} \xi_{(10)}^{13} - 2   x_2^{(4)} x_3^{(5)} \xi_{(10)}^{14} + x_2^{(4)} x_3^{(5)} \xi_{(10)}^{17} - x_4 x_2^{(9)} \xi_{(10)}^1 - x_4 x_4^{(9)} \xi_{(10)}^2 + x_4 x_8^{(9)} \xi_{(10)}^3 - 5 x_4 x_9^{(9)} \xi_{(10)}^4
    \\&\quad
    + x_4 x_9^{(9)} \xi_{(10)}^5 + 5 x_4 x_9^{(9)} \xi_{(10)}^6 + x_4 x_9^{(9)} \xi_{(10)}^7 + 2 x_4 x_{10}^{(9)} \xi_{(10)}^7 - 5   x_4 x_{11}^{(9)} \xi_{(10)}^8 - 5 x_4 x_{11}^{(9)} \xi_{(10)}^9 - x_4 x_{11}^{(9)} \xi_{(10)}^{10}
    \\&\quad
    + x_4 x_{11}^{(9)} \xi_{(10)}^{11} + 5   x_4 x_{11}^{(9)} \xi_{(10)}^{12} + x_4 x_{11}^{(9)} \xi_{(10)}^{13} + 3   x_4 x_{12}^{(9)} \xi_{(10)}^{13} - 5 x_4 x_{13}^{(9)} \xi_{(10)}^{14} - x_4 x_{13}^{(9)} \xi_{(10)}^{15}
    \\&\quad
    + 2 x_4 x_{13}^{(9)} \xi_{(10)}^{16} + 5 x_4 x_{13}^{(9)} \xi_{(10)}^{17} + 2   x_2 x_3 x_1^{(9)} \xi_{(10)}^4 + 3   x_2 x_3 x_2^{(9)} \xi_{(10)}^8 + 3   x_2 x_3 x_3^{(9)} \xi_{(10)}^9 + x_2 x_3 x_4^{(9)} \xi_{(10)}^{14}
    \\&\quad
    + 2 x_2 x_3 x_5^{(9)} \xi_{(10)}^6 + 3   x_2 x_3 x_6^{(9)} \xi_{(10)}^{12} + x_2 x_3 x_7^{(9)} \xi_{(10)}^{17}\\
    &\quad
    +10 x_1 x_3^{(10)} \xi_{(11)}^3 + 3 x_1 x_4^{(10)} \xi_{(11)}^7 + 3   x_1 x_6^{(10)} \xi_{(11)}^6 + 3 x_1 x_7^{(10)} \xi_{(11)}^8 + x_1 x_8^{(10)} \xi_{(11)}^{11} + x_1 x_9^{(10)} \xi_{(11)}^{12} + x_1 x_{12}^{(10)} \xi_{(11)}^{10}
    \\&\quad
    + x_1 x_{13}^{(10)} \xi_{(11)}^{13} + x_1 x_{14}^{(10)} \xi_{(11)}^{16} + x_1 x_{17}^{(10)} \xi_{(11)}^{15} - x_1^{(1)} x_1^{(9)} \xi_{(11)}^1 - 6 x_1^{(1)} x_1^{(9)} \xi_{(11)}^7 - 3   x_1^{(1)} x_2^{(9)} \xi_{(11)}^{11}
    \\&\quad
    - x_1^{(1)} x_3^{(9)} \xi_{(11)}^2 - 3   x_1^{(1)} x_3^{(9)} \xi_{(11)}^{12} - x_1^{(1)} x_4^{(9)} \xi_{(11)}^{16} - 6   x_1^{(1)} x_5^{(9)} \xi_{(11)}^6 - 3 x_1^{(1)} x_6^{(9)} \xi_{(11)}^{10} - x_1^{(1)} x_7^{(9)} \xi_{(11)}^{15}
    \\&\quad
    - 2 x_1^{(1)} x_8^{(9)} \xi_{(11)}^{18} - 3   x_1^{(1)} x_9^{(9)} \xi_{(11)}^{20} - x_1^{(1)} x_{10}^{(9)} \xi_{(11)}^{21} - x_1^{(1)} x_{11}^{(9)} \xi_{(11)}^{23} - x_1^{(1)} x_{12}^{(9)} \xi_{(11)}^{24}
    \\&\quad
    - x_1^{(1)} x_{13}^{(9)} \xi_{(11)}^{25} + x_1^{(2)} x_1^{(8)} \xi_{(11)}^1 + 3   x_1^{(2)} x_1^{(8)} \xi_{(11)}^5 - 3 x_1^{(2)} x_1^{(8)} \xi_{(11)}^6 + 9   x_1^{(2)} x_1^{(8)} \xi_{(11)}^7 + 3 x_1^{(2)} x_1^{(8)} \xi_{(11)}^8
    \\&\quad
    - 3   x_1^{(2)} x_1^{(8)} \xi_{(11)}^9 + x_1^{(2)} x_2^{(8)} \xi_{(11)}^2 - x_1^{(2)} x_2^{(8)} \xi_{(11)}^{10} + x_1^{(2)} x_2^{(8)} \xi_{(11)}^{11} + 4   x_1^{(2)} x_2^{(8)} \xi_{(11)}^{12} + x_1^{(2)} x_2^{(8)} \xi_{(11)}^{13}
    \\&\quad
    - x_1^{(2)} x_2^{(8)} \xi_{(11)}^{14} + 10 x_1^{(2)} x_3^{(8)} \xi_{(11)}^3 - 4   x_1^{(2)} x_3^{(8)} \xi_{(11)}^4 + 3 x_1^{(2)} x_4^{(8)} \xi_{(11)}^6 - 15   x_1^{(2)} x_4^{(8)} \xi_{(11)}^7 + 3 x_1^{(2)} x_4^{(8)} \xi_{(11)}^8
    \\&\quad
    + 6   x_1^{(2)} x_5^{(8)} \xi_{(11)}^8 - 3 x_1^{(2)} x_5^{(8)} \xi_{(11)}^9 + 2   x_1^{(2)} x_6^{(8)} \xi_{(11)}^{10} - 5 x_1^{(2)} x_6^{(8)} \xi_{(11)}^{11} - 5   x_1^{(2)} x_6^{(8)} \xi_{(11)}^{12} + x_1^{(2)} x_6^{(8)} \xi_{(11)}^{13}
    \\&\quad
    + 3   x_1^{(2)} x_7^{(8)} \xi_{(11)}^{13} - 2 x_1^{(2)} x_7^{(8)} \xi_{(11)}^{14} + 3   x_1^{(2)} x_8^{(8)} \xi_{(11)}^{15} - 5 x_1^{(2)} x_8^{(8)} \xi_{(11)}^{16} - 2   x_1^{(3)} x_1^{(7)} \xi_{(11)}^1
    \\&\quad
    - 3 x_1^{(3)} x_1^{(7)} \xi_{(11)}^5 + 3   x_1^{(3)} x_1^{(7)} \xi_{(11)}^6 - 9 x_1^{(3)} x_1^{(7)} \xi_{(11)}^7 - 3   x_1^{(3)} x_1^{(7)} \xi_{(11)}^8 + 3 x_1^{(3)} x_1^{(7)} \xi_{(11)}^9 - 2   x_1^{(3)} x_2^{(7)} \xi_{(11)}^2
    \\&\quad
    + x_1^{(3)} x_2^{(7)} \xi_{(11)}^{10} - x_1^{(3)} x_2^{(7)} \xi_{(11)}^{11} - 4 x_1^{(3)} x_2^{(7)} \xi_{(11)}^{12} - x_1^{(3)} x_2^{(7)} \xi_{(11)}^{13} + x_1^{(3)} x_2^{(7)} \xi_{(11)}^{14} - x_1^{(3)} x_3^{(7)} \xi_{(11)}^{18}
    \\&\quad
    + x_1^{(3)} x_3^{(7)} \xi_{(11)}^{19} - x_1^{(3)} x_4^{(7)} \xi_{(11)}^{20} + x_1^{(3)} x_4^{(7)} \xi_{(11)}^{22} - x_1^{(3)} x_5^{(7)} \xi_{(11)}^{23} + x_1^{(3)} x_5^{(7)} \xi_{(11)}^{24} - 4   x_1^{(4)} x_1^{(6)} \xi_{(11)}^3
    \\&\quad
    + 2 x_1^{(4)} x_1^{(6)} \xi_{(11)}^4 + 6   x_1^{(4)} x_2^{(6)} \xi_{(11)}^1 + 18 x_1^{(4)} x_2^{(6)} \xi_{(11)}^5 - 18   x_1^{(4)} x_2^{(6)} \xi_{(11)}^6 + 36 x_1^{(4)} x_2^{(6)} \xi_{(11)}^7 + 12   x_1^{(4)} x_2^{(6)} \xi_{(11)}^8
    \\&\quad
    - 12 x_1^{(4)} x_2^{(6)} \xi_{(11)}^9 - 2   x_1^{(4)} x_3^{(6)} \xi_{(11)}^1 - 9 x_1^{(4)} x_3^{(6)} \xi_{(11)}^5 + 6   x_1^{(4)} x_3^{(6)} \xi_{(11)}^6 - 12 x_1^{(4)} x_3^{(6)} \xi_{(11)}^7 - 6   x_1^{(4)} x_3^{(6)} \xi_{(11)}^8
    \\&\quad
    + 6 x_1^{(4)} x_3^{(6)} \xi_{(11)}^9 + 6   x_1^{(4)} x_4^{(6)} \xi_{(11)}^2 - 6 x_1^{(4)} x_4^{(6)} \xi_{(11)}^{10} + 6   x_1^{(4)} x_4^{(6)} \xi_{(11)}^{11} + 15 x_1^{(4)} x_4^{(6)} \xi_{(11)}^{12} + 3   x_1^{(4)} x_4^{(6)} \xi_{(11)}^{13}
    \\&\quad
    - 2 x_1^{(4)} x_4^{(6)} \xi_{(11)}^{14} - 3   x_1^{(5)} x_2^{(5)} \xi_{(11)}^1 - 9 x_1^{(5)} x_2^{(5)} \xi_{(11)}^5 + 6   x_1^{(5)} x_2^{(5)} \xi_{(11)}^6 - 12 x_1^{(5)} x_2^{(5)} \xi_{(11)}^7 - 6   x_1^{(5)} x_2^{(5)} \xi_{(11)}^8
    \\&\quad
    + 6 x_1^{(5)} x_2^{(5)} \xi_{(11)}^9 - 3   x_1^{(5)} x_3^{(5)} \xi_{(11)}^2 + x_1^{(5)} x_3^{(5)} \xi_{(11)}^{10} + 2   x_1^{(5)} x_3^{(5)} \xi_{(11)}^{11} - 7 x_1^{(5)} x_3^{(5)} \xi_{(11)}^{12} - x_1^{(5)} x_3^{(5)} \xi_{(11)}^{13}
    \\&\quad
    + x_1^{(5)} x_3^{(5)} \xi_{(11)}^{14} + 2   x_2^{(1)} x_1^{(9)} \xi_{(11)}^3 - 2 x_2^{(1)} x_1^{(9)} \xi_{(11)}^4 + x_2^{(1)} x_2^{(9)} \xi_{(11)}^1 + 3 x_2^{(1)} x_2^{(9)} \xi_{(11)}^5 - 3   x_2^{(1)} x_2^{(9)} \xi_{(11)}^6
    \\&\quad
    + 3 x_2^{(1)} x_2^{(9)} \xi_{(11)}^7 + 3   x_2^{(1)} x_2^{(9)} \xi_{(11)}^8 - 3 x_2^{(1)} x_2^{(9)} \xi_{(11)}^9 - 3   x_2^{(1)} x_3^{(9)} \xi_{(11)}^5 + x_2^{(1)} x_4^{(9)} \xi_{(11)}^2 - x_2^{(1)} x_4^{(9)} \xi_{(11)}^{10}
    \\&\quad
    + x_2^{(1)} x_4^{(9)} \xi_{(11)}^{11} + x_2^{(1)} x_4^{(9)} \xi_{(11)}^{12} + x_2^{(1)} x_4^{(9)} \xi_{(11)}^{13} - x_2^{(1)} x_4^{(9)} \xi_{(11)}^{14} - 2 x_2^{(1)} x_5^{(9)} \xi_{(11)}^4 - 3   x_2^{(1)} x_6^{(9)} \xi_{(11)}^9
    \\&\quad
    - x_2^{(1)} x_7^{(9)} \xi_{(11)}^{14} - x_2^{(1)} x_8^{(9)} \xi_{(11)}^{17} - x_2^{(1)} x_9^{(9)} \xi_{(11)}^{18} - x_2^{(1)} x_9^{(9)} \xi_{(11)}^{19} - 2 x_2^{(1)} x_{10}^{(9)} \xi_{(11)}^{19} + x_2^{(1)} x_{11}^{(9)} \xi_{(11)}^{20}
    \\&\quad
    - x_2^{(1)} x_{11}^{(9)} \xi_{(11)}^{21} - x_2^{(1)} x_{11}^{(9)} \xi_{(11)}^{22} - 3 x_2^{(1)} x_{12}^{(9)} \xi_{(11)}^{22} + x_2^{(1)} x_{13}^{(9)} \xi_{(11)}^{23} - 2 x_2^{(1)} x_{13}^{(9)} \xi_{(11)}^{24}
    \\&\quad
    - 6 x_2^{(4)} x_1^{(6)} \xi_{(11)}^1 - 9 x_2^{(4)} x_1^{(6)} \xi_{(11)}^5 + 18   x_2^{(4)} x_1^{(6)} \xi_{(11)}^6 - 30 x_2^{(4)} x_1^{(6)} \xi_{(11)}^7 - 12   x_2^{(4)} x_1^{(6)} \xi_{(11)}^8 + 9 x_2^{(4)} x_1^{(6)} \xi_{(11)}^9\\
   &\quad
    + x_2^{(4)} x_2^{(6)} \xi_{(11)}^{10} - 4 x_2^{(4)} x_2^{(6)} \xi_{(11)}^{11} + 5   x_2^{(4)} x_2^{(6)} \xi_{(11)}^{12} - x_2^{(4)} x_2^{(6)} \xi_{(11)}^{13} - 2   x_2^{(4)} x_3^{(6)} \xi_{(11)}^2 + x_2^{(4)} x_3^{(6)} \xi_{(11)}^{10}
    \\&\quad
    + 2   x_2^{(4)} x_3^{(6)} \xi_{(11)}^{11} - 7 x_2^{(4)} x_3^{(6)} \xi_{(11)}^{12} - x_2^{(4)} x_3^{(6)} \xi_{(11)}^{13} + x_2^{(4)} x_3^{(6)} \xi_{(11)}^{14} - x_2^{(4)} x_4^{(6)} \xi_{(11)}^{15} + 2 x_2^{(4)} x_4^{(6)} \xi_{(11)}^{16}
    \\&\quad
    + 2   x_2^{(5)} x_3^{(5)} \xi_{(11)}^{20} - x_2^{(5)} x_3^{(5)} \xi_{(11)}^{21} + x_2^{(5)} x_3^{(5)} \xi_{(11)}^{22} - x_4 x_4^{(10)} \xi_{(11)}^3 + x_4 x_4^{(10)} \xi_{(11)}^4 - 5 x_4 x_5^{(10)} \xi_{(11)}^3 + x_4 x_6^{(10)} \xi_{(11)}^4
    \\&\quad
    - x_4 x_8^{(10)} \xi_{(11)}^5 + x_4 x_8^{(10)} \xi_{(11)}^6 - x_4 x_8^{(10)} \xi_{(11)}^7 - x_4 x_8^{(10)} \xi_{(11)}^8 + x_4 x_8^{(10)} \xi_{(11)}^9 + x_4 x_9^{(10)} \xi_{(11)}^5 - 5 x_4 x_{10}^{(10)} \xi_{(11)}^6
    \\&\quad
    + 5   x_4 x_{10}^{(10)} \xi_{(11)}^7 - x_4 x_{10}^{(10)} \xi_{(11)}^8 - 6   x_4 x_{11}^{(10)} \xi_{(11)}^8 + x_4 x_{12}^{(10)} \xi_{(11)}^9 + x_4 x_{14}^{(10)} \xi_{(11)}^{10} - x_4 x_{14}^{(10)} \xi_{(11)}^{11}
    \\&\quad
    - x_4 x_{14}^{(10)} \xi_{(11)}^{12} - x_4 x_{14}^{(10)} \xi_{(11)}^{13} + x_4 x_{14}^{(10)} \xi_{(11)}^{14} - 5 x_4 x_{15}^{(10)} \xi_{(11)}^{10} + 5 x_4 x_{15}^{(10)} \xi_{(11)}^{11} + 5 x_4 x_{15}^{(10)} \xi_{(11)}^{12}
    \\&\quad
    - x_4 x_{15}^{(10)} \xi_{(11)}^{13} - 3 x_4 x_{16}^{(10)} \xi_{(11)}^{13} + x_4 x_{17}^{(10)} \xi_{(11)}^{14} - 5 x_4 x_{18}^{(10)} \xi_{(11)}^{15} + 5 x_4 x_{18}^{(10)} \xi_{(11)}^{16} + x_2 x_3 x_1^{(10)} \xi_{(11)}^1
    \\&\quad
    + x_2 x_3 x_2^{(10)} \xi_{(11)}^2 + x_2 x_3 x_3^{(10)} \xi_{(11)}^{17} + x_2 x_3 x_5^{(10)} \xi_{(11)}^{18} + x_2 x_3 x_7^{(10)} \xi_{(11)}^{19} + x_2 x_3 x_{10}^{(10)} \xi_{(11)}^{20} + x_2 x_3 x_{11}^{(10)} \xi_{(11)}^{21}
    \\&\quad
    + x_2 x_3 x_{13}^{(10)} \xi_{(11)}^{22} + x_2 x_3 x_{15}^{(10)} \xi_{(11)}^{23} + x_2 x_3 x_{16}^{(10)} \xi_{(11)}^{24} + x_2 x_3 x_{18}^{(10)} \xi_{(11)}^{25}.
\end{align*}
We were able to determine $\pi_m$ for $1\le m\le 12$:
\begin{align*}
    \pi_1        &=  x_1 x_2 \xi^1 \xi^2 - x_1 x_3 \xi^1 \xi^3 + x_2 x_3 \xi^2 \xi^3 - x_2 x_4 \xi^2 \xi^4 + x_3 x_4 \xi^3 \xi^4,
     \quad  \pi_2        =  0,
     \quad
    \pi_3        =  x_3 x_1^{(2)} \xi^3 \xi_{(2)}^1 - x_2 x_1^{(2)} \xi^2 \xi_{(2)}^1,
     \\
    \pi_4        &=  0,\quad
    \pi_5        =   - x_2 x_1^{(4)} \xi^2 \xi_{(4)}^1 - x_2 x_2^{(4)} \xi^2 \xi_{(4)}^2 + x_3 x_1^{(4)} \xi^3 \xi_{(4)}^1 + x_3 x_2^{(4)} \xi^3 \xi_{(4)}^2, \quad
    \pi_6        =  2 x_2 x_1^{(5)} \xi^2 \xi_{(5)}^1 - 2 x_3 x_1^{(5)} \xi^3 \xi_{(5)}^1,
     \\
    \pi_7        &=   - x_2 x_1^{(6)} \xi^2 \xi_{(6)}^1 - x_2 x_2^{(6)} \xi^2 \xi_{(6)}^2 - x_2 x_3^{(6)} \xi^2 \xi_{(6)}^3 - x_2 x_4^{(6)} \xi^2 \xi_{(6)}^4 + x_3 x_1^{(6)} \xi^3 \xi_{(6)}^1 + x_3 x_2^{(6)} \xi^3 \xi_{(6)}^2
    \\&\quad
    + x_3 x_3^{(6)} \xi^3 \xi_{(6)}^3 + x_3 x_4^{(6)} \xi^3 \xi_{(6)}^4,
     \\
    \pi_8        &=  2 x_2 x_1^{(7)} \xi^2 \xi_{(7)}^1 + 2 x_2 x_2^{(7)} \xi^2 \xi_{(7)}^2 - 2 x_3 x_1^{(7)} \xi^3 \xi_{(7)}^1 - 2   x_3 x_2^{(7)} \xi^3 \xi_{(7)}^2,
     \\
    \pi_9        &=   - x_2 x_1^{(8)} \xi^2 \xi_{(8)}^1 - x_2 x_2^{(8)} \xi^2 \xi_{(8)}^2 - x_2 x_3^{(8)} \xi^2 \xi_{(8)}^3 - x_2 x_4^{(8)} \xi^2 \xi_{(8)}^4 - x_2 x_5^{(8)} \xi^2 \xi_{(8)}^5 - x_2 x_6^{(8)} \xi^2 \xi_{(8)}^6
    \\&\quad
    - x_2 x_7^{(8)} \xi^2 \xi_{(8)}^7 - x_2 x_8^{(8)} \xi^2 \xi_{(8)}^8 + x_3 x_1^{(8)} \xi^3 \xi_{(8)}^1 + x_3 x_2^{(8)} \xi^3 \xi_{(8)}^2 + x_3 x_3^{(8)} \xi^3 \xi_{(8)}^3 + x_3 x_4^{(8)} \xi^3 \xi_{(8)}^4
    \\&\quad
    + x_3 x_5^{(8)} \xi^3 \xi_{(8)}^5 + x_3 x_6^{(8)} \xi^3 \xi_{(8)}^6 + x_3 x_7^{(8)} \xi^3 \xi_{(8)}^7 + x_3 x_8^{(8)} \xi^3 \xi_{(8)}^8,
     \\
    \pi_{10}        &=  2 x_2 x_1^{(9)} \xi^2 \xi_{(9)}^1 + 2 x_2 x_2^{(9)} \xi^2 \xi_{(9)}^2 + 2 x_2 x_3^{(9)} \xi^2 \xi_{(9)}^3 + 2   x_2 x_4^{(9)} \xi^2 \xi_{(9)}^4 + 2 x_2 x_5^{(9)} \xi^2 \xi_{(9)}^5
    \\&\quad
    + 2 x_2 x_6^{(9)} \xi^2 \xi_{(9)}^6 + 2   x_2 x_7^{(9)} \xi^2 \xi_{(9)}^7 - 2 x_3 x_1^{(9)} \xi^3 \xi_{(9)}^1 - 2 x_3 x_2^{(9)} \xi^3 \xi_{(9)}^2 - 2   x_3 x_3^{(9)} \xi^3 \xi_{(9)}^3 - 2 x_3 x_4^{(9)} \xi^3 \xi_{(9)}^4
    \\&\quad
    - 2 x_3 x_5^{(9)} \xi^3 \xi_{(9)}^5 - 2   x_3 x_6^{(9)} \xi^3 \xi_{(9)}^6 - 2 x_3 x_7^{(9)} \xi^3 \xi_{(9)}^7,
     \\
    \pi_{11}        &=   - 3 x_2 x_1^{(10)} \xi^2 \xi_{(10)}^1 - 3 x_2 x_2^{(10)} \xi^2 \xi_{(10)}^2 - x_2 x_3^{(10)} \xi^2 \xi_{(10)}^3 - x_2 x_4^{(10)} \xi^2 \xi_{(10)}^4 - x_2 x_5^{(10)} \xi^2 \xi_{(10)}^5
    \\&\quad
    - x_2 x_6^{(10)} \xi^2 \xi_{(10)}^6 - x_2 x_7^{(10)} \xi^2 \xi_{(10)}^7 - x_2 x_8^{(10)} \xi^2 \xi_{(10)}^8 - x_2 x_9^{(10)} \xi^2 \xi_{(10)}^9 - x_2 x_{10}^{(10)} \xi^2 \xi_{(10)}^{10} - x_2 x_{11}^{(10)} \xi^2 \xi_{(10)}^{11}
    \\&\quad
    - x_2 x_{12}^{(10)} \xi^2 \xi_{(10)}^{12} - x_2 x_{13}^{(10)} \xi^2 \xi_{(10)}^{13} - x_2 x_{14}^{(10)} \xi^2 \xi_{(10)}^{14} - x_2 x_{15}^{(10)} \xi^2 \xi_{(10)}^{15} - x_2 x_{16}^{(10)} \xi^2 \xi_{(10)}^{16}
    \\&\quad
    - x_2 x_{17}^{(10)} \xi^2 \xi_{(10)}^{17} - x_2 x_{18}^{(10)} \xi^2 \xi_{(10)}^{18} + 3   x_3 x_1^{(10)} \xi^3 \xi_{(10)}^1 + 3 x_3 x_2^{(10)} \xi^3 \xi_{(10)}^2 + x_3 x_3^{(10)} \xi^3 \xi_{(10)}^3
    \\&\quad
    + x_3 x_4^{(10)} \xi^3 \xi_{(10)}^4 + x_3 x_5^{(10)} \xi^3 \xi_{(10)}^5 + x_3 x_6^{(10)} \xi^3 \xi_{(10)}^6 + x_3 x_7^{(10)} \xi^3 \xi_{(10)}^7 + x_3 x_8^{(10)} \xi^3 \xi_{(10)}^8 + x_3 x_9^{(10)} \xi^3 \xi_{(10)}^9
    \\&\quad
    + x_3 x_{10}^{(10)} \xi^3 \xi_{(10)}^{10} + x_3 x_{11}^{(10)} \xi^3 \xi_{(10)}^{11} + x_3 x_{12}^{(10)} \xi^3 \xi_{(10)}^{12} + x_3 x_{13}^{(10)} \xi^3 \xi_{(10)}^{13} + x_3 x_{14}^{(10)} \xi^3 \xi_{(10)}^{14}
    \\&\quad
    + x_3 x_{15}^{(10)} \xi^3 \xi_{(10)}^{15} + x_3 x_{16}^{(10)} \xi^3 \xi_{(10)}^{16} + x_3 x_{17}^{(10)} \xi^3 \xi_{(10)}^{17} + x_3 x_{18}^{(10)} \xi^3 \xi_{(10)}^{18},
     \\
    \pi_{12}        &=  2 x_2 x_1^{(11)} \xi^2 \xi_{(11)}^1 + 2 x_2 x_2^{(11)} \xi^2 \xi_{(11)}^2 + 2 x_2 x_3^{(11)} \xi^2 \xi_{(11)}^3 + 2   x_2 x_4^{(11)} \xi^2 \xi_{(11)}^4 + 2 x_2 x_5^{(11)} \xi^2 \xi_{(11)}^5
    \\&\quad
    + 2 x_2 x_6^{(11)} \xi^2 \xi_{(11)}^6 + 2   x_2 x_7^{(11)} \xi^2 \xi_{(11)}^7 + 2 x_2 x_8^{(11)} \xi^2 \xi_{(11)}^8 + 2 x_2 x_9^{(11)} \xi^2 \xi_{(11)}^9 + 2   x_2 x_{10}^{(11)} \xi^2 \xi_{(11)}^{10}
    \\&\quad
    + 2 x_2 x_{11}^{(11)} \xi^2 \xi_{(11)}^{11} + 2 x_2 x_{12}^{(11)} \xi^2 \xi_{(11)}^{12} + 2   x_2 x_{13}^{(11)} \xi^2 \xi_{(11)}^{13} + 2 x_2 x_{14}^{(11)} \xi^2 \xi_{(11)}^{14} + 2 x_2 x_{15}^{(11)} \xi^2 \xi_{(11)}^{15}
    \\&\quad
    + 2   x_2 x_{16}^{(11)} \xi^2 \xi_{(11)}^{16} - 2 x_3 x_1^{(11)} \xi^3 \xi_{(11)}^1 - 2 x_3 x_2^{(11)} \xi^3 \xi_{(11)}^2 - 2   x_3 x_3^{(11)} \xi^3 \xi_{(11)}^3 - 2 x_3 x_4^{(11)} \xi^3 \xi_{(11)}^4
    \\&\quad
    - 2 x_3 x_5^{(11)} \xi^3 \xi_{(11)}^5 - 2   x_3 x_6^{(11)} \xi^3 \xi_{(11)}^6 - 2 x_3 x_7^{(11)} \xi^3 \xi_{(11)}^7 - 2 x_3 x_8^{(11)} \xi^3 \xi_{(11)}^8 - 2   x_3 x_9^{(11)} \xi^3 \xi_{(11)}^9 - 2 x_3 x_{10}^{(11)} \xi^3 \xi_{(11)}^{10}
    \\&\quad
    - 2 x_3 x_{11}^{(11)} \xi^3 \xi_{(11)}^{11} - 2   x_3 x_{12}^{(11)} \xi^3 \xi_{(11)}^{12} - 2 x_3 x_{13}^{(11)} \xi^3 \xi_{(11)}^{13} - 2 x_3 x_{14}^{(11)} \xi^3 \xi_{(11)}^{14} - 2   x_3 x_{15}^{(11)} \xi^3 \xi_{(11)}^{15}
    \\&\quad
    - 2 x_3 x_{16}^{(11)} \xi^3 \xi_{(11)}^{16}.
\end{align*}
\endgroup
In this example the $\pi_m$'s appear to have a diagonal shape and are remarkably simple in comparison to the coefficients of $\pi_0$. We observe that $\llbracket \pi_i,\pi_j\rrbracket=0$ for $1\le i,j \le 12$. Even though $S/I$ is not Koszul the $\pi_m$ are homogeneous of $x$-degree $2$ for $1\le m\le 12$.  Note the $\pi_0^{\le 11}$ is not of homogeneous $x$-degree.

\subsection{A summary of the data for more contrived examples}\label{subsec:contrived}

For more interesting examples the calculations quickly become overwhelming, and for lack of space we cannot present the details of our findings here; they are available by request from the authors. Let us just record how far could we get with the examples of Section \ref{sec:affine}.

\begin{align*}
\begin{tabular}{c||c|c|c}
 Example   & $\dim S$ & $|\mathcal I_m|$, $m=1,2,\dots$ & higher brack. computed\\\hline\hline
 $(-1,1,1)$-circle quotient (see Subsec. \ref{subsubsec:-1,1,1}) & $8$& $9,16,45,\dots$ & $\pi_1$, $\pi_2$, $\pi_3$, $\pi_4$\\\hline
  $2$ particles with zero ang. momentum (see Subsec. \ref{subsubsec:angmom}) & $10$ & $11,10,10,\dots$ &  $\pi_1$, $\pi_2$, $\pi_3^{02}$, $\pi_3^{011}$\\\hline
   deg. $2$ harm. polynomials of $3$ var. (see Subsec. \ref{subsubsec:harmonic}) & $3$ & $5,5,10,24,55,\dots$ &  $\pi_1$, $\pi_2$, $\pi_3$, $\pi_4$, $\pi_5$\\\hline
 $2\times 2$-minors of a $3\times 3$-matrix (see Subsec. \ref{subsubsec:detideal})   & $9$& $9,16,45,\dots$   & $\pi_1$, $\pi_2$, $\pi_3$, $\pi_4$
\end{tabular}
\end{align*}
The example of two particles in dimension three (see Subsec. \ref{subsubsec:angmom}) is not Koszul and $\pi_2$ and $\pi_3$ are not of homogeneous $x$-degree.

\appendix
\section{Poisson cohomology}\label{ap:Poissoncohomology}
Let $S=\bs k[x_1,x_2,\dots, x_n]$ be a polynomial algebra and $\{\:,\:\}$ be a Poisson bracket on $S$. The $S$-module of K\"ahler differentials $\Omega_{S|\bs k}$ is a free module dual to the $S$-module of derivations $D_S:=\{X:S\to S\mid X(fg)=f\;X(g)+g\; X(f)\}$. This extends to an isomorphism $\wedge^m D_S\cong\operatorname{Alt}^m_S(\Omega_{S|\bs k},S)=:\mathrm C^m_{\operatorname{Poiss}}(S)$ between the $m$th exterior power $D_S$ (understood to be $S$-linear) and the alternating $m$-multilinear maps from $\Omega_{S|\bs k}$ to $S$. The anchor map $\rho:\Omega_{S|\bs k}\to D_S$ sends $\mathrm df$ to the derivation $g\mapsto \{f,g\}$. The codifferential of Poisson cohomology $\dP$ is defined by the following formula for $X\in\mathrm C^m_{\operatorname{Poiss}}(S)$:
\begin{align*}
&(\dP X)(\alpha_0,\alpha_1,\dots,\alpha_m)=\sum_{j=0}^m (-1)^j \rho(\alpha_j)\left(X(\alpha_0,\dots,\widehat{\alpha_j},\dots,\alpha_m)\right)\\
&\qquad\qquad\qquad\qquad\qquad\qquad +\sum_{0\le k<\ell\le m}(-1)^{k+\ell} X([\alpha_k,\alpha_\ell],\alpha_0,\dots,\widehat{\alpha_k},\dots,\widehat{\alpha_\ell},\dots,\alpha_m),
\end{align*}
where $\alpha_0,\alpha_1,\dots,\alpha_m\in \Omega_{S|\bs k}$. In the above the hatted terms are understood to be omitted and the $[\:,\:]$ is the Koszul bracket, cf. \eqref{eq:Koszulbr}.
In fact $(\mathrm C^\bullet_{\operatorname{Poiss}}(S),\dP)$ forms a supercommutative dg algebra. Note that $\dP$ is not $S$-linear. If $S=\bs k[x_1,x_2,\dots, x_n]$ carries a (possibly nonstandard) $\Z_{\ge 0}$-grading $\deg$ preserved by the bracket it follows that $\deg(\dP)=\deg(\{\:,\:\})$.

\bibliographystyle{amsplain}
\bibliography{higherKoszul}

\end{document}